\documentclass[a4paper,11pt]{article}
\usepackage[utf8]{inputenc}

\usepackage[english]{babel}
\usepackage{csquotes}

\usepackage[usenames,dvipsnames,svgnames,table]{xcolor}
\usepackage{graphicx}
\graphicspath{{figures/}}
\usepackage{amsmath, amsfonts, amssymb, amsthm}
\usepackage{mathrsfs}
\usepackage{stmaryrd}
\usepackage{authblk} 

\makeatletter
\renewcommand*{\@textcolor}[3]{
  \protect\leavevmode
  \begingroup
    \color#1{#2}#3
  \endgroup
}
\makeatother

\usepackage[bookmarks=true]{hyperref}
\definecolor{niceblue}{rgb}{0, 0.125, 0.666}
\definecolor{nicebrown}{rgb}{0.4, 0.2, 0}
\hypersetup{colorlinks=true, linkcolor=niceblue, citecolor=nicebrown}

\newtheorem{thm}{Theorem}[section]

\newtheorem{cor}[thm]{Corollary}
\newtheorem{lem}[thm]{Lemma}
\newtheorem{prop}[thm]{Proposition}

\newtheorem*{thm*}{Theorem}
\newtheorem*{cor*}{Corollary}
\newtheorem*{lem*}{Lemma}
\newtheorem*{prop*}{Proposition}

\theoremstyle{definition}
\newtheorem{dfn}[thm]{Definition}

\newtheorem*{claim}{Claim}

\newtheorem*{def*}{Definition}
\newtheorem*{def/thm*}{Definition{/}Theorem}

\newcommand{\N}{\mathbb{N}}
\newcommand{\Z}{\mathbb{Z}}

\newcommand{\R}{\mathbb{R}}

\newcommand{\Tree}{\mathcal T} 
\newcommand{\zero}{\boldsymbol{0}}
\newcommand{\one}{\boldsymbol{1}}

\newcommand{\into}{\hookrightarrow}

\newcommand{\SPH}{\mathcal{S}}

\newcommand{\vsum}[2]{\, _{#1}\!{\bullet}_{#2}\,}

\DeclareMathOperator{\intr}{int} 
\DeclareMathOperator{\Db}{D} 

\newcommand{\real}{\mathfrak{Re}}

\renewcommand{\tilde}{\widetilde}

\numberwithin{figure}{section}

\usepackage[style=alphabetic,sorting=nyt,maxbibnames=7,maxcitenames=5,block=space,backend=biber]{biblatex}
\date{\today}

\title{Canonical decompositions\\ and algorithmic recognition of spatial graphs}

\author[1]{Stefan Friedl}
\author[1]{Lars Munser}
\author[1]{José Pedro Quintanilha}
\author[2]{Yuri Santos Rego}
\affil[1]{\small Universität Regensburg\normalsize}
\affil[2]{\small Otto-von-Guericke-Universit\"at Magdeburg\normalsize}

\begin{document}

  \maketitle
  
    \begin{abstract}
	We prove that there exists an algorithm for determining whether two piecewise-linear spatial graphs are isomorphic. In its most general form, our theorem applies to spatial graphs furnished with vertex colorings, edge colorings and/or edge orientations.
	
	We first show that spatial graphs admit canonical decompositions into blocks, that is, spatial graphs that are non-separable  and have no cut vertices, in a suitable topological sense. Then we apply a result of Haken and Matveev in order to algorithmically distinguish these blocks.
    \end{abstract}
   
   \newpage
      \tableofcontents 
      
     {\small
     \subsection*{Acknowledgments} The first, second and third authors were supported by the CRC~1085 \emph{Higher Invariants} (Universit\"at Regensburg, funded by the DFG). We are grateful to Benjamin Ruppik and Claudius Zibrowius for helpful conversations, and to Sofia Amontova for bringing our attention to some relevant literature.}
     \newpage 
  
\section{Introduction}

\subsection{Our main result}

This article is concerned with spatial graphs, which are finite graphs embedded in oriented $3$-spheres -- not merely as subspaces, but with an explicit decomposition into vertices and edges. We give a precise definition using piecewise-linear (henceforth abbreviated as ``PL'') topology (Definition~\ref{dfn.spatialgraph}). The PL setting makes it easy to formalize the intuitive idea of what an isomorphism of spatial graphs should be: an orientation-preserving PL homeomorphism of the ambient $3$-spheres mapping vertices to vertices and edges to edges bijectively (Definition~\ref{dfn.spatialgraphiso}). The reader might be amused to try and decide which among the two pairs of spatial graphs depicted in Figure~\ref{fig.funexample} consist of isomorphic spatial graphs.

  \begin{figure}[h]
      \centering
      \def \svgwidth{0.9\linewidth}
      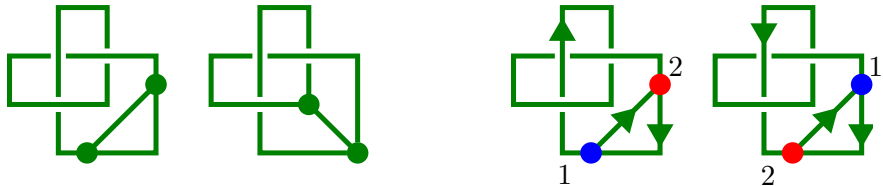
      \caption{Two pairs of (decorated) spatial graphs with the same combinatorial structure, but \textit{a priori} different topology.}
      \label{fig.funexample}
  \end{figure}

Our main result is the following:

\begin{thm}[Algorithmic detection of spatial graphs]\label{thm.algdetection}
There exists an algorithm for determining whether two spatial graphs are isomorphic.
\end{thm}

Theorem~\ref{thm.algdetection} is restated and proved in a more general context as Theorem~\ref{thm.spatialgraphrecognition}, where we allow the spatial graphs to come equipped with additional decorations, such as colorings of vertices and/or edges, and orientations of the edges (as in the right hand side of Figure~\ref{fig.funexample}). All this additional data must of course be respected by isomorphisms.

We reduce the task of producing a computer program for testing the isomorphism type of spatial graphs, to the implementation of algorithmic results in Matveev's text on computational $3$-manifold topology \cite{Matveev}. The input for such a program would then be a pair of oriented $3$-spheres given as finite simplicial complexes, with the vertices and edges of the spatial graphs specified as subcomplexes, and also possibly the data of decorations. We make no claims about the efficiency of such an algorithm.

\subsection{Matveev's Recognition Theorem}

Our proof uses as a main ingredient a theorem of Matveev that extends work of Haken \cite[Thm.~6.1.6]{Matveev}, concerning PL $3$-manifolds equipped with a $1$-dimensional subcomplex of their boundary (called a ``boundary pattern'', see Definition~\ref{dfn.mfdwithbpattern}). Matveev's Theorem (reproduced below as Theorem~\ref{thm.matveev}) states that it is possible to algorithmically detect whether two such $3$-manifolds with boundary pattern are PL-homeomorphic (via a homeomorphism respecting the boundary patterns), provided that they are ``Haken'' (Definition~\ref{dfn.Haken}).

We construct a PL $3$-manifold with boundary pattern out of a spatial graph, its ``marked exterior'' (Definition~\ref{dfn.decext}), such that two spatial graphs are isomorphic precisely if their marked exteriors are PL-homeomorphic. The marked exterior is built in a two-step process by first removing a suitably chosen open neighborhood of the vertices, and then one of (what is left of) the edges. The boundary is then marked with a pattern that allows for easily reconstructing the spatial graph. Figures \ref{fig.orientedexterior}~and~\ref{fig.markedexterior} illustrate the general idea of the construction. We then apply Matveev's Theorem to the marked exteriors.

The main difficulty in using Matveev's Recognition Theorem to deduce Theorem~\ref{thm.algdetection} is in the fact that Matveev's Theorem applies only to $3$-manifolds with boundary pattern that are Haken. This property encompasses three conditions: one about existence of non-trivial embedded surfaces (being ``sufficiently large''), one about triviality of embedded $2$-spheres (``irreducibility''), and one about triviality of properly embedded discs (``boundary-irreducibility''). Whilst the ``sufficiently large'' requirement is easy to guarantee, we will see that the marked exterior of a spatial graph may very well fail to be irreducible and boundary-irreducible. This issue will take considerable effort to resolve. 

\subsection{Decomposition results}

For an outline of the strategy, consider first the irreducibility requirement. The starting point is the observation (Proposition~\ref{prop.sep<=>red}) that for a spatial graph~$\Gamma$, irreducibility of its marked exterior is equivalent to~$\Gamma$ not being the ``disjoint union'' $\Gamma_1 \sqcup \Gamma_2$ of non-empty spatial graphs $\Gamma_1, \Gamma_2$, where this disjoint union is the operation of placing $\Gamma_1, \Gamma_2$ ``next to one another'' in the same ambient $3$-sphere (see Definition~\ref{dfn.disjointunion} for the precise notion). A non-empty spatial graph that is not a non-trivial disjoint union is called a ``piece'' (Definition~\ref{dfn.septhings}). In Proposition~\ref{prop.uniquenessdisjunion}, we show that the decomposition of a spatial graph into pieces is canonical in a suitable sense. This reduces the task of determining whether two spatial graphs are isomorphic, to testing whether the pieces in their decompositions are pairwise isomorphic. As these pieces have irreducible marked exteriors, we are one step closer to being able to apply Matveev's theorem.

The next step is to find a decomposition of non-separable graphs into spatial graphs whose marked exteriors are moreover boundary-irreducible. The strategy is similar to the one in the previous paragraph, except that the role of the disjoint union operation is played by the operation of ``vertex sum'' (Definition~\ref{dfn.vsum}). Roughly, the vertex sum of two spatial graphs, each with a distinguished vertex, is obtained by ``gluing them'' along those vertices. For non-separable spatial graphs~$\Gamma$, there is a very close correspondence between $\Gamma$~having boundary-irreducible marked exterior, and $\Gamma$~being indecomposable as a non-trivial vertex sum (Propositions \ref{prop.cut=>bred}~and~\ref{prop.bred=>cut}). We will see that non-separable spatial graphs admit a canonical decomposition as an iterated vertex sum (Propositions \ref{prop.treeexists}~and~\ref{prop.uniquetree}) of non-separable spatial graphs without cut vertices (which we call ``blocks'', see Definition~\ref{dfn.cutthings}). This will reduce the comparison of the isomorphism type of two non-separable spatial graphs, to comparing the blocks in their decomposition. Except for one easy special case, these blocks have marked exteriors amenable to Matveev's algorithm.

We remark that the iterated vertex sums from the previous paragraph are allowed to be performed along different vertices, so the canonical decomposition must come bundled with the combinatorial data of which vertices from different blocks are glued to which. To package this information, we introduce the notion of a ``tree of spatial graphs'' (Definition~\ref{dfn.tree}), and in case the spatial graphs being glued are blocks, we call it a ``tree of blocks'' (Definition~\ref{dfn.cutthings}). Our main results on decompositions as iterated vertex are summarized in the following theorem (see Propositions \ref{prop.treeexists}~and~\ref{prop.uniquetree} for precise statements):

\begin{thm}[Canonical decomposition as a tree of blocks]\label{thm.treeofblocks}
Every non-separable spatial graph other than a one-point graph admits a unique decomposition as a tree of blocks.
\end{thm}

This theory of decompositions has analogues in the setting of abstract graphs \cite[Excercise~8.3.3]{Jun05}. In the topological setting, Suzuki \cite{Su70} has established a unique factorization result  with respect to a ``composition'' operation similar in spirit to our vertex sum, but only for connected $1$-subcomplexes of the $3$-sphere, and only up to a ``neighborhood congruence'' relation. 
Our Theorem~\ref{thm.treeofblocks} differs from Suzuki's in the following: our spatial graphs come with vertex/edge decompositions and possibly decorations, we broaden the connectedness assumption to non-separability, and we have no identification ``up to neighborhood congruence'', instead keeping track of the vertices along which to glue.

We also deduce a triviality result for spatial trees, that is, spatial graphs~$\Gamma$ for which the underlying abstract graph~$\langle \Gamma \rangle$ is a tree:

\begin{thm}[Uniqueness of spatial trees]
If $\Gamma_1, \Gamma_2$~are spatial trees, then every isomorphism $\langle \Gamma_1 \rangle \to \langle \Gamma_2\rangle$ of their underlying abstract graphs is induced by an isomorphism of spatial graphs $\Gamma_1 \to \Gamma_2$.
\end{thm}

This result is restated in the text as Theorem~\ref{thm.trivialtrees}. It reduces the isomorphism test for spatial trees to an isomorphism test for abstract trees, bypassing the machinery of Matveev.

\subsection{The piecewise-linear setting}\label{sec.PL}

We have decided to work in the PL setting because it is the natural home for results in computational topology such as our main theorem, and it is the framework of the machinery developed by Matveev on which we rely. This is a standard framework in the field \cite{Su70,Kau89,Yam89,Bar15}, although a theory of smooth, rather than PL, spatial graphs has also been introduced by Herrmann and the first author \cite{StefanGerrit1, StefanGerrit3}.

The reader is referred to the textbook of Rourke and Sanderson \cite{RourkeSanderson} for the standard notions in PL topology. We will often give precise references for the results we import, but knowledge of basic concepts such as that of a polyhedron, or a PL manifold (possibly oriented, or with boundary) will be assumed. In particular, PL spaces (also called polyhedra) are by definition subspaces of some~$\R^n$ whose points admit a star neighborhood. The ambient space $\R^n$~is equipped with the metric induced from the $\ell^1$-norm, so by ``balls'' and ``spheres'' we always mean polyhedra that are PL-homeomorphic to cubes and their boundaries, respectively.

We will also make heavy usage of regular neighborhoods. If $X \subseteq P$ are polyhedra, with $X$~compact, then one may think of a regular neighborhood of~$X$ in~$P$ as a ``small, well-behaved neighborhood'' of~$X$ that deformation-retracts onto~$X$ \cite[Chapter~3]{RourkeSanderson}. If $P_0$~is a closed sub-polyhedron of~$P$, there is also the notion of a regular neighborhood~$N$ of~$X$ in the pair~$(P, P_0)$ \cite[p.~52]{RourkeSanderson}. In this case we use a lighter notation than the one in Rourke-Sanderson, who would instead have written that \emph{the pair}~$(N, N \cap P_0)$ is a regular neighborhood of \emph{the pair}~$(X, X \cap P_0)$ in~$(P, P_0)$.

The reader might be more familiar with the combinatorial definition of a PL space as (the topological realization of) an abstract simplicial complex. The two theories are equivalent: every polyhedron is a union of geometric simplices intersecting along faces \cite[Theorem~2.11]{RourkeSanderson}, and such a decomposition can be abstracted to a combinatorial setup \cite[Excercise~2.27(1)]{RourkeSanderson}. Moreover, PL maps between polyhedra can be expressed as simplicial maps between subdivisions of the abstract simplicial complexes that they realize \cite[Theorem~2.14]{RourkeSanderson}. (And although the notion of a simplicial subdivision is not combinatorial, the Alexander-Newman Theorem \cite[Theorem~4.5]{Li99} allows one to phrase 
purely combinatorially the property of two abstract simplicial complexes having PL-homeomorphic realizations.)

\subsection{Structure of the paper}

After laying out the most elementary terms in the theory of spatial graphs (Section~\ref{sec.preliminaries}), we describe the operation of disjoint union of spatial graphs (Section~\ref{sec.disjunion}), proving that the decomposition as a disjoint union of pieces is unique. 
This program is mirrored in Section~\ref{sec.vertexsum}, where we define the vertex sum operation on pointed spatial graphs, establish a framework for describing iterated vertex sums (as trees of spatial graphs), and show uniqueness of decompositions as trees of blocks, thus completing the proof of Theorem~\ref{thm.treeofblocks}. We point out that showing that the vertex sum operation is well-defined is one of the most technically demanding points of our program, with most of the hard work contained in the proof of Proposition~\ref{prop.penclosingball}.
Section~\ref{sec.decorations} is an addendum explaining how to adapt the theory developed so far to spatial graphs decorated with additional structure, namely edge orientations, vertex colorings, and edge colorings.

Section~\ref{sec.markedexterior} introduces one of the main characters of this paper, the marked exterior of a (decorated) spatial graph. We define it, explain how it encodes the spatial graph used to construct it, and translate indecomposability properties of spatial graphs into properties of their marked exteriors.
Finally, in Section~\ref{sec.algorithms} we import results from computational $3$-manifold topology in order to show that the canonical decompositions of Theorem~\ref{thm.treeofblocks} can be computed algorithmically, and apply Matveev's Theorem to marked exteriors of decorated blocks. Altogether, this culminates in the proof of our main result, Theorem~\ref{thm.algdetection}.
\section{Basic definitions} \label{sec.preliminaries}

In this short section, we formally define spatial graphs and introducing other basic terminology. We take a moment to remind the reader that all spaces considered are polyhedra: subspaces of some euclidean space~$\R^n$ having local cone neighborhoods at every point, and PL maps are defined as preserving this local cone structure \cite[Chapter~1]{RourkeSanderson}. Standard models of balls and spheres are defined using the $\ell^1$-norm, so they are effectively cubes and their boundaries. Orientations of PL manifolds are defined as PL isotopy classes of embeddings of balls \cite[pp.~43-46]{RourkeSanderson}.

\begin{dfn}\label{dfn.spatialgraph}
  A \textbf{spatial graph}~$\Gamma$ is a triple $(\SPH, V, E)$, where:
  \begin{itemize}
    \item $\SPH$~is an oriented $3$-sphere, called the \textbf{ambient sphere} of~$\Gamma$. We will occasionally say that ``$\Gamma$~is a spatial graph in~$\SPH$''.
    \item $V$~is a finite subset of~$\SPH$, whose elements are called \textbf{vertices} of~$\Gamma$, and
    \item $E$~is a finite set of subpolyhedra of~$\SPH$, called \textbf{edges} of~$\Gamma$, such that:
    \begin{itemize}
      \item each edge is PL-homeomorphic to an interval or to a PL $1$-sphere,
      \item each edge that is PL-homeomorphic to an arc intersects~$V$ precisely its endpoints,
      \item each edge that is PL-homeomorphic to a circle contains precisely one element of~$V$ (such edges are called \textbf{loops}),
      \item for every two distinct edges $e, e'$, we have $e\cap e' \subseteq V$.
    \end{itemize}
  \end{itemize}
  
  The \textbf{support} of~$\Gamma$ is the union
  \[|\Gamma| := V \cup \bigcup_{e\in E} e \subset \SPH.\]
  
  The \textbf{underlying graph} $\langle \Gamma \rangle$ of $\Gamma$ is the (undirected) abstract graph with vertex set~$V$ and edge set~$E$, where each edge is incident to the one or two elements of~$V$ that it contains. We will say that an edge of~$\Gamma$ is \textbf{incident} to a vertex if this is true in~$\langle \Gamma\rangle$. The \textbf{degree} of a vertex~$v$ is the same as its degree in~$\langle \Gamma \rangle$, that is, the number of edges incident to~$v$, with loops counting twice. A vertex of degree~$0$ is called an \textbf{isolated vertex}, and a vertex of degree~$1$ is called a \textbf{leaf}.
\end{dfn}

Observe that the two subsets $|\Gamma|$ and $V$ of~$\SPH$ determine the edge set, since there is a canonical bijection between~$E$ and $\pi_0 (|\Gamma| \setminus V )$.

\begin{dfn}
A \textbf{sub-graph} of a spatial graph~$\Gamma = (\SPH, V, E)$ is a spatial graph $\Gamma'=(\SPH, V', E')$, where $V'\subseteq V$ and $E' \subseteq E$.
\end{dfn}

\begin{dfn}\label{dfn.spatialgraphiso}
  Let $\Gamma_1=(\SPH_1, V_1, E_1)$ and $\Gamma_2 = (\SPH_2, V_2, E_2)$ be spatial graphs. An \textbf{isomorphism} $\Phi\colon \Gamma_1 \to \Gamma_2$ is a PL homeomorphism of triples $\Phi \colon (\SPH_1, |\Gamma_1|, V_1) \to (\SPH_2, |\Gamma_2|, V_2)$ respecting the orientation of the ambient spheres. If such~$\Phi$ exists, we say~$\Gamma_1, \Gamma_2$ are \textbf{isomorphic} and write $\Gamma_1 \cong \Gamma_2$.
\end{dfn}

By the characterization of the elements of~$E_1$ in terms of $|\Gamma_1|\setminus V_1$ (and similarly for~$E_2$), such~$\Phi$ also induces a bijection $E_1 \to E_2$, and we get an induced isomorphism of abstract graphs~$\langle \Phi \rangle \colon \langle \Gamma_1 \rangle \to \langle \Gamma_2 \rangle$.

We emphasize that there is stark loss of information in the passage from a spatial graph to its underlying graph.
In fact, one could claim that much of the field of knot theory is the study of the isomorphism classes of spatial graphs whose underlying graph is comprised of one vertex and one loop.

It will be convenient to loosen the notation by allowing ourselves to write an equality of spatial graphs ``$\Gamma_1 = \Gamma_2$'' whenever $\langle \Gamma_1 \rangle = \langle \Gamma_2 \rangle$ and there is an isomorphism $\Phi \colon \Gamma_1 \to \Gamma_2$ such that $\langle \Phi \rangle$~is the identity morphism.

Up to isomorphism, there is a unique spatial graph with no vertices (and hence also no edges), which we call the \textbf{empty spatial graph}, and denote by~$\zero$. Similarly, since the group of PL self-homeomorphisms of a $3$-sphere acts transitively on its points \cite[Lemma~3.33]{RourkeSanderson}, there is a unique spatial graph (up to isomorphism) with a single vertex and no edges. We call it a \textbf{one-point} spatial graph and denote it by $\one$.

\section{The disjoint union of spatial graphs}\label{sec.disjunion}

We want to define and establish properties of two operations on spatial graphs. For now, we focus on the disjoint union operation, and this will double as a warm-up for studying the more delicate vertex sum operation (Section~\ref{sec.vertexsum}). These operations implement constructions that are straightforward to define for abstract graphs, but in the setting of spatial graphs, a rigorous treatment requires some care.

\subsection{Assembling spatial graphs through disjoint unions}\label{sec.defdisjun}

In order to define the disjoint union of spatial graphs, we will need the following theorem:

\begin{thm}[Disc Theorem {\cite[Theorem~3.34]{RourkeSanderson}}]\label{thm.discthm}
  Every two orientation-preserving PL embeddings of an $n$-ball into the interior of a connected, oriented $n$-manifold~$M$ are PL-ambient-isotopic relative~$\partial M$.
\end{thm}

The above reference does not state that the ambient isotopy fixes~$\partial M$, but a closer inspection reveals that the stronger conclusion does follow from the proof. Later, we also present a stronger version of the Disc Theorem, which does include the boundary condition (Theorem~\ref{thm.discpair}).

\begin{dfn}
  An \textbf{enclosing ball} for a spatial graph~$\Gamma$ in~$\SPH$, is a PL-embedded $3$-ball~$B \subset \SPH$ such that $|\Gamma|\subset \intr(B)$.
\end{dfn}

\begin{dfn}\label{dfn.disjointunion}
  For each $i \in \{1,2\}$, let $\Gamma_i = (\SPH_i, V_i, E_i)$~be a spatial graph, and let $B_i$~be an enclosing ball for~$\Gamma_i$. Moreover, let $f\colon \partial B_1 \to \partial B_2$ be an orientation-reversing PL homeomorphism. Then the spatial graph
  \[\Gamma_1 \sqcup_f \Gamma_2 := (B_1 \cup_f B_2, V_1 \sqcup V_2, E_1 \sqcup E_2), \]
  where $B_1 \cup_f B_2$ denotes the $3$-sphere obtained by attaching $B_1$~to~$B_2$ using~$f$, is said to be a \textbf{disjoint union} of $\Gamma_1$~and~$\Gamma_2$.
\end{dfn}

We remark that gluing polyhedra along sub-polyhedra is a valid operation in the PL setting \cite[Exercise~2.27~(2)]{RourkeSanderson}.

As one would expect, the underlying graph~$\langle \Gamma_1 \sqcup_f \Gamma_2 \rangle$ is the disjoint union $\langle \Gamma_1 \rangle \sqcup \langle \Gamma_2 \rangle$, as usually defined for abstract graphs.

The next lemma and proposition show that the isomorphism type of a disjoint union of two spatial graphs does not depend on the choice of enclosing balls~$B_1, B_2$, nor on the attaching map~$f$.

\begin{lem}[Uniqueness of enclosing balls]\label{lem.enclosingball}
  Let $\Gamma$~be a spatial graph in~$\SPH$, and let $B, B'$~be enclosing balls for~$\Gamma$. Then every orientation-preserving PL homeomorphism $\Phi_\partial \colon \partial B \to \partial B'$ extends to an orientation-preserving PL homeomorphism $\Phi_B \colon B \to B'$ that restricts to the identity on~$|\Gamma|$.
\end{lem}

\begin{proof}
  Fix a regular neighborhood~$N_\Gamma$ of~$|\Gamma|$ in~$\SPH$ that is disjoint from~$\partial B \cup \partial B'$. Since $N_\Gamma$~and~$\SPH$ are $3$-manifolds, the subspace $\overline{N_\Gamma} := \SPH \setminus \intr(N_\Gamma)$ is also a $3$-manifold \cite[Corollary~3.14]{RourkeSanderson}. Moreover,
  since the closure of the complement of a PL-embedded $n$-ball in an $n$-sphere is always an $n$-ball \cite[Corollary~3.13]{RourkeSanderson}, we see that $\overline B :=  \SPH \setminus \intr (B)$ is a $3$-ball contained in $\intr(\overline{N_\Gamma})$ (and similarly for $\overline{B'} := \SPH \setminus \intr(B')$).
  
  Since a PL homeomorphism between the boundaries of two balls always extend to a PL homeomorphism of their interior \cite[Lemma~1.10]{RourkeSanderson}, we may extend $\Phi_\partial$ to an orientation-preserving PL homeomorphism $\Phi_{\overline B} \colon \overline{B} \to \overline{B'}$. Then we apply the Disc Theorem (Theorem~\ref{thm.discthm}) to produce a PL ambient isotopy of~$\overline{N_\Gamma}$ taking the inclusion $\overline{B} \hookrightarrow \overline{N_\Gamma}$ to the composition
  \[ \overline B \overset{\Phi_{\overline B}}{\longrightarrow} \overline{B'} \into \overline{N_\Gamma}.\]
  Since this ambient isotopy keeps $\partial N_\Gamma$~fixed, its final homeomorphism $\Phi_{\overline{N_\Gamma}} \colon \overline{N_\Gamma} \to \overline{N_\Gamma}$ can be extended to all of~$\SPH$ by setting it to the identity on~$N_\Gamma$. This extension $\Phi_\SPH \colon \SPH \to\SPH$, when restricted to~$B$, is a PL homeomorphism~$\Phi_B \colon B \to B'$ satisfying the conclusion of the lemma.
\end{proof}

\begin{prop}[Well-definedness of the disjoint union]\label{prop.uniquenessdisjunion}
  For each $i \in \{1,2\}$, let $\Gamma_i$~be a spatial graph in~$\SPH_i$, and let $B_i, B_i'$~be two enclosing balls for~$\Gamma_i$. Moreover, let $f\colon \partial B_1 \to \partial B_2$ and $f'\colon \partial B_1' \to \partial B_2'$ be orientation-reversing PL homeomorphisms. Then there is an isomorphism $\Phi \colon \Gamma_1 \sqcup_f \Gamma_2 \cong \Gamma_1 \sqcup_{f'} \Gamma_2$ such that $\langle\Phi\rangle$~is the identity on~$\langle \Gamma_1 \rangle \sqcup \langle \Gamma_2 \rangle$.
\end{prop}

Note that in the abbreviated notation introduced in Section~\ref{sec.preliminaries}, the conclusion of this proposition can be rephrased as ``$\Gamma_1 \sqcup_f \Gamma_2 = \Gamma_1 \sqcup_{f'} \Gamma_2$''.

\begin{proof}
By Lemma~\ref{lem.enclosingball}, there is an orientation-preserving PL homeomorphism $\Phi_1 \colon B_1 \to B_1'$ restricting to the identity on~$|\Gamma_1|$. Using the same lemma, let $\Phi_2 \colon B_2 \to B_2'$ be an orientation-preserving PL homeomorphism fixing~$|\Gamma_2|$ and whose restriction to~$\partial B_2$ is $f' \circ \Phi_1|_{\partial B_1} \circ f^{-1}$.
Then the maps~$\Phi_i$ assemble to a PL homeomorphism $\Phi \colon B_1 \sqcup_f B_2 \to B_1' \sqcup_{f'} B_2'$ giving the desired isomorphism between $\Gamma_1 \sqcup_f \Gamma_2$ and~$\Gamma_1\sqcup_{f'}\Gamma_2$. The fact that each~$\Phi_i$ restricts to the identity on $|\Gamma_i|$ implies that $\langle \Phi \rangle$~is the identity on~$\langle \Gamma_1 \rangle \sqcup \langle \Gamma_2 \rangle$.
\end{proof}

The disjoint union~$\Gamma_1 \sqcup_f \Gamma_2$ is thus well-defined without specifying enclosing balls~$B_1, B_2$ nor the attaching map~$f$, up to isomorphism of spatial graphs respecting the underlying combinatorial structure. Hence we will from now on most of the time suppress the $f$-subscript from the notation.

We now collect two elementary observations:

\begin{lem}[Disjoint union summands as sub-graphs]\label{lem.summandissubgraph}
Let $\Gamma = \Gamma_1 \sqcup \Gamma_2$ be a disjoint union of spatial graphs, and denote by~$\Gamma_1'$ the sub-graph of~$\Gamma$ obtained by discarding all vertices and edges of~$\Gamma_2$. Then $\Gamma_1' = \Gamma_1$.
\end{lem}
\begin{proof}
For each $i \in \{1,2\}$, let $\SPH_i$~be the ambient sphere for~$\Gamma_i$, let $B_i \subset \SPH_i$ be the enclosing ball from which the disjoint union was formed, and let $f \colon \partial B_1 \to \partial B_2$ be the attaching map. Our task is to find a PL homeomorphism $\Phi \colon \SPH_1 \to B_1 \cup_f B_2$ that restricts to the identity on~$|\Gamma_1|$. Setting $\Phi$~as the identity map on~$B_1$, we are left to find a PL homeomorphism $\overline{B_1} \to B_2$ extending~$f$, where $\overline{B_1} := \SPH_1 \setminus \intr(B_1)$. Such an extension always exits \cite[Lemma~1.10]{RourkeSanderson}.
\end{proof}

When working with such a disjoint union, we will often refer to the summand~$\Gamma_1$ as a sub-graph of~$\Gamma$ without explicit mention of this lemma.

\begin{lem}[Disjoint union of isomorphisms]\label{lem.isodisjunion}
Let $\Phi_1 \colon \Gamma_1 \to \Gamma_1'$ and $\Phi_2 \colon \Gamma_2 \to \Gamma_2'$ be isomorphisms of spatial graphs. Then there exists an isomorphism
\[\Phi_1 \sqcup \Phi_2 \colon \Gamma_1 \sqcup \Gamma_2 \to \Gamma_1' \sqcup \Gamma_2'\]
such that for each $i \in \{1,2\}$ the underlying isomorphism of abstract graphs $\langle \Phi_1 \sqcup \Phi_2\rangle$ restricts to~$\langle \Phi_i \rangle$ on~$\langle \Gamma_i \rangle$.
\end{lem}
\begin{proof}
Form the disjoint union~$\Gamma_1 \sqcup_f \Gamma_2$ by using a suitable PL homeomorphism  $f \colon \partial B_1 \to \partial B_2$ between the boundaries of enclosing balls for~$\Gamma_1, \Gamma_2$. Writing $B_1':= \Phi_1(B_1)$, $B_2':= \Phi_2(B_2)$, and defining $f' \colon \partial B_1' \to \partial B_2'$ as $f' := \Phi_2|_{\partial B_2} \circ f\circ \Phi_1^{-1}|_{\partial B_1'}$, we can form the disjoint union  $\Gamma_1' \sqcup_{f'} \Gamma_2'$. The restrictions~$\Phi_i|_{B_i}$ then assemble to the desired isomorphism $\Phi_1 \sqcup \Phi_2$.
\end{proof}

Note that this lemma strongly depends on the fact that the ambient $3$-spheres of spatial graphs carry an orientation, which is preserved by isomorphisms. If we were to drop this requirement, then a spatial graph~$\Gamma$ comprised of one vertex and one edge in the shape of a trefoil would be isomorphic to its mirror-image~$\tilde \Gamma$. The spatial graphs $\Gamma \sqcup \Gamma$ and $\Gamma \sqcup \tilde \Gamma$ would however not be isomorphic.

We finish this subsection by recording basic algebraic properties of the disjoint union.

\begin{prop}[Properties of the disjoint union]\label{prop.disjunprop}
	Let $\Gamma_1, \Gamma_2, \Gamma_3$ be spatial graphs. Then:
  	\begin{itemize}
    \item \emph{identity element:} $\Gamma_1 \sqcup \zero = \Gamma_1$,
    \item \emph{commutativity:} $\Gamma_1 \sqcup \Gamma_2 = \Gamma_2 \sqcup \Gamma_1$,
    \item \emph{associativity:} $(\Gamma_1 \sqcup \Gamma_2) \sqcup \Gamma_3 = \Gamma_1 \sqcup (\Gamma_2 \sqcup \Gamma_3)$.
  	\end{itemize}
\end{prop}

\begin{proof}
	The first statement follows from Lemma~\ref{lem.summandissubgraph}.

	Commutativity is straightforward if one uses the same enclosing balls for both disjoint unions, and mutually inverse attaching maps.
	
	For associativity, we need only a bit of care when choosing enclosing balls along which to perform the disjoint unions. The idea should be clear from the schematic in Figure~\ref{fig.associativity}, but we now supply a bit more detail. Denote the ambient $3$-sphere of each~$\Gamma_i$ by~$\SPH_i$. Let $B_1$, $B_3$ be enclosing balls for $\Gamma_1, \Gamma_3$, respectively, and let $B_{21}$, $B_{23}$ be enclosing balls for~$\Gamma_2$ such that $\intr(B_{21})\cup \intr(B_{23}) = \SPH_2$. Equivalently, the $3$-balls $\overline{B_{21}}:= \SPH \setminus \intr(B_{21})$ and $\overline{B_{23}}:= \SPH \setminus \intr(B_{23})$ should be disjoint. After fixing attaching maps $f_1\colon \partial B_1 \to \partial B_{21}$, $f_3 \colon \partial B_{23} \to \partial B_3$, it follows that $B_1 \cup_{f_1} (B_{21} \cap B_{23})$ is an enclosing ball for $\Gamma_1 \sqcup_{f_1} \Gamma_2$, and $(B_{21} \cap B_{23}) \cup_{f_3} B_3$ is an enclosing ball for $\Gamma_2 \sqcup_{f_3} \Gamma_3$. The spatial graphs $(\Gamma_1 \sqcup_{f_1} \Gamma_2) \sqcup_{f_3} \Gamma_3$ and $\Gamma_1 \sqcup_{f_1} (\Gamma_2 \sqcup_{f_3} \Gamma_3)$ are then the same on the nose. 
	\begin{figure}[h]
	      \centering
	      \def \svgwidth{0.75\linewidth}
	      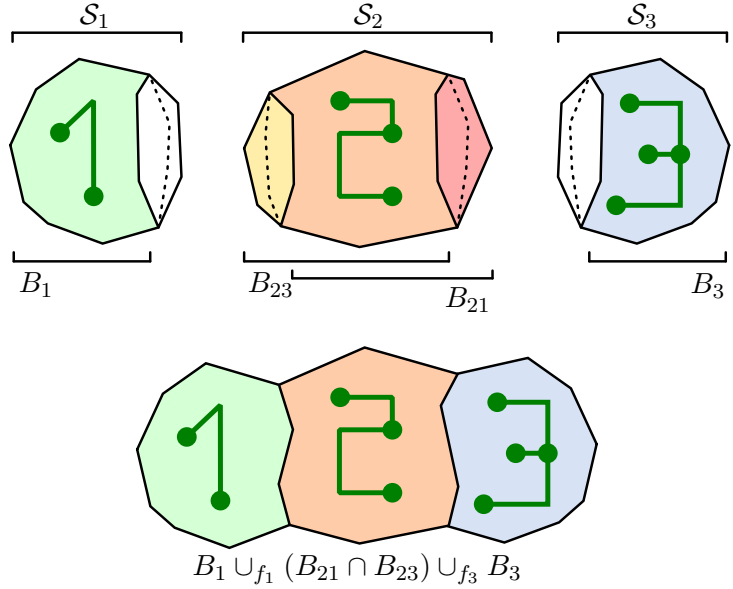
	      \caption{The proof of associativity of the disjoint union, with ambient spheres and enclosing balls depicted one dimension below. Top: the ambient spheres and enclosing balls for the spatial graphs~$\Gamma_i$. Bottom: the spatial graph $\Gamma_1 \sqcup_{f_1} \Gamma_2 \sqcup_{f_3} \Gamma_3$ in its ambient sphere $B_1 \cup_{f_1} (B_{21} \cap B_{23}) \cup_{f_3} B_3$.}
	      \label{fig.associativity}
	  \end{figure}
\end{proof}

Commutativity and associativity allow us to unambiguously write down iterated disjoint unions. More precisely, if $\{ \Gamma_i \}_{i \in I}$ is a collection of spatial graphs indexed by a finite set~$I$, then $\bigsqcup_{i\in I} \Gamma_i$ is well-defined up to isomorphism inducing the identity on~$\bigsqcup_{i \in I} \langle \Gamma_i \rangle$.

\subsection{Decomposing spatial graphs as disjoint unions}

We now start working towards proving that every spatial graph can be expressed as an iterated disjoint union in a canonical way. We will (often implicitly) use the fact that every PL-embedded $2$-sphere in a PL $3$-sphere decomposes it into two $3$-balls. Although the topological version of this statement is known to be false by the famous counter-example of the ``Alexander horned sphere'', it holds in the PL setting, also due to work of Alexander \cite{Alex24}.

We begin with a simple observation.

\begin{lem}[If it looks like a disjoint union, it is a disjoint union]\label{lem.spheresdontlie}
Let $\Gamma$~be a spatial graph in~$\SPH$ and $S\subset \SPH \setminus |\Gamma|$ a PL-embedded $2$-sphere. Denote the closures of the two components of~$\SPH \setminus S$ by $B_1$~and~$B_2$. For each $i \in \{1,2\}$, let $\Gamma_i$~be the sub-graph of~$\Gamma$ comprised of the vertices and edges that are contained in~$B_i$. Then $\Gamma = \Gamma_1 \sqcup \Gamma_2$.
\end{lem}
\begin{proof}
We use Lemma~\ref{lem.summandissubgraph} to regard each $\Gamma_i$ as a sub-graph of~$\Gamma_1 \sqcup \Gamma_2$, and take~$B_i$ as an enclosing ball for~$\Gamma_i$. If $f \colon S \to S$ is the identity map, then~$\Gamma = \Gamma_1 \sqcup_f \Gamma_2$.
\end{proof}

Of course by definition of the disjoint union, if $\Gamma= \Gamma_1 \sqcup \Gamma_2$, then there exists such a sphere~$S$.

\begin{dfn} \label{dfn.septhings}
Let $\Gamma$ be a spatial graph in~$\SPH$.
  \begin{itemize}
  	\item If $S\subset \SPH \setminus |\Gamma|$ is a $2$-sphere and $\Gamma_1, \Gamma_2$ are as in Lemma~\ref{lem.spheresdontlie}, we say that ``$S$ decomposes $\Gamma$ as $\Gamma_1 \sqcup \Gamma_2$''.
  
    \item $\Gamma$ is said to be \textbf{separable} if it is the disjoint union of two non-empty spatial graphs; otherwise it is \textbf{non-separable}.
    
    \item If $S$~is a $2$-sphere in~$\SPH$ decomposing $\Gamma$ as~$\Gamma_1 \sqcup \Gamma_2$ with $\Gamma_1, \Gamma_2$ non-empty, then $S$~is called a \textbf{separating sphere} for~$\Gamma$.
    
    \item We will call a spatial graph a \textbf{piece} if it is non-empty and non-separable. We will also say that a spatial graph~$\Lambda$ is a piece of~$\Gamma$ if $\Lambda$ is a piece and~$\Gamma = \Gamma' \sqcup \Lambda$ for some~$\Gamma'$.
  \end{itemize}
\end{dfn}

We use the word ``piece'', rather than ``component'', to avoid suggesting that for such~$\Lambda$, the support~$|\Lambda|$ (or equivalently the underlying graph~$\langle \Lambda \rangle$) would have to be connected. Indeed, a spatial graph with non-connected support may very well be non-separable. Take, for example, a spatial graph isotopic to a Hopf link, such as the one in Figure~\ref{fig.hopflink}.

  \begin{figure}[h]
      \centering
      \def \svgwidth{0.2\linewidth}
      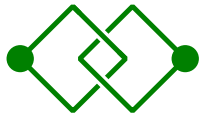
      \caption{The notion of non-separability of spatial graphs does not coincide with the notion of connectedness. The depicted spatial graph~$\Gamma$ is non-separable, but its support~$|\Gamma|$ and underlying graph~$\langle \Gamma\rangle$ are disconnected.}
      \label{fig.hopflink}
  \end{figure}

Every spatial graph~$\Gamma$ can be decomposed as a disjoint union of finitely many pieces: if $\Gamma = \zero$ we take the empty union, and if $\Gamma$~is itself a piece, we take a disjoint union indexed over a one-element set. If $\Gamma$ is non-empty and not a piece, then it can be expressed as a disjoint union of two non-empty graphs $\Gamma = \Gamma_1 \sqcup \Gamma_2$, each~$\Gamma_i$ thus having strictly fewer vertices than~$\Gamma$. By induction on the number of vertices, the~$\Gamma_i$ can be decomposed into pieces, and hence so can~$\Gamma$. 

We now work towards proving that such a decomposition is unique.

\begin{lem}[Spheres sort pieces]\label{lem.sortedpieces}
Let~$\Lambda$ be a piece in~$\SPH$, and let $S\subset \SPH \setminus |\Lambda|$ be a PL-embedded $2$-sphere. Denote the closures of the two components of~$\SPH \setminus S$ by~$B_1, B_2$. Then $|\Lambda|$ is contained in exactly one of the~$B_i$.
\end{lem}

\begin{proof}
Since $\Lambda \ne \zero$, it can certainly not be contained in both~$B_i$. For each $i \in \{1,2\}$, denote by~$\Lambda_i$ the sub-graph of~$\Lambda$ whose vertices and edges are contained in~$B_i$. Then by Lemma~\ref{lem.spheresdontlie} we have that $S$ decomposes~$\Lambda$ as $\Lambda_1 \sqcup \Lambda_2$. Since $\Lambda$~is non-separable, one of the summands, say~$\Lambda_1$, is empty. By the first part of Proposition~\ref{prop.disjunprop}, it follows that $\Lambda_2 = \Lambda$.
\end{proof}

\begin{prop}[Uniqueness of decomposition into pieces]\label{prop.uniquepieces}
Let $(\Lambda_i)_{i \in I_1}$ and $(\Lambda_i)_{i \in I_2}$ be collections of pieces indexed by finite sets $I_1, I_2$. Then for every isomorphism of spatial graphs $\Phi \colon \bigsqcup_{i \in I_1} \Lambda_i \to \bigsqcup_{i \in I_2} \Lambda_i$, there is a bijection $f \colon I_1 \to I_2$ such that for each $i \in I_1$, the PL homeomorphism~$\Phi$ is an isomorphism of the sub-graphs~$\Phi \colon \Lambda_i \to \Lambda_{f(i)}$.
\end{prop}

\begin{proof}
Write $\Gamma_1:= \bigsqcup_{i \in I_1} \Lambda_i$ and $\Gamma_2:=\bigsqcup_{i \in I_2} \Lambda_i$. We induct on the cardinality of~$I_1$.

If $I_1 = \emptyset$ then $\Gamma_1 = \zero = \Gamma_2$, so since for all~$i \in I_2$ we know $\Lambda_i$~is non-empty, we conclude $I_2 = \emptyset$ and there is nothing left to show. If~$I_1$ contains only one element~$i_1$, then $\Gamma_1 = \Lambda_{i_1}$ is a piece. Hence $\Gamma_2$~is also a piece and therefore, again since the~$\Lambda_i$ are non-empty, we conclude $I_2=\{i_2\}$. We thus set $f(i_1) := i_2$.

If $I_1$~has more than one element, we partition it into two non-empty subsets $I_1 = I_1 ^+ \sqcup I_1^-$. Let $S_1$~be a $2$-sphere decomposing~$\Gamma_1$ as $\bigl(\bigsqcup_{i \in I_1^+} \Lambda_i\bigr) \sqcup \bigl(\bigsqcup_{i \in I_1^-} \Lambda_i\bigr)$, and write $\Gamma_1^+ := \bigsqcup_{i \in I_1^+} \Lambda_i$ and $\Gamma_1^- := \bigsqcup_{i \in I_1^-} \Lambda_i$.

Now $S_2 := \Phi(S_1)$ is a $2$-sphere in the ambient sphere of~$\Gamma_2$ disjoint from~$|\Gamma_2|$. One side of~$S_2$ corresponds to the ``$+$''-summand of~$\Gamma_1$, and the other to the ``$-$''-summand. By Lemma~\ref{lem.sortedpieces}, for each $i \in I_2$, we have $|\Lambda_i|$ contained in either the ``$+$''-side or the ``$-$''-side of~$S_2$. Partition $I_2$ accordingly as~$I_2 = I_2^+ \sqcup I_2^-$, and write $\Gamma_2^\pm := \bigsqcup_{i \in I_2^\pm} \Lambda_i$.

Since $\Phi$~maps the support $|\Gamma_1^+|$ into $|\Gamma_2^+|$, and similarly for ``$-$'', we conclude $\Phi$ doubles as a pair of isomorphisms of sub-graphs $\Phi^\pm \colon \Gamma_1^\pm \to \Gamma_2^\pm$. Both $I_1^\pm$ have fewer elements than~$I_1$, so by induction we obtain bijections $f^\pm \colon I_1^\pm \to I_2^\pm$, which assemble to the required $f \colon I_1 \to I_2$.
\end{proof}

By an iterated application of Lemma~\ref{lem.isodisjunion}, we see the converse direction of this proposition also holds, that is, any two such isomorphic collections of pieces assemble to isomorphic disjoint unions. Hence, if we are able to decompose spatial graphs $\Gamma_1, \Gamma_2$ as disjoint unions of pieces, then testing whether $\Gamma_1$~is isomorphic to~$\Gamma_2$ boils down to testing whether the pieces are pairwise isomorphic. However, our core machinery for testing isomorphisms of spatial graphs (Proposition~\ref{prop.blockrecognition}) still needs pieces to be further decomposed. Indeed, there is a further simplification of pieces that is canonical in much the same way as the decomposition of a general spatial graph into pieces. This decomposition is the subject of the next section.

\section{The vertex sum of spatial graphs}\label{sec.vertexsum}

We now define another operation, the ``vertex sum'', whose input data is a pair of spatial graphs with distinguished vertices. The relevant case of the construction is when they are non-separable, but we will nevertheless formulate our definitions and statements without this hypothesis until it becomes indispensable.
The overall structure of this section will be rather similar to that of the previous one, with many definitions and results having obvious counterparts.

\subsection{Defining the vertex sum}

\begin{dfn}
  A \textbf{pointed spatial graph} is a pair~$(\Gamma, v)$ of a spatial graph~$\Gamma$ and a vertex~$v$ of~$\Gamma$. The underlying graph of a pointed spatial graph is pointed with the same distinguished vertex: $\langle (\Gamma, v) \rangle := (\langle \Gamma \rangle, v)$. An isomorphism of pointed spatial graphs is an isomorphism of the spatial graphs that preserves the distinguished vertices.
\end{dfn}

\begin{dfn}
  An \textbf{enclosing ball} for a pointed spatial graph $(\Gamma, v)$ in~$\SPH$, is a PL-embedded $3$-ball $B \subset \SPH$ such that $|\Gamma|\subset B$ and~$|\Gamma | \cap \partial B = \{v\}$.
\end{dfn}

\begin{dfn}\label{dfn.vsum}
  For each $i \in \{1,2\}$, let $\Gamma_i = (\SPH_i, V_i, E_i)$~be a non-empty spatial graph, let $v_i \in V_i$, and let $B_i$~be an enclosing ball for~$(\Gamma_i, v_i)$. Moreover, let $f\colon \partial B_1 \to \partial B_2$ be an orientation-reversing PL homeomorphism mapping $v_1$~to~$v_2$. We consider the spatial graph
  \[\Gamma_1 \vsum{v_1}{v_2} \Gamma_2 := (B_1 \cup_f B_2, (V_1 \sqcup V_2) / v_1 \sim v_2, E_1 \sqcup E_2), \]
  where $B_1 \cup_f B_2$ denotes the PL $3$-sphere obtained by attaching $B_1$~to~$B_2$ using~$f$, and define the pointed spatial graph~$(\Gamma_1 \vsum{v_1}{v_2} \Gamma_2, v_1=v_2)$ to be a \textbf{vertex sum} of $(\Gamma_1, v_1)$~and~$(\Gamma_2, v_2)$.
\end{dfn}

We will use the same notation to denote the analogous operation on pointed abstract graphs. For every such vertex sum of pointed spatial graphs we thus have $\langle \Gamma_1 \vsum{v_1}{v_2} \Gamma_2 \rangle = \langle \Gamma_1 \rangle \vsum{v_1}{v_2} \langle \Gamma_2 \rangle$.

We retrace the steps taken when discussing the disjoint union of spatial graphs, our next goal being to show that the vertex sum of pointed spatial graphs is independent of the choice of enclosing spheres and the attaching map~$f$. The key is the following proposition.

\begin{prop}[Uniqueness of enclosing balls for pointed spatial graphs]\label{prop.penclosingball}
  Let $(\Gamma, v)$~be a pointed spatial graph in~$\SPH$, let $B, B'$~be enclosing balls for~$(\Gamma, v)$. Then every orientation-preserving PL homeomorphism $\Phi_\partial \colon (\partial B, v) \to (\partial B', v)$ extends to an orientation-preserving PL homeomorphism $\Phi_B \colon B \to B'$ restricting to the identity on~$|\Gamma|$.
\end{prop}

Proving this proposition will require substantially more work than its non-pointed counterpart, Lemma~\ref{lem.enclosingball}, as one should expect from the very particular behavior demanded of~$\Phi$ near~$v$. One of the necessary ingredients will be a generalization of the Disc Theorem (Theorem~\ref{thm.discthm}).

\begin{dfn}[{\cite[p.~50]{RourkeSanderson}}]
	An \textbf{unknotted ball pair}~$(B, B_0)$ is a pair of polyhedra PL-homeomorphic to a standard ball pair $([-1,1]^n, [-1,1]^m\times \{0\}^{n-m})$ (for some $n \ge m \ge 0$). A \textbf{PL manifold pair}~$(M,M_0)$ is a pair of polyhedra that are manifolds, such that $\partial M \cap M_0 = \partial M_0$ (``properness''), and  such that each point of~$M_0$ has a neighborhood in~$(M, M_0)$ PL-homeomorphic to an unknotted ball pair (``local flatness'')\footnote{The definition given by Rourke-Sanderson on p.~50 requires only that $M, M_0$~both be manifolds, but the remark on p.~51 adds the local flatness and properness conditions.}.
\end{dfn}

\begin{thm}[Disc Theorem for pairs {\cite[Theorem~4.20]{RourkeSanderson}}] \label{thm.discpair}
  Let~$(M, M_0)$ be a pair of connected, oriented PL manifolds, let $(B, B_0)$~be an unknotted ball pair with the same dimensions, and let $\iota_1, \iota_2\colon (B, B_0) \to (\intr(M), \intr(M_0))$ be PL embeddings that preserve the orientation on both components. Then there is a PL ambient isotopy of~$(M, M_0)$ relative~$\partial M$ that carries $\iota_1$ to~$\iota_2$.
\end{thm}

The reason we need the Disc Theorem for pairs is because it has the following corollary:

\begin{cor}[Disc Theorem at the boundary] \label{cor.boundarydisc}
  Let $M$ be a connected, oriented PL $n$-manifold, let $N\subseteq \partial M$ be a connected PL-embedded $(n-1)$-manifold that is closed in~$\partial M$, let $B$~be a PL $n$-ball, and $D \subset \partial B$ a PL $(n-1)$-ball. For every two orientation-preserving PL embeddings $\iota_1, \iota_2\colon (B , D) \to (\intr(M) \cup \intr(N), \intr(N))$, there is a PL ambient isotopy of $(M, N)$ relative $\partial M \setminus \intr(N)$ carrying $\iota_1$ to~$\iota_2$.
\end{cor}

\begin{proof}
  We consider the double~$\Db_N(M)$ of $M$ along~$N$, which is a union of two copies of~$M$ glued along the identity map on~$N$, one of the copies with its orientation reversed. Using the fact that $N$~is closed in~$\partial M$ one sees that $(\Db_N(M), N)$~is a PL manifold pair, and its boundary is~$(\Db_{\partial N} (\partial M \setminus \intr(N)), \partial N)$. 
  Doubling also~$B$ along~$D$ yields an unknotted ball pair~$(\Db_D(B),D)$.
  
  Now, the maps $\iota_1, \iota_2$~extend to orientation-preserving PL embeddings 
  \[\Db(\iota_1), \Db(\iota_2)\colon(\Db_D(B), D) \to (\intr(\Db_N(M)), \intr(N)).\]
  Theorem~\ref{thm.discpair} yields a PL ambient isotopy of $(\Db_N(M), N)$ relative $\Db_{\partial N} (\partial M \setminus \intr(N))$ that carries $\Db(\iota_1)$ to $\Db(\iota_2)$. A connectivity argument shows that it restricts to an isotopy from $\iota_1$~to~$\iota_2$ relative~$\partial M \setminus \intr(N)$.
\end{proof}

We will also need the observation in Lemma~\ref{lem.annulus} below, but before stating it we remind the reader  of some standard terminology.

\begin{dfn}Let $P$~be a polyhedron in some~$\R^n$, and let $v \in \R^n$. We write~$vP$ to denote the polyhedron comprised of all points of the form~$t p + (1-t)v$, with $p \in P$ and $t \in [0,1]$. If each point of~$vP$ admits a unique such expression, we say $vP$~is a \textbf{cone} with base~$P$ and vertex~$v$.

Given two cones $vP, wQ$ with bases $P,Q$ and vertices $v,w$, respectively, and a PL map $f \colon P \to Q$, the \textbf{cone} of~$f$  (with respect to~$v, w$) \cite[Exercise~1.6(3)]{RourkeSanderson} is the PL map $vP \to wQ$ given by 
\[ tp + (1-t)v \mapsto tf(p) + (1-t)w.\]
\end{dfn}

\begin{lem}[Interpolating annulus]\label{lem.annulus}
Let $A_0$~be a PL annulus in some~$\R^n$, and let $v\in \R^n$ be such that $v A_0$~is a cone with base~$A_0$ and vertex~$v$. Denote the two boundary circles of~$A_0$ by~$\gamma_0, \delta_0$, and let $\gamma\subset v\gamma_0$ and $\delta \subset v\delta_0$ be PL circles such that $v \gamma,v \delta$ are cones with bases $\gamma, \delta$ respectively, and vertex~$v$. Then there exists a PL annulus $A \subset v A_0$ with $\partial A = \gamma \cup \delta$, such that $vA$ is a cone with base~$A$ and vertex~$v$.
\end{lem}

This lemma is illustrated in Figure~\ref{fig.annulus}.

\begin{figure}[h]
   \centering
   \def \svgwidth{0.4\linewidth}
   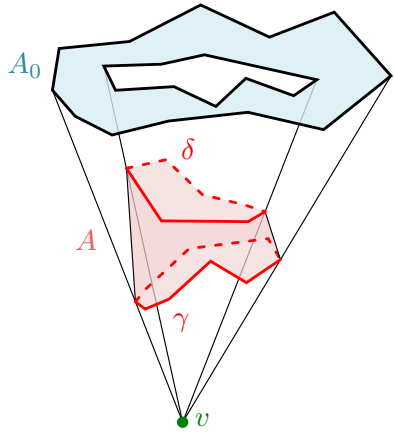
   \caption{The setup of Lemma~\ref{lem.annulus}. The PL annulus~$A$ ``interpolates'' between the PL circles $\gamma, \delta$.}
   \label{fig.annulus}
\end{figure}

\begin{proof}
We may assume without loss of generality that $A_0 = C \times [0,1] \subset \R^{n}$ for some PL circle~$C\subset \R^{n-1}$, with $\gamma_0 = C\times \{0\}$ and $\delta_0 = C \times \{1\}$, because for every PL homeomorphism $C \times [0,1] \to A_0$, the cone $v (C \times [0,1]) \to vA_0$ preserves cones at~$v$.

Choose a finite set of points in~$\gamma \subset v(C \times \{0\})$ subdividing~$\gamma$ into straight line segments (see Rourke-Sanderson for details \cite[Theorem~2.2]{RourkeSanderson}). Pushing these points radially onto~$\gamma_0 = C \times \{0\}$ and projecting onto~$C$ yields a finite set of points in~$C$ (note that since $v\gamma$~is a cone by assumption, no two points from~$\gamma$ get pushed onto the same point of~$\gamma_0$). Doing the same with~$\delta$ yields a second finite subset of~$C$. Finally, choose a third finite subset of~$C$ subdividing $C$~itself into straight line segments. We denote by $p_1, \ldots, p_k$ the points in the union of these three subsets, ordered cyclically around~$C$. The indices $1, \ldots, k$ should thus be interpreted as lying in~$\Z / k$. We now push the points $(p_1, 0), \ldots, (p_k,0) \in \gamma_0$ radially onto~$\gamma$ to obtain points $p_1^\gamma, \ldots, p_k^\gamma$. Similarly, pushing~$(p_1, 1), \ldots,  (p_k, 1)$ radially yields $p_1^\delta, \ldots, p_k^\delta$. 

Since the points~$p_i$ subdivide $C$~into straight line segments, we see that for each~$i \in \Z / k$, the points $(p_i, 0), (p_{i+1} , 0), (p_i, 1), (p_{i+1}, 1)$ are the vertices of a rectangle~$R_i$ contained in~$A_0$.
In particular, the cone~$vR_i\subset vA_0$ is convex. This is the crucial observation that will allow us to find the desired annulus~$A$.

For each $i \in \Z/k$, denote by~$T_i^\gamma$ the triangle spanned by the points $p_i^\gamma, p_{i+1}^\gamma, p_i^\delta$, and by~$T_i^\delta$ the one spanned by $p_i^\delta, p_{i+1}^\delta, p_{i+1}^\gamma$. By the previous observation, both of these triangles are contained in~$vR_i$. The union $A:=\bigcup_{i \in \Z / k} (T_i^\gamma \cup T_i^\delta)$ is then a PL annulus embedded in~$vA_0$, with $\partial A = \gamma \cup \delta$. It is also easy to see that each point of~$A$ lies in a unique ray from~$v$ through a point in~$A$, whence the cone condition on~$vA$ follows.
\end{proof}

Finally we are equipped to prove Proposition~\ref{prop.penclosingball}.

\begin{proof}[Proof of Proposition~\ref{prop.penclosingball}]
 As in the proof of Lemma~\ref{lem.enclosingball}, write $\overline{B} := \SPH \setminus \intr(B)$ and $\overline{B'}:=\SPH \setminus \intr(B')$, and choose any extension of~$\Phi_\partial$ to a PL homeomorphism $\Phi_{\overline B}\colon \overline B \to \overline{B'}$.  We will find an extension~$\Phi_\SPH \colon \SPH \to \SPH$ of~$\Phi_{\overline{B}}$ that fixes~$|\Gamma|$, and whose restriction~$\Phi_B$ to~$B$ will therefore satisfy the conclusion of the lemma. The construction of this extension~$\Phi_\SPH$ is somewhat intricate, so we need to introduce some notation, which we illustrate in Figure~\ref{fig.enclosingsphnbhds}.
 
   \begin{figure}[h]
       \centering
       \def \svgwidth{0.95\linewidth}
       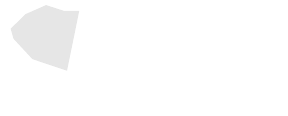
       \caption{The $3$-balls $\overline B$~and~$\overline{B'}$ in the setup of the proof of Lemma~\ref{prop.penclosingball}. We emphasize the action of~$\Phi_{\overline B}$ on~$vD$ as the cone of a PL homeomorphism~$D \to D'$.}
       \label{fig.enclosingsphnbhds}
   \end{figure}
 	 
	First, choose a star neighborhood~$N_0$ for~$v$ in the pair $(\SPH, \overline{B'} \cup |\Gamma|)$.
	More explicitly, $N_0$~is a $3$-ball such that the polyhedron $(\overline{B'} \cup |\Gamma|) \cap N_0$ is a cone with base its intersection with~$\partial N_0$, and with vertex~$v$ \cite[p.~50]{RourkeSanderson}. In particular, $D_0:= \overline{B'} \cap \partial N_0$ is a disc and $\overline{B'} \cap N_0$~is the cone~$vD_0$ with base~$D_0$ and vertex~$v$.
	
	 We then pick a smaller star neighborhood~$N_v \subset \intr(N_0)$ of~$v$ in~$(\SPH, \overline{B'} \cup |\Gamma|)$,
	 such that $N_v$~is also a star neighborhood of~$v$ in~$(\SPH, \overline B \cup |\Gamma|)$,
	 and $\overline B \cap N_v$ is mapped conically by~$\Phi_{\overline B}$ into~$\intr(N_0)$.
	 Denoting by~$D$ the disc~$\overline B \cap \partial N_v$, so $\overline B \cap N_v$~is a cone~$vD$ with base~$D$ and vertex~$v$, this means that $\Phi_{\overline B}(vD)$~is a cone~$vD'$ with base the disc $D':=  \Phi_{\overline B}(D)$ and vertex~$v$, and that $\Phi_{\overline B}|_{vD}\colon vD \to vD'$~is the cone of $\Phi_{\overline B}|_D\colon D \to D'$.
	 The existence of such~$N_v$ follows from the definitions of PL map and polyhedron, say, by taking $N_v$~to be a sufficiently small $\epsilon$-neighborhood of~$v$.
	 We will denote by~$\overline{N_v}$ the $3$-ball~$\SPH \setminus \intr(N_v)$.
	 
	 In order to apply the disc theorem at the boundary, we will first need to move~$\overline{B'}$ into a nicer configuration. More precisely, we will use the following fact, illustrated in Figure~\ref{fig.enclosingsphpsi}.
	 
	 \begin{claim}
	 There exists an orientation-preserving PL homeomorphism $\Psi\colon \SPH \to \SPH$ such that:
	 \begin{itemize}
	 \item $\Psi$~maps the pair $(\Phi_{\overline B}(\overline B \cap \overline{N_v}), D')$ into the pair $(\overline{N_v}, \partial N_v)$,
	 \item writing $\tilde D := \Psi(D')$, the map~$\Psi$ is given on~$vD'$ as the cone $v D' \to v\tilde D$ of the PL homeomorphism~$D' \to \tilde D$, and
	 \item $\Psi$ fixes~$|\Gamma|$.
	 \end{itemize}
	 \end{claim}
	 
	\begin{figure}[h]
	  \centering
	  \def \svgwidth{0.95\linewidth}
	  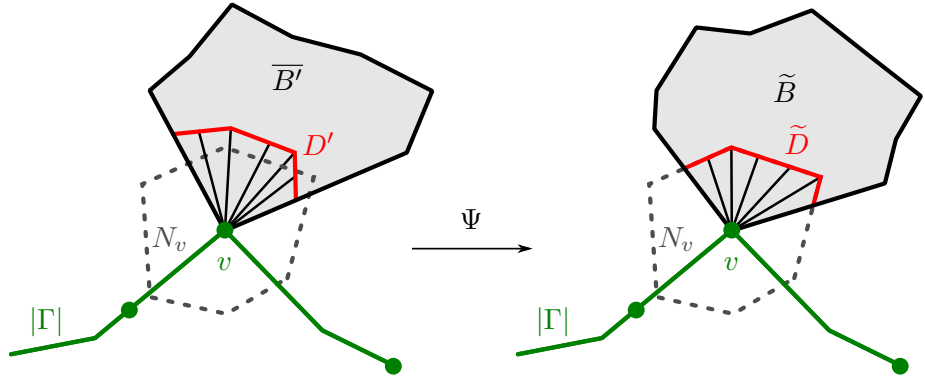
	  \caption{The $3$-ball $\overline{B'}$ and its $\Psi$-image~$\tilde B$. The disc~$D'$ is mapped by~$\Psi$ to a disc~$\tilde D$ in~$\partial N_v$, and $\Psi$~acts on~$vD'$ as the cone of this map.}
	  \label{fig.enclosingsphpsi}
	\end{figure}
	 
	 Assume for the moment that this claim holds, and let us see how to use the resulting~$\Psi$ to construct the desired extension~$\Phi_\SPH$ of~$\Phi_{\overline B}$.
	 
	 Let $\tilde B$~be the $3$-ball~$\Psi(\overline{B'})$, and choose a regular neighborhood~$N_\Gamma$ of~$|\Gamma| \cap \overline{N_v}$ in~$\overline{N_v}$, small enough to be disjoint from $\overline B$~and~$\tilde B$. Moreover, denote by~$M$ the closure in~$\overline{N_v}$ of~$\overline{N_v} \setminus N_\Gamma$, and consider the closed codimension-0 submanifold $N:= \partial N_v \cap M$ of~$\partial M$. 
	 
	 By construction of~$\Psi$, its restriction to~$\Phi_{\overline B}(\overline{B}\cap \overline {N_v})$ is a PL homeomorphism of pairs~$(\Phi_{\overline B}(\overline B \cap \overline{N_v}), D') \to (\tilde B \cap \overline{N_v}, \tilde D)$. We may thus apply the Disc Theorem at the boundary (Corollary~\ref{cor.boundarydisc}) to the inclusion $(\overline B \cap \overline{N_v},D) \into (M,N)$  and the composition
	 \[(\overline B \cap \overline{N_v}, D)
	\overset{\Phi_{\overline B}}{\longrightarrow} (\Phi_{\overline B}(\overline B \cap \overline{N_v}), D')
	\overset{\Psi}{\longrightarrow} (\tilde B \cap \overline N_v, \tilde D)
	\into (M,N),\]
	with the maps labeling the arrows appropriately restricted. This is illustrated in Figure~\ref{fig.discthmatboundary}.  

	\begin{figure}[h]
	 \centering
	 \def \svgwidth{0.95\linewidth}
	 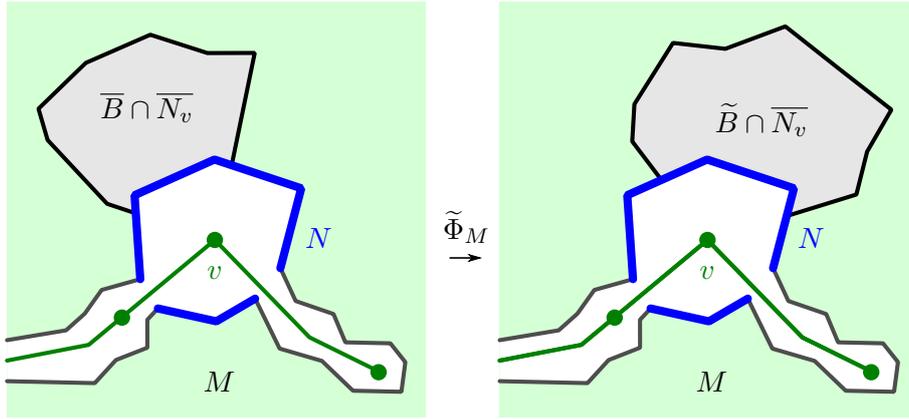
	 \caption{Applying the Disc Theorem at the boundary in order to ambiently isotope the pair $(\overline B \cap \overline{N_v}, D)$ onto $(\tilde B \cap \overline{N_v}, \tilde D)$ within $(M, N)$.}
	 \label{fig.discthmatboundary}
	 	\end{figure}

	The final PL homeomorphism $\tilde \Phi_M \colon M \to M$ of the resulting PL isotopy of~$M$ extends
		the composition $\Psi|_{\Phi_{\overline B}(\overline B \cap \overline {N_v})} \circ \Phi_{\overline B}|_{\overline B \cap \overline{N_v}}$ and fixes $\partial M \setminus \intr(N) = \partial M \cap N_\Gamma$. We may thus extend~$\tilde \Phi_M$ to a PL homeomorphism~$\tilde \Phi_{\overline{N_v}} \colon \overline{N_v} \to \overline{N_v}$ by setting it to be the identity on~$N_\Gamma$. In particular, $\tilde \Phi_{\overline{N_v}}$ fixes~$|\Gamma| \cap \overline{N_v}$.
	Finally, extend~$\tilde \Phi_{\overline{N_v}}$ to a PL homeomorphism $\tilde \Phi_\SPH \colon \SPH\to \SPH$ by defining it on~$N_v = v(\partial N_v)$ as the cone of the already prescribed PL homeomorphism $\partial N_v \to \partial N_v$.
	
	The restriction~$\tilde \Phi_\SPH|_{\overline B}$ is now the composition $\Psi|_{\overline B}\circ\Phi_{\overline B}$: indeed, we have already seen that the two maps agree on~$\overline B \cap \overline{N_v}$, and on~$\overline B \cap N_v = vD$ both are defined as the cone of $D \overset{\Phi_{\overline B}} \to D' \overset{\Psi} \to \tilde D$. Moreover, $\tilde \Phi_\SPH$~clearly fixes~$|\Gamma|$. Hence, the map~$\Phi_\SPH:= \Psi^{-1}\circ \tilde \Phi_\SPH$ extends~$\Phi_{\overline B}$ and fixes~$|\Gamma|$, as desired. We are only left to prove the above claim.
		
	 \begin{proof}[Proof of the claim]
	 Most of the work consists of adding enough detail to our pictures that the map can be made explicit. The construction is illustrated in Figure~\ref{fig.definingpsi}.
	 
	 	\begin{figure}[h]
	 	  \centering
	 	  \def \svgwidth{0.95\linewidth}
	 	  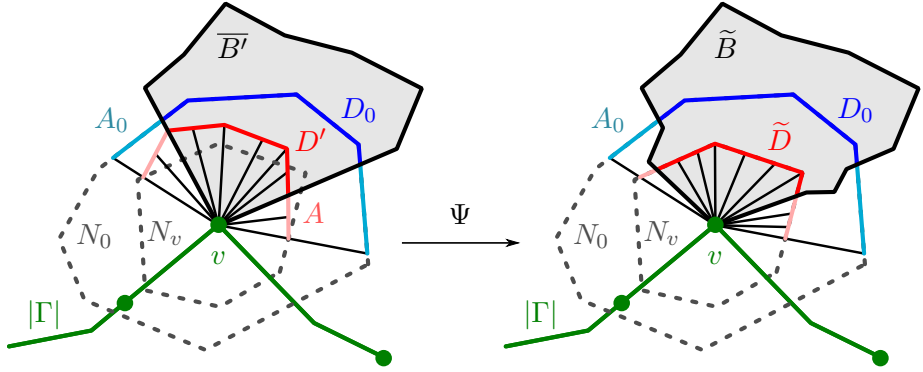
	 	  \caption{The construction of~$\Psi$. Outside the $3$-ball~$vD_0^+$, the map~$\Psi$ is defined as the identity. On~$D'^+$, it is an extension of the identity $\partial D'^+ \to \partial \tilde D^+$ to the interior of the discs, and this in turn is coned to a PL homeomorphism $vD'^+ \to v\tilde D^+$. On the remaining $3$-ball~$C$, we take any extension of the already prescribed map on~$\partial C$.}
	 	  \label{fig.definingpsi}
	 	\end{figure}

	 Choose a collar for~$\partial D_0$ in~$B' \cap \partial N_0$, that is, a PL embedding $c\colon \partial D_0 \times [0,1] \to B' \cap \partial N_0$ such that $c(-, 0)$~is the identity on~$\partial D_0$, and $c(\partial D_0 \times [0,1[)$~is an open neighborhood of~$\partial D_0$ in~$B' \cap \partial N_0$. We may also assume that the image~$A_0$ of~$c$ is disjoint from~$|\Gamma|$. See Rourke-Sanderson for a discussion on collars \cite[p.~24]{RourkeSanderson}.

	 Let $D_0^+$~be the ``enlarged disc''~$D_0 \cup A_0$, and consider the $3$-ball~$vD_0^+$, which is a cone with base~$D_0^+$ and vertex~$v$. We will define~$\Psi$ as the identity on~$\SPH \setminus \intr(v D_0^+)$, and then find a suitable extension of the identity on $\partial (vD_0^+)$ to all of~$vD_0^+$.
	 
	 Denote by~$\tilde D^+$ the disc $vD_0^+ \cap \partial N_v$ and consider the PL circles $\partial D'$~and~$\partial \tilde D^+$, each lying in the cone of a distinct component of~$\partial A_0$. Each of these circles is the base of a cone with vertex~$v$, so we can use Lemma~\ref{lem.annulus} to find an annulus~$A$ with $\partial A = \partial D' \cup \partial \tilde D^+$, such that $vA$~is a cone with base~$A$ and vertex~$v$. We will denote by~$D'^+$ the disc~$D' \cup A$. Notice that by construction, we have $\partial D'^+= \partial \tilde D^+$.
	 
	 To define~$\Psi$ inside $v D_0^+$, we first choose any extension of the identity map $\partial D'^+ \to \partial \tilde D^+$ to a PL homeomorphism $D'^+ \to \tilde D^+$.
	 Since both $vD'^+$~and~$v\tilde D^+$ are cones at~$v$, we can define~$\Psi$ on~$vD'^+$ as the cone of the above PL homeomorphism~$D'^+ \to \tilde D^+$. Note that this is consistent with the definition of~$\Psi$ as the identity on $\partial(vD_0^+)$.
	 
	 It remains only to define $\Psi$ on~$vD_0^+ \setminus vD'^+$, whose closure in~$\SPH$ is a $3$-ball~$C$ (because it is the complement in~$vD_0^+$ of an open regular neighborhood of a boundary point). Writing~$\tilde C$ to denote the closure in~$\SPH$ of~$vD_0^+ \setminus v\tilde D^+$, which is another $3$-ball, this amounts to choosing a PL homeomorphism $C \to \tilde C$ that agrees with the already prescribed map $\partial C \to \partial \tilde C$. We choose any extension, and this completes the definition of~$\Psi$. It is straightforward to verify that all required conditions on~$\Psi$ are satisfied. \phantom{\qedhere}
	 \end{proof}
	 
	 With the claim established, the proof is complete.
\end{proof}

We now collect the dividends from our work proving Proposition~\ref{prop.penclosingball}.

\begin{prop}[Well-definedness of the  vertex sum]
  Any two vertex sums of pointed graphs $(\Gamma_1,v_1),(\Gamma_2, v_2)$ are isomorphic via an isomorphism that induces the identity on~$(\langle \Gamma_1 \rangle \vsum{v_1}{v_2} \langle \Gamma_2 \rangle, v_1 = v_2)$.
\end{prop}

\begin{proof}
The argument can be copied almost word-by-word from the proof of Proposition~\ref{prop.uniquenessdisjunion}, with the role of Lemma~\ref{lem.enclosingball} now of course being played by Proposition~\ref{prop.penclosingball}.
\end{proof}

We can now rest at ease knowing that the ambiguity about enclosing balls and attaching maps in the notation ``$\Gamma_1 \vsum{v_1}{v_2} \Gamma_2$'' is immaterial.

We remark that in the purely combinatorial setting of abstract graphs, we can just as well define the vertex sum along an ordered $k$-tuple of distinct vertices. For spatial graphs, however, we would need a generalization of the notion of an enclosing ball: a $3$-ball containing the support of the spatial graph, and whose boundary intersects it precisely at the $k$~distinguished vertices. But such balls could very well be non-unique in the sense of Proposition~\ref{prop.penclosingball}, as we illustrate in Figure~\ref{fig.knottedball}. Hence our efforts to define a vertex sum of spatial graphs along multiple vertices would necessarily fall short, unless we were willing to also encode the data for the enclosing balls into the operation.

    \begin{figure}[h]
        \centering
        \def \svgwidth{0.6\linewidth}
        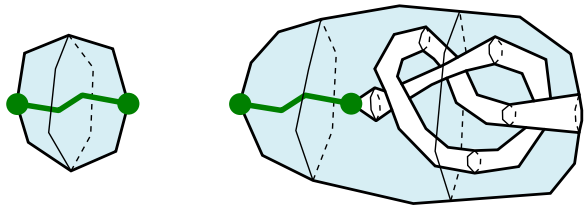
        \caption{An example of non-uniqueness of enclosing balls for spatial graphs with more than one distinguished vertex. We depict a spatial graph comprised of exactly two vertices and one edge connecting them. If both vertices are distinguished, one can find not only the obvious enclosing ball on the left, but also more complicated ones, such as the one on the right.}
        \label{fig.knottedball}
    \end{figure}

There are analogues of Lemmas \ref{lem.summandissubgraph}~and~\ref{lem.isodisjunion} for vertex sums, whose proofs are the same:

\begin{lem}[Vertex summands as sub-graphs]\label{lem.vsummandissubgraph}
Let $\Gamma = \Gamma_1 \vsum{v_1}{v_2} \Gamma_2$ be a vertex sum of pointed spatial graphs, and denote by~$\Gamma_1'$ the sub-graph of~$\Gamma$ obtained by discarding all vertices and edges that are not in~$\Gamma_1$. Then $\Gamma_1' = \Gamma_1$.
\end{lem}

We take a brief moment to note how we have slightly extended our ongoing abuse of notation when writing ``$\Gamma_1' = \Gamma_1$''. Implicit in this statement is an equality between the vertex~$v_1$ of~$\Gamma_1$ and the vertex of $\Gamma_1 \vsum{v_1}{v_2} \Gamma_2$ obtained from the identification $v_1 \sim v_2$. We will allow ourselves to make such abuses in several harmless situations.

\begin{lem}[Vertex sum of isomorphisms]\label{lem.isovsum}
Let $\Phi_1 \colon (\Gamma_1, v_1) \to (\Gamma_1', v_1')$ and $\Phi_2 \colon (\Gamma_2, v_2) \to (\Gamma_2', v_2')$ be isomorphisms of pointed spatial graphs. Then there exists an isomorphism
\[\Phi_1 \vsum{v_1}{v_2} \Phi_2 \colon (\Gamma_1 \vsum{v_1}{v_2} \Gamma_2, v_1=v_2) \to (\Gamma_1' \vsum{v_1'}{v_2'}\Gamma_2', v_1' = v_2')\]
such that for each $i \in \{1,2\}$ the underlying isomorphism $\langle \Phi_1 \vsum{v_1}{v_2} \Phi_2\rangle$ restricts to~$\langle \Phi_i \rangle$ on~$\langle\Gamma_i\rangle$.
\end{lem}

We also collect the following properties of the vertex sum, analogous to the ones given in Proposition~\ref{prop.disjunprop} for the disjoint union.

\begin{prop}[Properties of the vertex sum]\label{prop.propvsum}
  Let $(\Gamma_1,v_1), (\Gamma_2,v_2), (\Gamma_3,v_3)$ be pointed spatial graphs. Then:
  \begin{itemize}
  \item identity element: $(\Gamma_1 \vsum{v_1}{} \one, v_1) = (\Gamma_1, v_1)$,
  \item commutativity: $(\Gamma_1 \vsum{v_1}{v_2} \Gamma_2, v_1=v_2) = (\Gamma_2 \vsum{v_2}{v_1} \Gamma_1, v_1 = v_2)$,
  \item associativity: for $v_{21}$~and~$v_{23}$ (not necessarily distinct) vertices of~$\Gamma_2$, we have
  \[(\Gamma_1 \vsum{v_1}{v_{21}} \Gamma_2) \vsum{v_{23}}{v_3} \Gamma_3 = \Gamma_1 \vsum{v_1}{v_{21}} (\Gamma_2 \vsum{v_{23}}{v_3}\Gamma_3).\]
  \end{itemize}  
\end{prop}
\begin{proof}
The proofs are identical to their counterparts in Proposition~\ref{prop.disjunprop}, save for the following obvious modifications:
\begin{itemize}
\item The first item relies on vertex summands being sub-graphs (Lemma~\ref{lem.vsummandissubgraph}), rather than disjoint union summands being sub-graphs (Lemma~\ref{lem.summandissubgraph}).
\item For proving associativity in the case $v_{21} = v_{23}$, the requirement on the enclosing balls is that $\intr(B_1)\cup \intr(B_2) = \SPH \setminus\{v_{21}\}$ (equivalently, $\overline{B_{21}} \cap \overline{B_{23}} = \{v_{21}\}$).\qedhere
\end{itemize}
\end{proof}

\subsection{Iterated vertex sums and trees of spatial graphs}

At this juncture, we would like to make a claim about how the basic algebraic properties in Proposition~\ref{prop.propvsum} allow us to write down iterated vertex sums without keeping track of the order in which they are performed (as in the last paragraph of Subsection~\ref{sec.defdisjun}). We will indeed establish such a statement, but since we wish to allow for several vertices in each summand to be used (as in the associative property), we need to introduce terminology that allows us to package the more involved combinatorics.

Before doing that, however, let us spend a moment on the comparatively easy case where only one vertex of each summand is used. We denote by $(\bigstar_{i \in I} (\Gamma_i, v_i), v)$ the vertex sum of a collection of pointed spatial graphs~$(\Gamma_i, v_i)$ indexed by a finite set~$I$ (with $\one$~being the vertex sum over the empty set). If $V_i$ is the vertex set of~$\Gamma_i$, then the vertex set of $\bigstar_{i \in I} (\Gamma_i, v_i)$ is $\left( \bigsqcup_{i \in I} V_i \right) / \sim$, where $v_i \sim v_{i'}$ for all $i, i' \in I$. The distinguished vertex~$v$ of this vertex sum is the one obtained from identifying all the~$v_i$. Here it is clear from commutativity and the ``$v_{21} = v_{23}$'' case of associativity in Proposition~\ref{prop.propvsum} that the omission of parentheses or an ordering of~$I$ is immaterial -- all choices yield pointed spatial graphs that are isomorphic via maps that induce the identity on the underlying pointed graph $(\bigstar_{i \in I}(\langle \Gamma_i \rangle, v_i), v)$.

To formalize iterated vertex sums where the vertices along which to sum are allowed to vary, we use the following notion:

\begin{dfn}\label{dfn.tree}
A \textbf{tree of spatial graphs} is a tuple $(T, I, J, L, (\Gamma_i)_{i \in I}, (v(l))_{l \in L})$, where:
\begin{itemize}
\item $T$ is an abstract finite tree with vertex set $I \sqcup J$ and edge set~$L$. 
\item The partition of the vertex set of~$T$ into $I$~and~$J$ is a bipartition of~$T$, that is, each edge~$l \in L$ has one of its endpoints in~$I$, and the other in~$J$. We will write $i(l), j(l)$, respectively, to denote the endpoints of~$l$ in $I$ and~$J$.
\item Each vertex in~$J$ is adjacent to at least two edges of~$T$. Equivalently, all degree-one vertices of~$T$ are in~$I$, and $T$~is not comprised of only one vertex in~$J$.
\item The $\Gamma_i$ are spatial graphs indexed by elements of $I$.
\item For each $l \in L$, $v(l)$ is a vertex of the spatial graph $\Gamma_{i(l)}$.
\item If two different edges $l, l' \in L$ satisfy $i(l)=i(l')$, then~$v(l)\ne v(l')$.
\end{itemize}
\end{dfn}

One should think of such~$\Tree= (T, I, J, L, (\Gamma_i)_{i \in I}, (v(l))_{l \in L})$ as a blueprint for assembling a spatial graph~$[\Tree]$, called its \textbf{realization}, out of the~$\Gamma_i$ through iterated vertex sums. Roughly, when two distinct edges~$l, l' \in L$ satisfy $j(l)=j(l')$ (and hence $i(l)\ne i(l')$), we understand this as an instruction to perform the vertex sum of~$\Gamma_i, \Gamma_{i'}$ along~$v_l, v_{l'}$. Before making this more precise, we invite the reader to study the example in Figure~\ref{fig.treerealization}.
  
  \begin{figure}[h]
      \centering
      \def \svgwidth{0.95\linewidth}
      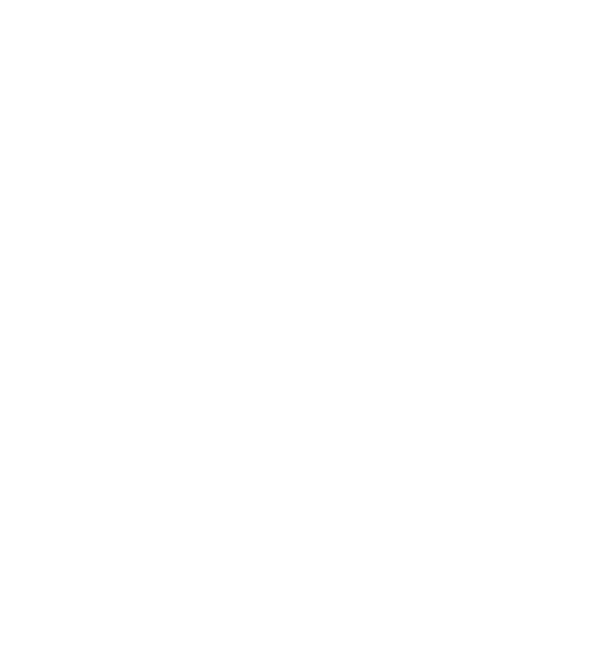
      \caption{From a tree of spatial graphs $\Tree= (T, I, J, L, (\Gamma_i)_{i \in I}, (v(l))_{l \in L})$, we assemble its realization, the \emph{spatial} graph~$[\Tree]$.}
      \label{fig.treerealization}
  \end{figure}

We will use an inductive argument to define, given a tree of spatial graphs~$\Tree$ as above, its realization~$[\Tree]$, and show that the underlying graph~$\langle [\Tree]\rangle$ is what one expects: 
\begin{itemize}
\item the vertex set of~$\langle [\Tree]\rangle$ is $\left(\bigsqcup_{i \in I} V_i \right) / \sim$, where $V_i$~is the vertex set of~$\Gamma_i$ and $v(l)\sim v(l')$ whenever $j(l) = j(l')$,
\item the edge set of~$\langle [\Tree]\rangle$ is~$\bigsqcup_{i \in I}E_i$, where $E_i$~is the edge set of~$\Gamma_i$, each edge being incident to the one or two vertices it contains.
\end{itemize} 

We induct on the cardinality of~$J$. If $J = \emptyset$, then either $T$~is the empty tree, in which case we set $[\Tree] := \zero$, or $T$~has a single vertex~$i_0 \in I$, in which case $[\Tree] := \Gamma_{i_0}$. Either way, $\langle [\Tree]\rangle$~is as claimed.

For the inductive step, we first introduce the following notation: for each edge~$l \in L$, the sub-graph of~$T$ obtained by removing~$l$ has precisely two connected components, each containing one endpoint of~$l$. We will designate by~$T_l$ the component that contains $i(l)$. Moreover, we denote by $I_l, J_l, L_l$, respectively, the subsets of $I, J, L$ comprised of vertices/edges in~$T_l$. This allows us to define a new tree of spatial graphs $\Tree_l := (T_l, I_l, J_l, L_l, (\Gamma_i)_{i \in I_l}, (v(l'))_{l' \in L_l})$.

Now, if $J$~contains at least one vertex $j_0$ (whose choice we will soon show to be immaterial) let $L_0 \subseteq L$ be the set of edges incident to~$j_0$. For each $l \in L_0$, note that $J_l$~has strictly fewer elements than~$J$. Hence we have by induction constructed realizations~$[\Tree_l]$, whose underlying graphs~$\langle [\Tree_l]\rangle$ are as described above. In particular, $\langle[\Tree_l]\rangle$~has $v(l)$~as a vertex, and hence so does~$[\Tree_l]$. We define
\[ [\Tree] := \underset {l \in L_0} \bigstar ([\Tree_l], v(l)), \]
and call it a realization of~$\Tree$.

Showing that $\langle[\Tree]\rangle$~is as claimed is now a matter of bookkeeping. One way of seeing it is to compare the claimed description of~$\langle[\Tree]\rangle$ above with $\bigsqcup_{l \in L_0} \langle [ \Tree_l] \rangle$: these graphs differ only in that the vertices~$v(l)$ with $l \in L_0$ are identified in the former, but not in the latter. This identification is exactly what one obtains from the vertex sum $\bigstar_{l \in L_0} (\langle[\Tree_l]\rangle, v(l))$.

This finishes \emph{an} inductive construction of the realization~$[\Tree]$ with $\langle [\Tree]\rangle$ independent of choices. Next we show that $[\Tree]$~itself is independent of the choice of vertex~$j_0$.

\begin{lem}[Well-definedness of the realization]\label{lem.welldefrealization}
Let $\Tree = (T, I, J, L, (\Gamma_i)_{i \in I}, (v(l))_{l \in L})$ be a tree of spatial graphs. Then any two realizations of~$\Tree$ are isomorphic via an isomorphism that induces the identity on~$\langle [ \Tree ] \rangle$.
\end{lem}

\begin{proof}
We again proceed by induction on the cardinality of~$J$. When $J$~has at most one element, no choices are made in defining~$[\Tree]$, so there is nothing to show.

Suppose then that $J$~contains two elements~$j_1 \neq j_2$. For each $k \in \{1,2\}$, denote by~$[\Tree ]_k$ the realization of~$\Tree$ constructed by splitting~$T$ at~$j_k$. Moreover, let $L_k \subset L$ be the set of edges incident with~$j_k$, and consider, for each $l \in L_k$, the tree of spatial graphs $\Tree_l := (T_l, I_l, J_l, L_l, (\Gamma_i)_{i \in I_l}, (v(l'))_{l' \in L_l})$ defined as before.

Now, there is exactly one edge $l_1 \in L_1$ such that the tree~$T_{l_1}$ contains the vertex~$j_2$, and one edge~$l_2 \in L_2$ such that $T_{l_2}$~contains~$j_1$. Since the intersection of two sub-trees of a tree is always itself a tree, we see that $\dot T:=T_{l_1} \cap T_{l_2}$ is a tree, and indeed we have a tree of spatial graphs
\[ \dot{\Tree} := (\dot T, \dot I , \dot J, \dot L, (\Gamma_i)_{i \in \dot I}, (v(l'))_{l' \in \dot L}), \]
where $\dot I := I_{l_1} \cap I_{l_2}, \dot J := J_{l_1} \cap J_{l_2}$, and $\dot L := L_{l_1} \cap L_{l_2}$. The tree~$\dot T$ is illustrated in Figure~\ref{fig.middletree}.

  \begin{figure}[h]
      \centering
      \def \svgwidth{0.7\linewidth}
      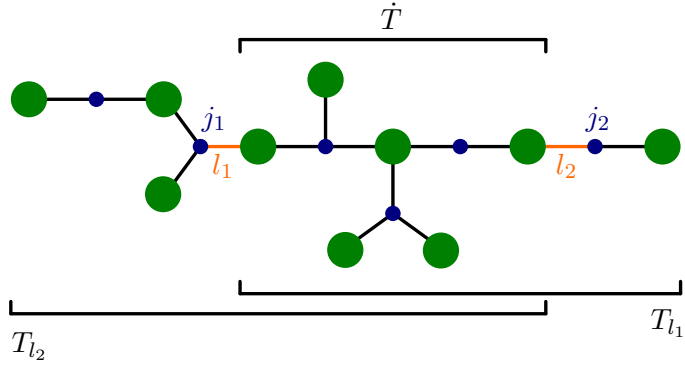
      \caption{The setup in the proof of Lemma~\ref{lem.welldefrealization}. We illustrate an example tree~$T$ together with the relevant sub-trees $T_{l_1}, T_{l_2}$, and their intersection~$\dot T$. Larger vertices represent elements of~$I$, and smaller ones elements of~$J$.}
      \label{fig.middletree}
  \end{figure}

By inductive hypothesis, for each $k \in \{1, 2\}$ the realization~$[\Tree_{l_k}]$ is well-defined. One then easily checks that
\begin{align*}
[\Tree_{l_1}] &= [\dot{\Tree}] \vsum{v(l_2)}{v_2} \underset {l \in L_2 \setminus \{l_2\}} \bigstar ([ \Tree_l],v(l)),\\
[\Tree_{l_2}] &= [\dot{\Tree}] \vsum{v(l_1)}{v_1} \underset {l \in L_1 \setminus \{l_1\}} \bigstar ([ \Tree_l],v(l)),
\end{align*}
where $v_k$~denotes the result of identifying the vertices~$v(l)$ with $l \in L_k \setminus \{l_k\}$.

Having established all the notation, we are ready to wrap up the proof and see that it boils down to an application of the ``$v_{21} \neq v_{23}$'' case of the associative property in Proposition~\ref{prop.propvsum}:
\begin{align*}
[\Tree]_1 & = [\Tree _{l_1}] \vsum{v(l_1)}{v_1}  \bigg(\underset {l \in L_1 \setminus \{l_1\}} \bigstar ([ \Tree_l],v(l))\bigg)\\
& = \left( \bigg(\underset {l \in L_2 \setminus \{l_2\}} \bigstar ([ \Tree_l],v(l))\bigg)  \vsum{v_2}{v(l_2)} [\dot{\Tree}] \right) \vsum{v(l_1)}{v_1}  \bigg( \underset {l \in L_1 \setminus \{l_1\}} \bigstar ([ \Tree_l],v(l))\bigg)\\
& =  \bigg(\underset {l \in L_2 \setminus \{l_2\}} \bigstar ([ \Tree_l],v(l))\bigg)  \vsum{v_2}{v(l_2)} \left( [\dot{\Tree}]  \vsum{v(l_1)}{v_1} \bigg( \underset {l \in L_1 \setminus \{l_1\}} \bigstar ([ \Tree_l],v(l)) \bigg) \right)\\
& =  \bigg(\underset {l \in L_2 \setminus \{l_2\}} \bigstar ([ \Tree_l],v(l))\bigg)  \vsum{v_2}{v(l_2)} [\Tree_{l_2}]\\
& = [\Tree]_2. \qedhere
\end{align*}
\end{proof}

Since realizations of trees are constructed by iterated vertex sums, the observations in Lemmas \ref{lem.vsummandissubgraph}~and~\ref{lem.isovsum} have the following straightforward generalizations.

\begin{lem}[Sub-graphs of the realization of a tree of spatial graphs]\label{lem.treesubgraph}
Let $\Tree = (T, I, J, L, (\Gamma_i)_{i \in I}, (v(l))_{l \in L})$~be a tree of spatial graphs, and for each $i \in I$, let~$\Gamma_i'$ be the sub-graph of~$[\Tree]$ comprised of the vertices and edges of~$\Gamma_i$. Then $\Gamma_i' = \Gamma_i$.
\end{lem}

\begin{lem}[Trees of isomorphisms]\label{lem.treeofisos}
For each~$k \in \{1,2\}$, let $\Tree_k=(T_k, I_k, J_k,  L_k, (\Gamma_i)_{i \in I_k}, (v(l))_{l \in L_k})$ be trees of spatial graphs. Fix also the data of:
\begin{itemize}
\item an isomorphism of trees $f\colon T_1 \to T_2$ such that $f(I_1) = I_2$ (hence also $f(J_1) = J_2$),
\item for each $i \in I_1$, an isomorphism of spatial graphs $\Phi_i \colon \Gamma_i \to \Gamma_{f(i)}$, such that the collection $(\Phi_i)_{i \in I_1}$ respects the assignments~$l \mapsto v(l)$ on $L_1$~and~$L_2$, that is, for every $l \in L_1$, we have $\Phi_{i(l)}(v(l)) = v(f(l))$.
\end{itemize}
Then there is an isomorphism $\Phi \colon [\Tree_1]\to [\Tree_2]$ such that for every~$i \in I_1$, the underlying isomorphism~$\langle \Phi \rangle$ restricts to $\langle \Phi_i \rangle$ on the sub-graph~$\langle \Gamma_i\rangle$ of~$\langle [\Tree_1]\rangle$.
\end{lem}

\subsection{Decomposing pieces as trees of blocks}

As a first step towards establishing the existence of a canonical expression of a piece as the realization of a tree of spatial graphs, we give an analogue of Lemma~\ref{lem.spheresdontlie}, whose proof is essentially the same:

\begin{lem}[If it looks like a vertex sum, it is a vertex sum]\label{lem.vspheresdontlie}
Let $\Gamma$~be a spatial graph in~$\SPH$ and $S\subset \SPH$ a PL-embedded $2$-sphere that intersects~$|\Gamma|$ precisely at one vertex~$v$ of~$\Gamma$. Denote the closures of the two components of~$\SPH \setminus S$ by $B_1$~and~$B_2$. For each $i \in \{1,2\}$, let $\Gamma_i$ be the sub-graph of~$\Gamma$ comprised of the vertices and edges that are contained in~$B_i$. Then $\Gamma = \Gamma_1 \vsum{v}{v} \Gamma_2$.
\end{lem}

We draw the reader's attention to the fact that from now on several statements will include a non-separability assumption on spatial graphs. Incidentally, we collect the following observation:

\begin{lem}[Vertex sum preserves non-separability]\label{lem.sepvsum}
Let $(\Gamma_1, v_1), (\Gamma_2, v_2)$ be pointed spatial graphs. Then $\Gamma_1 \vsum{v_1}{v_2} \Gamma_2$ is non-separable if and only if both $\Gamma_1, \Gamma_2$ are non-separable.
\end{lem}

\begin{proof}
If one of the vertex summands, say~$\Gamma_1$, is separable, denote by~$\SPH_1$ its ambient sphere and let $S\subset \SPH_1$~be a separating sphere. Choose an enclosing ball~$B_1$ for~$(\Gamma_1, v_1)$ that contains~$S$ in its interior. Then if we use~$B_1$ to form the vertex sum, $S$~will be contained in the ambient sphere~$\SPH$ of~$\Gamma_1 \vsum{v_1}{v_2} \Gamma_2$, with both sides of~$S$ intersecting~$|\Gamma_1 \vsum{v_1}{v_2} \Gamma_2|$. Hence $S$~will be a separating sphere for~$\Gamma_1 \vsum{v_1}{v_2} \Gamma_2$.

Conversely, suppose $S$~is a separating sphere for~$\Gamma_1 \vsum{v_1}{v_2} \Gamma_2$. The component of~$\SPH \setminus S$ that does not contain the vertex $v_1= v_2$ has non-empty intersection with the support of one of the summands, say~$|\Gamma_1|$. Then both components of~$\SPH \setminus S$ intersect~$|\Gamma_1|$ and so, regarding $\Gamma_1$~as a sub-graph of~$\Gamma_1 \vsum{v_1}{v_2} \Gamma_2$, we see $S$ is a separating sphere for~$\Gamma_1$.
\end{proof}

Since realizations of trees of spatial graphs are constructed by iterated vertex sums, we have, more generally:

\begin{cor}[Realizations of trees of spatial graphs preserve non-separability.]\label{cor.septrees}
Let $\Tree = (T,I, J, L, (\Gamma_i)_{i \in I}, (v(l))_{l \in L})$ be a tree of spatial graphs. Then $[\Tree]$ is non-separable if and only if all the $\Gamma_i$ are non-separable.
\end{cor}

\begin{dfn} \label{dfn.cutthings}
	Let $\Gamma$ be a spatial graph in~$\SPH$.
  \begin{itemize}
  	\item If $S\subset \SPH$ is a PL-embedded $2$-sphere as in Lemma~\ref{lem.vspheresdontlie}, we say that ``$S$ decomposes $\Gamma$ as $\Gamma_1 \vsum{v}{v} \Gamma_2$''.
  	   
    \item Suppose $\Gamma$~is non-separable. If $S$~is a $2$-sphere decomposing~$\Gamma$ as $\Gamma_1 \vsum{v}{v} \Gamma_2$, with $\Gamma_1, \Gamma_2$~both non-isomorphic to~$\one$, then $v$~is called a \textbf{cut vertex} of~$\Gamma$ and $S$~is a \textbf{cut sphere} of~$\Gamma$.
    
    \item $\Gamma$ is called a \textbf{block} if it is a piece that has no cut vertices and is not isomorphic to~$\one$.
    
    \item A tree of spatial graphs $\Tree = (T, I, J, L, (\Lambda_i)_{i \in I}, (v(l))_{l \in L})$ where each~$\Lambda_i$ is a block is called a \textbf{tree of blocks}. In that case we also say that $\Tree$~is a tree of blocks for~$[\Tree]$.
  \end{itemize}
\end{dfn}

There is also a standard notion of cut vertex for a connected abstract graph~$G$ that is similar in spirit: a vertex~$v$ in an abstract graph~$G$ is \textbf{cut} if $G$~is the union of two sub-graphs~$G_1, G_2$ intersecting precisely at~$v$, with neither~$G_i$ comprised only of a single vertex. We should however note that it is possible for a vertex of a spatial graph~$\Gamma$ to be cut in~$\langle \Gamma \rangle$ but not in~$\Gamma$, as exemplified in Figure~\ref{fig.linkedhandcuffs}.

\begin{figure}[h]
  \centering
  \def \svgwidth{0.2\linewidth}
  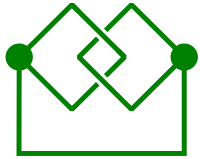
  \caption{The notion of cut vertex for spatial graphs does not coincide with the one for abstract graphs. Both vertices of the depicted graph~$\Gamma$ are cut in~$\langle \Gamma\rangle$, but not in~$\Gamma$.}
  \label{fig.linkedhandcuffs}
\end{figure}

One of the reasons one should care about having a spatial graph expressed as the realization of a tree of blocks is because cut vertices are easily read-off. In order to make this connection precise, we need an analogue of Lemma~\ref{lem.sortedpieces}:

\begin{lem}[Spheres sort blocks]\label{lem.sortedblocks}
Let~$\Lambda$ be a block in~$\SPH$, and let $S\subset \SPH$ be a PL-embedded $2$-sphere that intersects~$|\Lambda|$ either at a single vertex of~$\Lambda$, or not at all. Denote the closures of the two components of~$\SPH \setminus S$ by $B_1, B_2$. Then $|\Lambda|$ is contained in exactly one of the~$B_i$.
\end{lem}

\begin{proof}
Since~$\Lambda$ is a piece, the case where $S \cap |\Lambda| = \emptyset$ follows from Lemma~\ref{lem.sortedpieces}. 

If $S \cap |\Lambda|$ is comprised precisely of one vertex~$v$ of~$\Lambda$, then since $\Lambda \not \cong \one$, certainly $\Lambda$ cannot be contained in both~$B_i$. By Lemma~\ref{lem.vspheresdontlie}, $S$ decomposes $\Lambda$ as~$\Lambda_1 \vsum{v}{v} \Lambda_2$. Since $\Lambda$~has no cut vertices, one of the summands, say $\Gamma_1$, is isomorphic to~$\one$. This means $\Lambda = \Lambda_2$.
\end{proof}

\begin{prop}[Cut vertices in the realization of a tree of blocks]\label{prop.cutvertices}
Let $\Tree = (T, I, J, L, (\Lambda_i)_{i \in I}, (v(l))_{l \in L})$ be a tree of blocks. For each $j \in J$, denote by $v(j)$ the vertex of~$[\Tree]$ that results from identifying all~$v(l)$ with $l$~incident to~$j$. Then the correspondence $j \mapsto v(j)$ is a bijection between~$J$ and the set of cut vertices of~$[\Tree]$.
\end{prop}

\begin{proof}
To see that each~$v(j)$ is cut: by definition of a tree of spatial graphs, $j$~has degree at least~$2$, so if $L_0 \subseteq L$~is the set of edges incident to~$j$, one may write some non-trivial partition~$L_0 = L_1 \sqcup L_2$. For each $k\in\{1,2\}$, choose some $l_k \in L_k$. The spatial graph~$\Lambda_{i(l_k)}$, being a block, has an edge, and hence also~$[\Tree_{l_k}]$ has an edge. Thus each vertex summand in
\[[\Tree] = \biggl( \underset {l \in L_1}\bigstar([\Tree_l], v(l)) \biggr) \vsum {v(l_1)}{v(l_2)} \biggl( \underset {l \in L_2} \bigstar([\Tree_l], v(l)) \biggr),\]
has an edge and so is not isomorphic to~$\one$. Thus $v(j) = v(l_1)=v(l_2)$ is cut.

It is clear from the vertex set of~$[\Tree]$, as given by the description of~$\langle[\Tree]\rangle$, that the assignment $j \mapsto v(j)$ is injective.

Conversely, suppose $v$~is a vertex of $[\Tree]$ that does not result from such an identification, and consider a PL-embedded $2$-sphere~$S$ in the ambient sphere~$\SPH$ of~$[\Tree]$ intersecting~$|[\Tree]|$ precisely at~$v$. Say $S$~decomposes $[\Tree]$ as $\Gamma_1 \vsum{v}{v} \Gamma_2$ -- we aim  to show that one of the~$\Gamma_i$ is isomorphic to~$\one$. Our assumption on~$v$ implies that the edges of~$[\Tree]$ incident to~$v$ all come from the same block~$\Lambda_i$. Using Lemma~\ref{lem.treesubgraph} to regard~$\Lambda_i$ as a sub-graph of~$[\Tree]$, we see from Lemma~\ref{lem.sortedblocks} that all edges of~$[\Tree]$ incident to~$v$ are in one of the~$\Gamma_i$, say in~$\Gamma_1$. Hence $v$~is a vertex of~$\Gamma_2$ without incident edges. All we need is to show that $\Gamma_2$ is non-separable, and it will follow that $\Gamma_2 \cong \one$. But since $[\Tree]$~is non-separable by Corollary~\ref{cor.septrees}, non-separability of~$\Gamma_2$ follows from Lemma~\ref{lem.sepvsum}.
\end{proof}

Having at least somewhat motivated the usefulness of trees of blocks, we now establish their existence for non-separable graphs (with the exception of~$\one$).

\begin{prop}[Existence of trees of blocks]\label{prop.treeexists}
Every non-separable spatial graph~$\Gamma \not \cong \one$ is the realization of some tree of blocks.
\end{prop}

\begin{proof}
We use induction on the number of edges in~$\Gamma$ to produce a tree of blocks $\Tree = (T, I, J, L, (\Lambda_i)_{i \in I}, (v(l))_{l \in L})$ realizing~$\Gamma$. If $\Gamma$~has no edges, then since~$\Gamma$ is non-separable we either have $\Gamma = \zero$ or $\Gamma \cong \one$. The second case is ruled out by assumption, and in the first one we take $T$~to be the empty tree.

Now suppose~$\Gamma$ has at least one edge. If $\Gamma$~is a block, then we are done by taking $T$~to be a tree with a single vertex $i_0 \in I$ and $\Lambda_{i_0} := \Gamma$. If $\Gamma$~is not a block, then it can be expressed as $\Gamma_1 \vsum{v}{v} \Gamma_2$ with each~$\Gamma_i$ not isomorphic to~$\one$, and also non-separable by Lemma~\ref{lem.sepvsum}. Hence each of~$\Gamma_1, \Gamma_2$ has at least one edge, and thus both have fewer edges than~$\Gamma$ and our induction hypothesis applies to them.

For each~$k\in \{1,2\}$, let $\Tree_k = (T_k, I_k, J_k, L_k, (\Lambda_i)_{i \in I_k}, (v(l))_{l \in L_k})$ be a tree of blocks for~$\Gamma_k$. We will construct~$T$ as illustrated in Figure~\ref{fig.findtree} from modified versions~$T_k'$ of the~$T_k$, according to the following two cases:
\begin{itemize}
\item If $v$~is not a cut vertex of~$\Gamma_k$, so by Proposition~\ref{prop.cutvertices} there is no edge~$l_k \in L_k$ with $v(l_k) = v$, construct a new tree~$T_k'$ from $T_k$ by adding a new vertex~$j_k$ and a new edge~$l_k$ connecting~$j_k$ to the vertex~$i_k \in I$ such that $\Lambda_{i_k}$ contains~$v$. We also write 
\[J_k':= J_k \sqcup \{j_k\},\quad  L'_k:= L_k \sqcup \{ l_k \}, \quad L^0_k:=\{l_k\},\]
and set~$v(l_k):= v$. Note that in this case, $T_k'$~with its vertex set partitioned as $I_k \sqcup J_k'$ is no longer admissible as the tree in a tree of spatial graphs, since $j_k$~is a leaf. 

\item If $v$~is a cut vertex of~$\Gamma_k$, then there is a corresponding $j_k \in J_k$, whose set of incident edges we denote by~$L_k^0$. In $\Gamma_k$, the vertices~$v(l)$ with $l \in L_k^0$ are identified into the vertex~$v$. For convenience, we write~$T_k':= T_k, J_k':= J_k, L_k':=L_k$.
\end{itemize}

\begin{figure}[h]
  \centering
  \def \svgwidth{0.7\linewidth}
  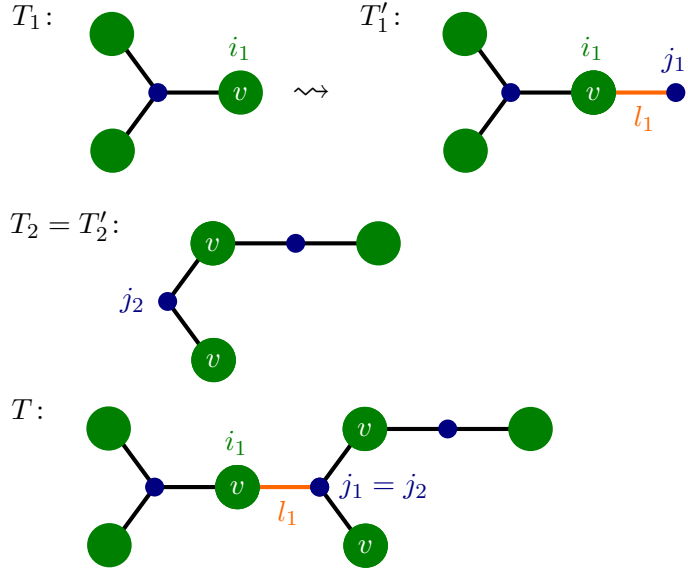
  \caption{Constructing $T$ from $T_1$~and~$T_2$. In the depicted trees, large vertices represent elements of~$I_1, I_2$, and small vertices represent elements of~$J_1, J_2$. We also indicate the elements of~$I_1, I_2$ whose corresponding blocks contain the vertex~$v$. In this example, $v$~is not a cut vertex in~$\Gamma_1$, but it is in~$\Gamma_2$ (where it corresponds to~$j_2$). Accordingly, $T_1'$~is obtained from~$T_1$ by adding a vertex~$j_1$ and an edge~$l_1$, whereas $T_2' = T_2$. $T$~is then obtained from $T_1'$~and~$T_2'$ by identifying $j_1$~with~$j_2$.}
  \label{fig.findtree}
\end{figure}

We define
\begin{align*}
T&:= T_1' \vsum{j_1}{j_2} T_2', & I&:= I_1 \sqcup I_2,\\
J&:= (J_1' \sqcup J_2')/j_1 \sim j_2, & L&:= L_1' \sqcup L_2',
\end{align*}
and this turns $\Tree$~into a tree of spatial graphs whose realization is~$\Gamma$:
\begin{align*}
[\Tree] & = \underset {\substack{l \in L\\ j(l)=j_1=j_2}} \bigstar ([\Tree_l], v(l))\\
& = \bigg( \underset {l \in L_1^0} \bigstar ([\Tree_l], v(l))\bigg) \vsum{v}{v}  \bigg(\underset {l \in L_2^0} \bigstar ([\Tree_l], v(l))\bigg) \\
& = [\Tree_1] \vsum{v}{v} [\Tree_2]\\
& = \Gamma_1 \vsum{v}{v} \Gamma_2. \qedhere
\end{align*}
\end{proof}

We now give a uniqueness statement for the expression of non-separable spatial graphs as the realization of a tree of blocks.

\begin{prop}[Uniqueness of decomposition into a tree of blocks]\label{prop.uniquetree}
For each $k \in \{1,2\}$, let $\Tree_k = (T_k, I_k, J_k, L_k, (\Lambda_i)_{i \in I_k}, (v(l))_{l \in L_k})$ be a tree of blocks, and let $\Phi \colon [\Tree_1] \to [\Tree_2]$ be an isomorphism. Then there is an isomorphism of trees $f \colon T_1 \to T_2$ satisfying $f(I_1) = I_2$, such that:
\begin{itemize}
\item for each $i \in I_1$, the map $\Phi$ is an isomorphism of sub-graphs $\Lambda_i \to \Lambda_{f(i)}$, and
\item for each $l \in L_1$, we have $\Phi(v(l)) = v(f(l))$.
\end{itemize}
\end{prop}

Note that the second item implies that $f$~respects the bijective correspondence given by Proposition~\ref{prop.cutvertices} between the~$J_k$ and the set of cut vertices of~$[\Tree_k]$. In other words, for each $j \in J_1$ we have $\Phi(v(j)) = v(f(j))$.

\begin{proof}
We will proceed by induction on the cardinality of~$I_1$. If $I_1 = \emptyset$, then $[\Tree_1] = \zero = [\Tree_2]$, so $T_2$ is the empty tree and there is nothing to show. If $I_1$~is comprised of a single element~$i_1$, and hence $J_1 = \emptyset$, then $[\Tree_1] =  \Lambda_{i_1}$ is a block, so $[\Tree_2]$~is a block. In particular, $[\Tree_2]$~has no cut vertices and so by Proposition~\ref{prop.cutvertices} we conclude $J_2 = \emptyset$. Hence $I_2$~contains exactly one element~$i_2$, with $[\Tree_2] = \Lambda_{i_2}$. We are thus done with this case by setting $f(i_1):=i_2$.

Assume now that $I_1$ contains at least two elements, and so $J_1 \neq \emptyset$. Choose $j_1 \in J_1$, write $v_1:=v(j_1)$, and let $S_1$~be a cut sphere for~$[\Tree_1]$ decomposing it as $[\Tree_1] = \Gamma_1^+ \vsum{v_1}{v_1} \Gamma_1^-$, so $\Gamma_1^+, \Gamma_1^-$ are pieces non-isomorphic to~$\one$.
Similarly, we have that $v_2:= \Phi(v_1)$ is a cut vertex for~$[\Tree_2]$, so let $j_2 \in J_2$ be the corresponding element. The sphere $S_2 := \Phi(S_1)$ is now a cut sphere for~$[\Tree_2]$ decomposing it as $[\Tree_2] = \Gamma_2^+ \vsum{v_2}{v_2} \Gamma_2^-$, with the map $\Phi$~giving a pair of isomorphisms of sub-graphs $\Phi^\epsilon \colon \Gamma_1^\epsilon \to \Gamma_2^\epsilon$, for each~$\epsilon \in \{ +,-\}$.

Let $k \in \{1,2\}$. Our goal is to extract from each $\Tree_k$ a description of the spatial graphs~$\Gamma_k^+, \Gamma_k^-$ as realizations of trees of blocks, to which we will then apply the induction hypothesis. The procedure is entirely analogous for all four spatial graphs, so let us also fix a sign~$\epsilon \in \{+, -\}$. 

Denote by~$L_k^0 \subseteq L_k$ the set of edges incident to~$j_k$, and recall that Lemma~\ref{lem.sortedblocks} tells us that for each $i \in I_k$, the block~$\Lambda_i$ is a sub-graph of exactly one among~$\Gamma_k^+, \Gamma_k^-$. We consider the partition $L_k^0 = L_k^{0+} \sqcup L_k^{0-}$, where an edge~$l\in L_k^0$ is in $L_k^{0\epsilon}$ if and only if $\Lambda_{i(l)}$~is a sub-graph of~$\Gamma_k ^\epsilon$.

Consider the decomposition of~$[\Tree_k]$ as
\[ [\Tree_k] = \bigg( \underset {l \in L_k^{0+}} \bigstar ([(\Tree_k)_l], v(l))\bigg) \vsum{v_k}{v_k}  \bigg(\underset {l \in L_k^{0-}} \bigstar ([(\Tree_k)_l], v(l))\bigg).\]
We claim that this is the same as the decomposition given by~$S_k$, that is, $\Gamma_k^\epsilon = \bigstar_{l \in L_k^{0\epsilon}} ([(\Tree_k)_l], v(l))$.

To see this, first notice that since all $\Lambda_i$ are non-separable, Corollary~\ref{cor.septrees} tells us that each $[(\Tree_k)_l]$~is non-separable. Now, for each~$l \in L_k^0$, it follows from Proposition~\ref{prop.cutvertices} that $v(l)$~is not a cut vertex of~$[(\Tree_k)_l]$. Therefore, $S_k$~decomposes~$[(\Tree_k)_l]$ as a trivial vertex sum $[(\Tree_k)_l]\vsum{v(l)}{}\one$. In other words, $|[(\Tree_k)_l]|$~is entirely contained in one side of~$S_k$, which must of course be the same as~$|\Lambda_{i(l)}|$, since $\Lambda_{i(l)}$~is a sub-graph of~$|[(\Tree_k)_l]|$. From here, we get that each $\bigstar_{l \in L_k^{0\epsilon}} ([(\Tree_k)_l], v(l))$ is a sub-graph of~$\Gamma_k^\epsilon$, whence the desired equalities follow.

Next, we write down an explicit tree of blocks~$\Tree_k^\epsilon$ for~$\bigstar_{l \in L_k^{0\epsilon}} ([(\Tree_k)_l], v(l))$. If $L_k^{0 \epsilon}$ has only one element~$l_k^\epsilon$, put $\Tree_k^\epsilon := (\Tree_k)_{l_k^\epsilon}$. Otherwise, recover the notation introduced when defining the realization of a tree of spatial graphs
\[(\Tree_k)_l = ((T_k)_l, (I_k)_l, (J_k)_l, (L_k)_l, (\Lambda_i)_{i \in (I_k)_l}, (v(l'))_{l' \in (L_k)_l} ),\]
and set $\Tree_k^\epsilon := (T_k^\epsilon, I_k^\epsilon, J_k^\epsilon, L_k^\epsilon, (\Lambda_i)_{i \in I_k^\epsilon},(v(l))_{l \in L_k^\epsilon})$ to be the tree of blocks comprised of the branches of $\Tree_k$ at~$j_k$ that stem from edges in~$L_k^{0\epsilon}$. Explicitly, $T_k^\epsilon$~is the sub-tree of~$T_k$ with vertex and edge sets given by
\[I_k^\epsilon := \bigsqcup_{l \in L_k^{0\epsilon}} (I_k)_l, \quad  J_k^\epsilon:= \{ j_k \} \sqcup \bigsqcup_{l \in L_k^{0\epsilon}} (J_k)_l, \quad L_k^\epsilon:= {L_k^{0\epsilon}} \sqcup  \bigsqcup_{l \in L_k^{0\epsilon}} (J_k)_l.\]

Observe that in the first case $[\Tree_k^\epsilon]$~does not have $v_k$~as a cut vertex, and in the second case it does, with $j_k$~being the corresponding element of~$J_k^\epsilon$.

It is now clear that, in either case, $[\Tree_k^\epsilon] = \bigstar_{l \in L_k^{0\epsilon}} ([(\Tree_k)_l], v(l)) = \Gamma_i^\epsilon$. By induction hypothesis, the spatial graph isomorphisms~$\Phi^\epsilon \colon [\Tree_1^\epsilon] \to [\Tree_2^\epsilon]$ yield tree isomorphisms $f^\epsilon \colon T_1^\epsilon \to T_2 ^\epsilon$, which we now want to assemble to the desired $f \colon T_1 \to T_2$. On each sub-tree $T_1^\epsilon$ of~$T_1$, we want to set $f = f^\epsilon$, but we have to ensure that $f^+$~and~$f^-$ agree where they overlap, and we must also define~$f$ on the vertices and edges of~$T_1$ that are not in one of the~$T_1^\epsilon$.

Fix~$\epsilon \in \{+,-\}$ for this paragraph. The isomorphism~$\Phi^\epsilon$ ensures that $v_1$~is a cut vertex of~$[\Tree_1^\epsilon]$ if and only if $v_2$~is a cut vertex of~$[\Tree_2^\epsilon]$. If this is the case, then for both~$k \in \{1,2\}$, the vertex~$j_k$ of~$T_k$ is in~$T_k^\epsilon$, along with the edges in~$L_k^{0 \epsilon}$. Moreover, in this situation we have in $\Tree_2^\epsilon$ that~$v(j_2) = v_2 = \Phi^\epsilon (v_1) = \Phi^\epsilon (v(j_1)) = v(f^\epsilon(j_1))$, whence by injectivity of $j \mapsto v(j)$ it follows that $j_2 = f^\epsilon(j_1)$.
On the other hand, if one (hence both)~$v_k$ is not cut in~$[\Tree_k^\epsilon]$, then the corresponding~$j_k$ and the unique edge~$l_k^\epsilon$ in~$L_k^{0\epsilon}$ are not in~$T_k^\epsilon$. In this situation, $i(l_k^\epsilon)$~is the only element of $I_k^\epsilon$ whose corresponding block~$\Lambda_{i(l_k^\epsilon)}$ contains~$v_k$ as a vertex.

We now consider the following three cases:
\begin{itemize}
\item If for one (hence both)~$k\in \{1,2\}$ the vertex~$v_k$~is cut in $[\Tree_k^+]$~and~$[\Tree_k^-]$, then the two sub-trees~$T_k^+, T_k^-$ jointly cover all of~$T_k$, and they overlap precisely at the vertex~$j_k$.  As we have seen that $f^+(j_1) = j_2 = f^-(j_1)$, we are allowed to glue together the~$f^\epsilon$ into the desired $f \colon T_1 \to T_2$.

\item Suppose for both $k \in \{1,2\}$, the vertex~$v_k$ is cut in $[\Tree_k^+]$ but not in~$[\Tree_k^-]$ (the reverse situation being analogous). Then $T_k^+$~and~$T_k^-$ do not overlap, and they jointly they cover all of~$T_k$ except for the edge $l_k^-$ described above.
In this case, we extend the definition of $f^+, f^-$ to all of~$T_1$ by setting $f(l_1^-) := l_2^-$. This respects the endpoints of the edge: we have seen that $f^-(j_1) = j_2$, and the characterization of $i(l_k^\epsilon)$ given above, together with the fact that $\Phi^-(v_1)=v_2$, shows that $f^-(i(l_1^-)) = i(l_2^-)$.

\item If for both $k \in \{1,2\}$ the vertex~$v_k$ is not cut in~$[\Tree_k^+]$ nor in~$[\Tree_k^-]$, then the trees $T_k^+, T_k^-$ are disjoint and cover all of~$T_k$ except for the vertex~$j_k$ and its only two incident edges $j_k^+, j_k^-$. We extend $f^+, f^-$ by putting $f(j_1):=j_2$ and, for each $\epsilon \in \{+,-\}$, setting $f(l_1^\epsilon):=l_2^\epsilon$. This clearly respects the incidence of each~$l_1^\epsilon$ at the endpoint~$j_1$, and for the other endpoint we argue exactly as in the previous item.
\end{itemize}

Having defined the isomorphism $f: T_1 \to T_2$, almost all stated properties are directly inherited from the~$f^\epsilon$. We are only left to check that, in the second and third cases above, the definition of~$f$ on the new edge(s)~$l_1^\epsilon$ satisfies $\Phi(v(l_1^\epsilon)) = v(f(l_1^\epsilon))$. And indeed it does: $\Phi(v(l_1^\epsilon)) = \Phi(v_1) = v_2 = v(l_2^\epsilon) = v(f(l_1^\epsilon))$.
\end{proof}

This proposition is meant to work in tandem with the result about uniqueness of decomposition into pieces (Proposition~\ref{prop.uniquepieces}): as explained earlier, Proposition~\ref{prop.uniquepieces} reduces the isomorphism problem for spatial graphs to decomposing them as disjoint unions of pieces, and comparing the pieces. Now given the task of testing whether two pieces are isomorphic, we further decompose each of them as the realization of a tree of blocks (the details of how this is done algorithmically are given later, in Lemma~\ref{lem.findtree}). Then Proposition~\ref{prop.uniquetree} guarantees that two pieces are isomorphic if and only if the blocks in the decompositions are pairwise isomorphic, via isomorphisms respecting the overall combinatorial structure of the tree. Though this might seem like little gain, comparing the isomorphism type of blocks is now within our reach using Matveev's Recognition Theorem (Theorem~\ref{thm.matveev}). Before explaining how to accomplish this, we make a brief detour to discuss a few extensions of our notion of spatial graphs, and the decomposition results presented thus far.

\section{Extension to decorated graphs}\label{sec.decorations}

The theory developed so far can be easily generalized to spatial graphs equipped with additional structure. Three natural extensions are directed spatial graphs, and spatial graphs with colorings of the edges and/or vertices. In this section, we formalize these concepts and comment on how the operations and decomposition results given so far are adapted to these other settings. All main proofs will however carry over with no need for additional insight, so we will omit them.

\begin{dfn}
A \textbf{directed} spatial graph is spatial graph~$\Gamma$ together with a choice of orientation of each edge. If $e$~is a non-loop edge of~$\Gamma$ and $h\colon [-1,1]\to e$ is a PL homeomorphism orienting~$e$, we say the vertex~$h(-1)$ is the \textbf{source} of~$e$, and $h(1)$~is its \textbf{target}. When $e$~is a loop, the only vertex of~$\Gamma$ contained in~$e$ is simultaneously the source and the target. We denote the source and target of an edge~$e$ by $s(e)$~and~$t(e)$, respectively.

An isomorphism of directed spatial graphs is an isomorphism of the underlying spatial graphs such that the induced PL homeomorphisms between the edges are all orientation-preserving.
\end{dfn}

\begin{dfn}Let $\Gamma = (\SPH, V, E)$ be a spatial graph.
A \textbf{vertex coloring} of~$\Gamma$ is a function $f \colon V \to \N$ from the vertex set to the non-negative integers. For each vertex~$v\in V$, we refer to~$f(v)$ as ``the color of~$v$''. Given two spatial graphs $\Gamma_1 = (\SPH_1, V_1, E_1), \Gamma_2 = (\SPH_2, V_2, E_2)$ with vertex colorings~$f_1, f_2$, an isomorphism $\Phi \colon \Gamma_1\to \Gamma_2$ is said to preserve the vertex coloring if the induced bijection of vertices $\Phi|_{V_1} \colon V_1 \to V_2$ satisfies $f_1 = f_2 \circ \Phi|_{V_1}$.

In an entirely similar fashion, we define an \textbf{edge coloring} $g\colon E \to \N$, and what it means for an isomorphism of spatial graphs to preserve the edge coloring.
\end{dfn}

One may consider spatial graphs with any (possibly empty) combination of these three types of structure, and we will broadly refer to such spatial graphs as \textbf{decorated}. By two spatial graphs carrying a decoration ``of the same type'', we mean that the combination of additional structures is the same.

Most of the notions we have discussed until now transfer without difficulty to the setting of decorated graphs. For example, an isomorphism of spatial graphs with decorations of the same type is an isomorphism of the corresponding undecorated spatial graphs that respects all additional pieces of structure. Sub-graphs of decorated spatial graphs inherit a decoration of the same type in an obvious way, and also the underlying abstract graph of a decorated spatial graph inherits a decoration:
\begin{itemize}
	\item If $\Gamma$~is a directed spatial graph, then the source and target functions on the edge set make $\langle \Gamma \rangle$~a directed abstract graph.
	\item Vertex and edge colorings of abstract graphs are defined in exactly the same way as for spatial graphs, and if $\Gamma$~is decorated with a vertex and/or edge coloring, then so is~$\langle \Gamma \rangle$, in an obvious manner.
\end{itemize}

The induced decoration on~$\langle \Gamma \rangle$ determines the decoration of~$\Gamma$, except for one ambiguity: if $\Gamma$~is directed, the orientation of a loop~$e$ cannot be inferred from $s(e)$~and~$t(e)$. We record the following straightforward consequence:
\begin{lem}[Probing compatibility with decorations through underlying graphs]\label{lem.undrlyingdecor}
Let $\Gamma_1, \Gamma_2$~be spatial graphs with decorations of the same type, and let $\Phi \colon \Gamma_1^- \to \Gamma_2^-$ be an isomorphism between the corresponding undecorated spatial graphs. Assume moreover that the~$\Gamma_i$ have no loops, or are not directed. Then $\Phi$~respects the decorations on the~$\Gamma_i$ if and only if $\langle \Phi\rangle \colon \langle \Gamma_1^- \rangle \to \langle \Gamma_2^- \rangle$ respects the decorations on the~$\langle \Gamma_i\rangle$.
\end{lem}

We now summarize the adaptations of the main definitions and statements regarding the operations of disjoint union and vertex sum:

\begin{itemize}
\item The disjoint union of spatial graphs with decorations of the same type is defined as the disjoint union of the underlying spatial graphs, and it carries a decoration of the same type in the obvious way.

\item The vertex sum of two pointed spatial graphs with decorations of the same type is similarly defined provided that, in case a vertex coloring is part of the decoration, the basepoints are of the same color.

\item Isomorphisms between decorated spatial graphs can be assembled along disjoint unions and vertex sums, in the sense of Lemmas \ref{lem.isodisjunion}~and~\ref{lem.isovsum}.

\item In the setting of decorated spatial graphs, there is still a well-defined identity element~$\zero$ for the disjoint union, but if a vertex coloring is part of the decoration, there is one isomorphism type~$\one_c$ of one-point spatial graph for each color $c\in \N$.

\item The properties of disjoint union and vertex sum listed in Propositions \ref{prop.disjunprop}~and~\ref{prop.propvsum} hold for spatial graphs with decorations of the same type. If a vertex coloring is part of the decoration, the occurrence of~$\one$ in the first item of Proposition~\ref{prop.propvsum} should be read as~$\one_c$, where $c$~is the color of~$v_1$.

\item The definitions of separating sphere, separable spatial graph, and piece remain unchanged in the decorated setting (Definition~\ref{dfn.septhings}). Every decorated spatial graph can be expressed as an iterated disjoint union of (decorated) pieces, in a way that is unique in the sense of Proposition~\ref{prop.uniquepieces}.

\item In the definition of a tree of spatial graphs (Definition~\ref{dfn.tree}), all~$\Gamma_i$ should have a decoration of the same type. Moreover, if a vertex coloring is part of the decoration, we require that for each $j \in J$, the vertices~$v(l)$ with $l$~adjacent to~$j$ all be of the same color. Realizations of trees of decorated spatial graphs are then well-defined in the sense of Lemma~\ref{lem.welldefrealization}, carrying a canonical decoration of the same type.

\item In the definitions of cut vertex, cut sphere and block (Definition~\ref{dfn.cutthings}), if we are in the colored vertex setting, occurrences of the expression ``not isomorphic to~$\one$'' should be read as ``not isomorphic to any~$\one_c$''. The definition of a tree of blocks remains unchanged.

\item Propositions \ref{prop.treeexists}~and~\ref{prop.uniquetree} apply to decorated graphs: every non-separable decorated spatial graph that is not isomorphic to a one-point graph is the realization of a tree of decorated blocks in a unique way.
\end{itemize}

We finish this section by remarking that vertex colorings will be used in an essential way for proving our main algorithmic recognition result for pieces (Proposition~\ref{prop.piecerecognition}), even when the pieces do not come with vertex colorings (explicitly, we use vertex colorings in the proof of Lemma~\ref{lem.multipointedblocs}).
\section{The marked exterior of a spatial graph}\label{sec.markedexterior}

In this section we will construct the ``marked exterior'' of a (decorated) spatial graph, which turns out to be a ``manifold with boundary pattern''. We will explain how it encodes the spatial graph used to construct it, and translate indecomposability properties of spatial graphs into properties of their marked exteriors.

\subsection{Construction and faithfulness}\label{sec.defmarkedexterior}
We will first introduce a simpler variant of the marked exterior of a spatial graph $\Gamma$, which we will call the ``oriented marked exterior''. This will be a pair $(X_\Gamma^\circ,P_\Gamma^\circ)$, where $X_\Gamma^\circ$ is an oriented PL 3-manifold and $P_\Gamma^\circ$ is an oriented one-dimensional submanifold of the boundary $\partial X_\Gamma^\circ$. In case the support $|\Gamma|$ is a union of circles, so it can be thought of as a link, the manifold $X_\Gamma^\circ$ will be homeomorphic to the exterior of that link.

Let $\Gamma=(\SPH,V,E)$ be a spatial graph. First, choose a regular neighborhood $N_V$ of $V$ in the pair $(\SPH,|\Gamma|)$.
Since $V$ is a finite disjoint union of points, $N_V$ will be a disjoint union of 3-balls \cite[Corollary~3.12]{RourkeSanderson}, each containing exactly one vertex $v\in V$. We will denote that ball by $N_v$.
Further, we will denote by $X_V$ the compact PL 3-manifold $\SPH\setminus \intr(N_V)$.

For the next step, choose a regular neighborhood $N_E$ of $|\Gamma|\cap X_V$ in $X_V$. As $|\Gamma|\cap X_V$ is the disjoint union of all properly embedded arcs $e\cap X_V$, with $e\in E$, the regular neighborhood $N_e$ of each $e\cap X_V$ is a 3-ball and $(N_e,e\cap X_V)$ is an unknotted ball pair \cite[Corollary~3.27]{RourkeSanderson} (this reference states only that $N_e$ is a 3-ball, but inspecting the proof of Rourke-Sanderson's Theorem~3.26 in our particular case reveals that the ball pair $(N_e,e\cap X_V)$ is unknotted).

Clearly, $N_V\cup N_E$ is a disjoint union of handlebodies, thus $X_\Gamma^\circ := \SPH\setminus \intr(N_V\cup N_E)$ is a compact PL 3-manifold. The boundary of $X_\Gamma^\circ$ is contained in $\partial X_V\cup \partial N_E$. For each vertex $v\in V$, we denote by $R_v$ the \textbf{vertex region} of~$v$, which is the part of $\partial X_\Gamma^\circ$ contained in $N_v$. To be precise $R_v:=\partial N_v \cap \partial X_\Gamma^\circ$. Similarly, we call $R_e:=\partial N_e \cap \partial X_\Gamma^\circ$ the \textbf{edge region} of $e\in E$. Notice that $R_e$ is always an annulus, and if $v$ is a vertex of degree $d$, $R_v$ is a 2-sphere with the interiors of $d$ disjoint discs removed. 

The next step is to define $P_\Gamma^\circ$. Notice that the vertex regions and edge regions of $\partial X_\Gamma^\circ$ intersect in circles, which we will sometimes refer to as ``junctures''. We set $P_\Gamma^\circ:=\left(\bigcup_{v\in V}R_v\right) \cap \left(\bigcup_{e\in E} R_e\right)$ as the union of these junctures. Then $P_\Gamma^\circ$ separates $\partial X_\Gamma^\circ$ into the vertex and edge regions. As $X_\Gamma^\circ$ inherits an orientation from the ambient sphere $\SPH$, also $\partial X_\Gamma^\circ$ has a canonical orientation. Now we orient $P_\Gamma^\circ = \partial \left(\bigcup_{v\in V}R_v\right)$ as the boundary of the union of all vertex regions. Notice that instead viewing $P_\Gamma^\circ$ as the boundary of the union of the edge regions would induce the opposite orientation on $P_\Gamma^\circ$ -- in other words, the orientation of~$P_\Gamma^\circ$ determines which regions of $\partial X_\Gamma^\circ$ are vertex regions. As each edge is incident to at least one vertex, each vertex or edge region of $\partial X_\Gamma^\circ$ that has no boundary has to be a vertex region (corresponding to an isolated vertex).

An example of such a pair $(X_\Gamma^\circ,P_\Gamma^\circ)$ is illustrated in Figure~\ref{fig.orientedexterior}.

\begin{dfn}\label{dfn.orext}
  An \textbf{oriented marked exterior} of a spatial graph $\Gamma$ is a pair $(X_\Gamma^\circ, P_\Gamma^\circ)$, where $X_\Gamma^\circ$ and $P_\Gamma^\circ$ are oriented PL manifolds obtained by the above construction.
\end{dfn}
\begin{figure}[h]
  \centering
  \def \svgwidth{0.7\linewidth}
  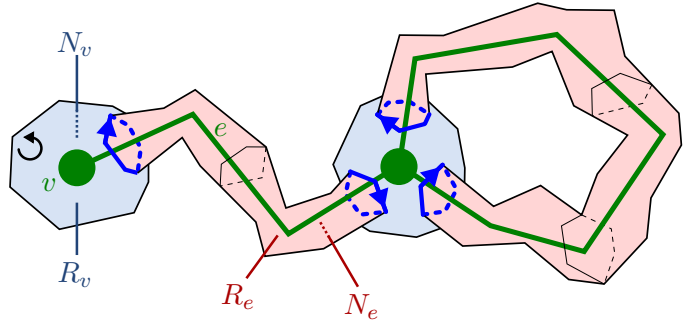
  \caption{The oriented marked exterior of a spatial graph. We indicate the orientation of the boundary and of the junctures by arrows. We also highlight the vertex and edge neighborhoods $N_v, N_e$ used in its construction, as well as the corresponding regions of the boundary.}
  \label{fig.orientedexterior}
\end{figure}

The next two propositions will show that the oriented marked exterior is indeed a useful notion for studying spatial graphs.

\begin{prop}[Well-definedness of oriented marked exteriors]\label{prop.orext_welldef}
  Let $\Gamma$ be a spatial graph and consider two oriented marked exteriors $(X_k^\circ,P_k^\circ)$ for $k\in\{1,2\}$. Then there exists an orientation-preserving homeomorphism of pairs $\Phi\colon (X_1^\circ,P_1^\circ)\to (X_2^\circ,P_2^\circ)$ such that for each vertex $v$ of $\Gamma$, the map~$\Phi$ sends the corresponding vertex region $R_{v,1}$ to $R_{v,2}$, and similarly for edge regions.
\end{prop}

\begin{proof}
  Let $N_{V,k}$ and $N_{E,k}$ be the regular neighborhoods that appear in the construction of $X_k^\circ$. Similarly, we will use the corresponding notation $N_{v,k}, N_{e,k}$ for the components corresponding to single vertices and edges, as well as $X_{V,k}$ for $\SPH\setminus\intr(N_{V,k})$.
  
  By the Regular Neighborhood Theorem for pairs \cite[Theorem~4.11]{RourkeSanderson}, there is a PL isotopy of $\SPH$ that carries $N_{V,1}$ to $N_{V,2}$, fixing $|\Gamma|$. This induces an orientation-preserving PL homeomorphism $\Phi_V\colon N_{V,1}\to N_{V,2}$. As $|\Gamma|$ is fixed during the isotopy, $\Phi_V$ maps $N_{v,1}$ to $N_{v,2}$ for each vertex $v$ of $\Gamma$ and fixes the edges as well. 
  
  Since PL homeomorphisms take regular neighborhoods to regular neighborhoods, $\Phi_V(N_{E,1})$ and $N_{E,2}$ are two  regular neighborhoods of $|\Gamma|\cap X_{V,2}$ in~$X_{V,2}$. Using the Regular Neighborhood Theorem again, we find a PL isotopy of~$X_{V,2}$ that carries $\Phi_V(N_{E,1})$ to~$N_{E,2}$.
  This isotopy carries $\Phi_V(X_1^\circ)$ to~$X_2^\circ$. This shows that the two marked exteriors are in fact PL-homeomorphic via an orientation-preserving PL homeomorphism $\Phi\colon X_1^\circ\to X_2^\circ$. As these isotopies fix $|\Gamma|$, each vertex and edge region of the boundary gets mapped to the corresponding vertex respectively edge region, thus mapping $P_1^\circ$ to~$P_2^\circ$. As $\Phi$ is orientation-preserving and the orientation of the junctures is determined by the orientation of $X_k$, $\Phi$ also preserves the orientation of the~$P_k^\circ$. 
\end{proof}

As we will often only use the oriented PL homeomorphism type of the oriented marked exterior, which is indeed unique by the above proposition, we will frequently refer to it as \emph{the} oriented marked exterior of the spatial graph.
We now show that the oriented marked exterior $(X_\Gamma^\circ,P_\Gamma^\circ)$ fully encodes the isomorphism type of $\Gamma$.

\begin{prop}[Faithfulness of oriented marked exteriors]\label{prop.orext_uniqueness}
  Let $\Gamma_1, \Gamma_2$ be two spatial graphs, $(X_1^\circ,P_1^\circ)$ and $(X_2^\circ,P_2^\circ)$ their oriented marked exteriors, and $\Phi_X\colon (X_1^\circ,P_1^\circ)\to (X_2^\circ,P_2^\circ)$ a PL homeomorphism preserving the orientation of both factors. Then $\Phi_X$ extends to an isomorphism $\Phi\colon \Gamma_1\to\Gamma_2$ of spatial graphs, where $\langle \Phi\rangle\colon \langle\Gamma_1\rangle \to \langle\Gamma_2\rangle$ is the induced map of underlying graphs.
\end{prop}

\begin{proof}
  Recall that the orientation of~$P_\Gamma^\circ$ determines which regions of~$\partial X_\Gamma^\circ$ are vertex regions. Thus, preserving the orientation of~$P_k^\circ$ is equivalent to mapping vertex regions to vertex regions (and edge regions to edge regions). We need only to show that an orientation-preserving PL homeomorphism $\Phi_X\colon X_1^\circ\to X_2^\circ$ respecting vertex/edge regions of the boundary extends to an isomorphism of spatial graphs $\Gamma_1\to \Gamma_2$.
  
  We will denote the regular neighborhoods that are used in the construction of~$X_k$ by $N_{V,k}$ and $N_{E,k}$.
  The discs $N_{E,k}\cap \partial N_{V,k}$ are bounded by the junctures in $P_k^\circ$, and each disc intersects $|\Gamma_k|$ in exactly one point in its interior.
  Since discs are always PL-homeomorphic to cones over any of their interior points, we can use the cone construction to extend~$\Phi_X$ to a PL homeomorphism $\Phi_{X^+}\colon X_1^\circ\cup (N_{E,1}\cap \partial N_{V,1})\to X_2^\circ \cup (N_{E,2}\cap \partial N_{V,2})$ mapping the intersection points of $|\Gamma_k|$ with $N_{E,k}\cap \partial N_{V,k}$ to each other.
  
  Note that each $(N_{e,k},e\cap X_{V,k})$ is an unknotted ball pair. As any PL homeomorphism of the boundary of an unknotted ball pair extends to the interior \cite[Theorem~4.4]{RourkeSanderson}, we can extend $\Phi_{X^+}$ to a PL homeomorphism $\Phi_{X_V}\colon X_{V,1}\to X_{V,2}$ that maps $|\Gamma_1|\cap X_{V,1}$ to $|\Gamma_2|\cap X_{V,2}$.
  
  Since each $N_{v,k}$ is PL-homeomorphic to a cone with base a 2-sphere containing $R_{v,k}$ and cone point $v$, such that $|\Gamma|\cap N_{v,k}$ corresponds to the cone over $|\Gamma|\cap \partial N_{v,k}$, we may cone the already defined map on each $\partial N_{v,1}$. This finishes the extension of $\Phi_X$ to the whole ambient sphere of~$\Gamma_1$, giving the desired isomorphism of graphs.
\end{proof}

We have shown how to fully encode a spatial graph as its oriented exterior, but
as Matveev's Recognition Theorem is insensitive to orientations, we need to refine our construction. Thus, we will introduce further markings on the boundary, which encode the orientation of the manifold and possibly decorations of the spatial graph.

\begin{dfn}[{\cite[Definition~3.3.9]{Matveev}}]\label{dfn.mfdwithbpattern}
  A \textbf{manifold with boundary pattern} $(M,P)$ is a PL 3-manifold $M$ together with a 1-dimensional subpolyhedron $P\subset\partial M$ containing no isolated points.
  
  A \textbf{homeomorphism} of manifolds with boundary pattern is a PL homeomorphism of pairs $(M_1,P_1)\to (M_2,P_2)$.
\end{dfn}

Of course if we ignore orientations, then $(X_\Gamma^\circ,P_\Gamma^\circ)$ as defined above is an example of a manifold with boundary pattern.

In the following two paragraphs we describe the idea behind the upgraded marking, and afterwards we give a precise definition. Note that in order to avoid restricting the spatial graph isomorphisms that can be detected by comparing marked exteriors, we need to ensure that all modifications to the oriented marked exterior are independent of artificial choices.

First, we have to encode which regions of $\partial X_\Gamma^\circ$ are vertex regions, and the orientations of $X_\Gamma^\circ$ and $P_\Gamma^\circ$. The orientation of $X_\Gamma^\circ$ can be recovered by the orientation of the junctures and knowing which regions of $\partial X_\Gamma^\circ$ are the vertex regions. To encode the orientation of a juncture $\gamma$, choose three points on it. An orientation of $\gamma$ is then the same data as a cyclic ordering of these three points. From one of the points extend one arc into the vertex region. From the next point in the cyclic ordering extend two arcs into the vertex region, and three arcs from the third point. This new marking encodes which regions are the vertex regions as well as the orientations of the junctures, and hence also the orientation of~$X_\Gamma^\circ$.

To also account for decorations of the spatial graph, we have to further modify these markings. To encode the vertex coloring $f\colon V\to \N$, rather then extending three arcs from the third point of each juncture, extend $f(v)+3$ arcs instead. Similarly, to encode the edge coloring $g\colon E\to \N$, extend $g(e)$ arcs from the same third point into the edge region. We are left to encode the orientations of the edges. Note that there are two junctures around each edge~$e$, corresponding to the positive and negative ends of~$e$ (with respect to any PL embedding $[-1,1]\to e\setminus V$ orienting~$e$). At the juncture corresponding to the positive end of~$e$, we extend a single arc into the edge region, starting from the second point of the cyclic ordering. After these modifications we are still able to tell apart vertex and edge regions locally around the junctures, as there are three points from which arcs extend into the vertex region, and from at most two of them there are arcs extending into the edge region.

To make this marking rigorous, we define $(k,l,m)$-model discs, where $k,l\in \N$ and $m\in\{0,1\}$, as seen in the left hand side of Figure~\ref{fig.modeldisc}.
A model disc is the standard oriented ball pair $(D,I):=([-1,1]^2,[-1,1]\times \{0\})$ together with a few line segments in $D$ emanating from points in $I$, as we now describe. Consider the points $p_1:=(-\frac{1}{2},0)$, $p_2:=(0,0)$ and $p_3:=(\frac{1}{2},0)$ of $I$. From $p_1$ we extend one line segment towards the upper region of $D$ (which induces the orientation of $I$). Starting at $p_2$ extend two line segments towards the same region and $m$ line segments into the lower region. From~$p_3$ extend $k+3$ line segments towards the upper region and $l$ line segments into the lower region of $D$.

To extend the marking on $X_\Gamma^\circ$, we first choose a regular neighborhood of the boundary pattern $P_\Gamma^\circ$ in the boundary $\partial X_\Gamma^\circ$. This regular neighborhood is comprised of an annulus $A_\gamma$ for each juncture $\gamma$. If the juncture $\gamma$ bounds the vertex region $R_v$ and the edge region $R_e$, choose an orientation-preserving PL embedding of the $(f(v),g(e),m)$-model disc $\Phi\colon(D,I)\to (A_\gamma,\gamma)$, where $f\colon V\to \N$ and $g\colon E\to \N$ are the coloring functions described in Section~\ref{sec.decorations} and $m=1$ if it is the juncture corresponding to the positive end of the incident edge, else $m=0$. Then, add the images of the line segments of the previous paragraph to the boundary pattern $P_\Gamma^\circ$. An example of this can be seen in Figure~\ref{fig.markedexterior}.

Additionally, we have to take care of isolated vertices. These are the only vertices whose regions of $\partial X_\Gamma$ have no marking yet. We will be using as a local model the marking seen in the right hand side of Figure~\ref{fig.modeldiscisolated}, again with 1, 2 and $f(v)+3$ lines extending from the triangle in the middle. We push forward this marking to each region $R_v$ corresponding to an isolated vertex $v$ via an orientation-preserving PL embedding of the disc. This also encodes the orientation of the boundary, which is important in case the spatial graph consists of only isolated vertices. 

Altogether,~$P_\Gamma^\circ$ together with the added markings near the junctures and at vertex regions corresponding to isolated vertices makes up a new marking~$P_\Gamma$.

\begin{dfn}\label{dfn.decext}
  A \textbf{marked exterior} of a decorated spatial graph~$\Gamma$ is a manifold with boundary pattern $(X_\Gamma,P_\Gamma)$, where $X_\Gamma=X_\Gamma^\circ$ as unoriented PL manifolds and $P_\Gamma$ is obtained from $P_\Gamma^\circ$ as described above.
\end{dfn}

We stress that whereas for an oriented marked exterior $(X_\Gamma^\circ,P_\Gamma^\circ)$, both $X_\Gamma^\circ$ and $P_\Gamma^\circ$ are oriented, the manifold $X_\Gamma$ is not explicitly oriented and $P_\Gamma$ is not even a manifold. Our construction was designed to allow for recovering the orientation data of $(X_\Gamma^\circ, P_\Gamma^\circ)$ from the unoriented object $(X_\Gamma,P_\Gamma)$.

\begin{figure}[h]
  \centering
  \def \svgwidth{0.35\linewidth}
  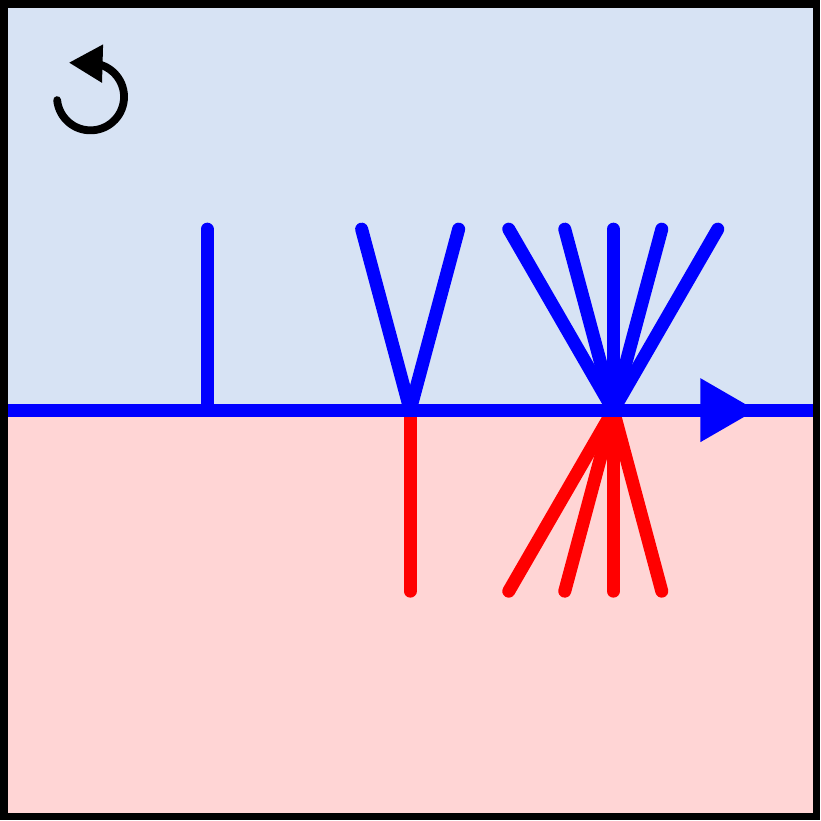 \hspace{4em}
  \def \svgwidth{0.35\linewidth}
  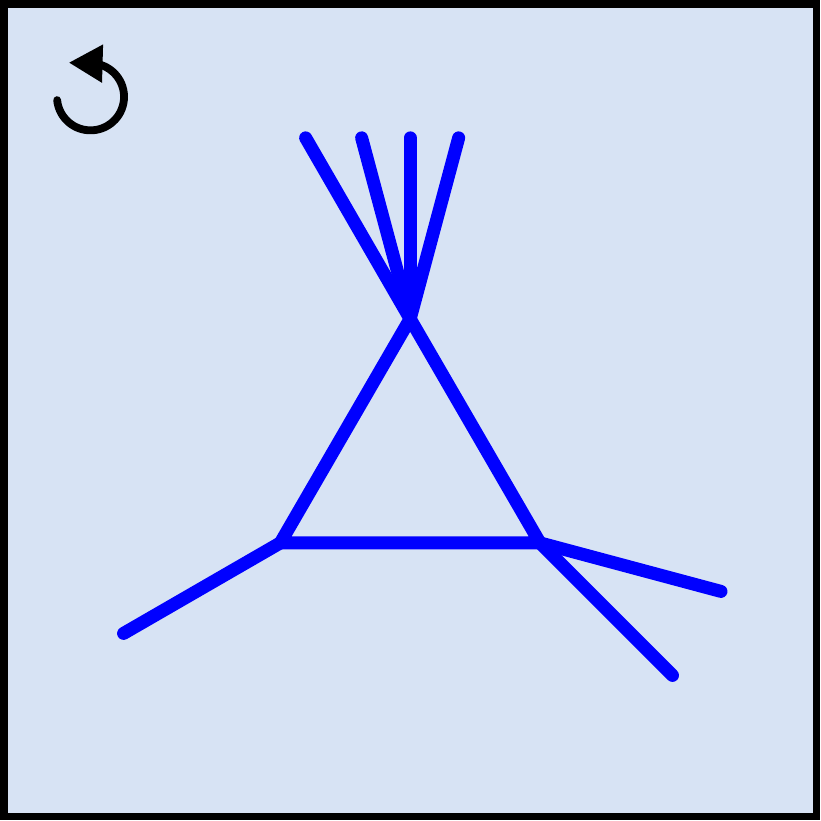
  \caption{Left: A $(2,4,1)$-model disc used in modifying $P_\Gamma^\circ$ into $P_\Gamma$. Right: Marking for an isolated vertex $v$ with vertex color $f(v)=1$. Arrows indicate orientations.}
  \label{fig.modeldisc}
  \label{fig.modeldiscisolated}
\end{figure}
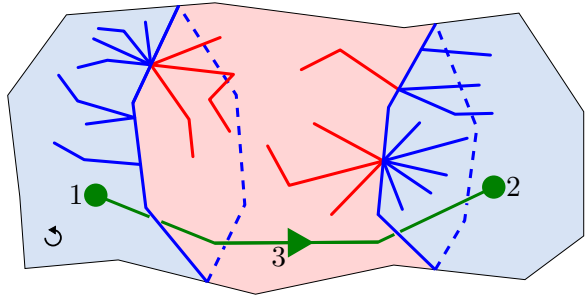
\begin{figure}[h]
  \centering
  \def \svgwidth{0.6\linewidth}
\begingroup%
  \makeatletter%
  \providecommand\color[2][]{%
    \errmessage{(Inkscape) Color is used for the text in Inkscape, but the package 'color.sty' is not loaded}%
    \renewcommand\color[2][]{}%
  }%
  \providecommand\transparent[1]{%
    \errmessage{(Inkscape) Transparency is used (non-zero) for the text in Inkscape, but the package 'transparent.sty' is not loaded}%
    \renewcommand\transparent[1]{}%
  }%
  \providecommand\rotatebox[2]{#2}%
  \newcommand*\fsize{\dimexpr\f@size pt\relax}%
  \newcommand*\lineheight[1]{\fontsize{\fsize}{#1\fsize}\selectfont}%
  \ifx\svgwidth\undefined%
    \setlength{\unitlength}{708.60735075bp}%
    \ifx\svgscale\undefined%
      \relax%
    \else%
      \setlength{\unitlength}{\unitlength * \real{\svgscale}}%
    \fi%
  \else%
    \setlength{\unitlength}{\svgwidth}%
  \fi%
  \global\let\svgwidth\undefined%
  \global\let\svgscale\undefined%
  \makeatother%
  \begin{picture}(1,0.51053282)%
    \lineheight{1}%
    \setlength\tabcolsep{0pt}%
    \put(0,0){\includegraphics[width=\unitlength,page=1]{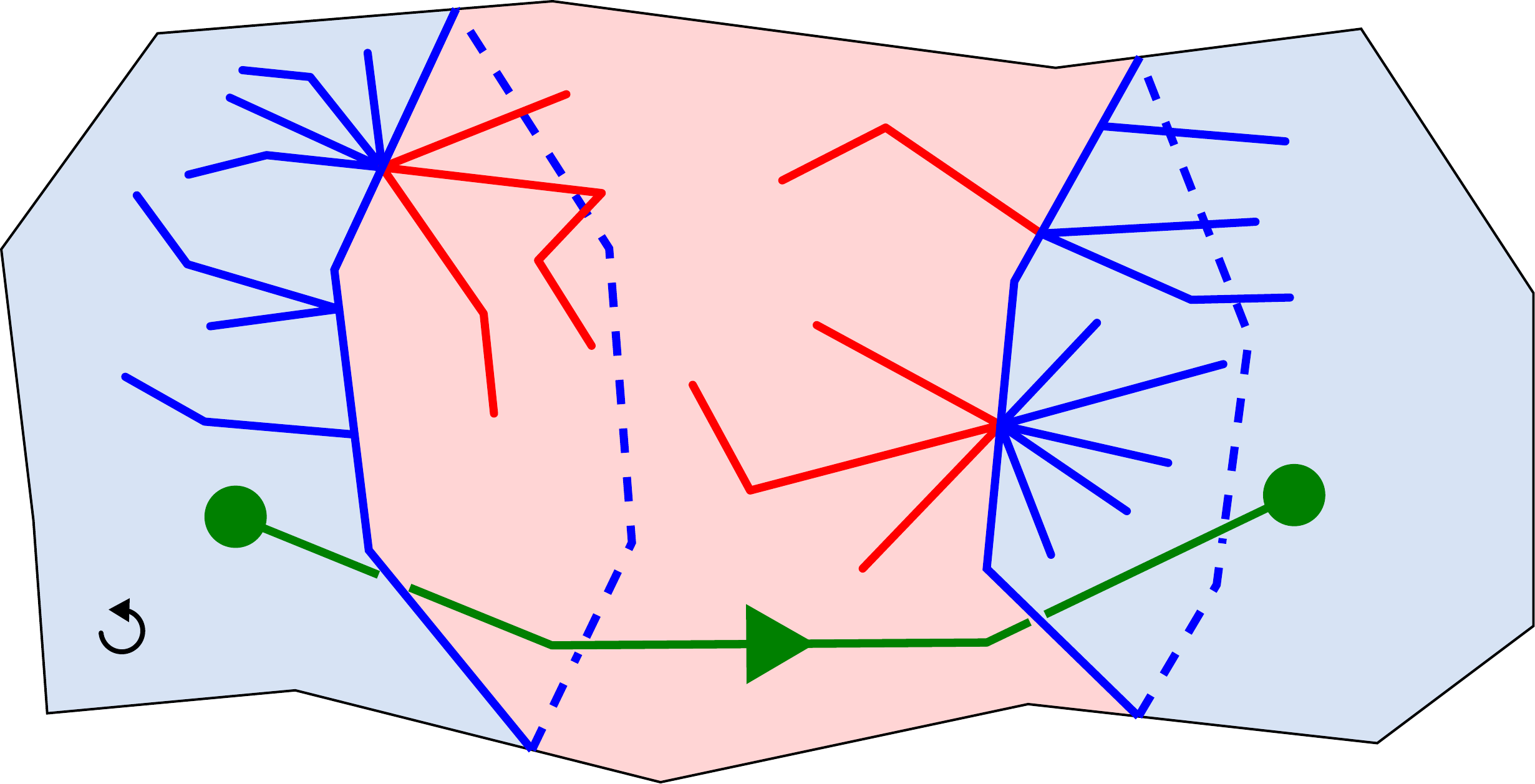}}%
    \put(0.10658515,0.15835437){\color[rgb]{0,0,0}\makebox(0,0)[lt]{\lineheight{1.25}\smash{\begin{tabular}[t]{l}1\end{tabular}}}}%
    \put(0.4579623,0.04924619){\color[rgb]{0,0,0}\makebox(0,0)[lt]{\lineheight{1.25}\smash{\begin{tabular}[t]{l}3\end{tabular}}}}%
    \put(0.86450968,0.17220406){\color[rgb]{0,0,0}\makebox(0,0)[lt]{\lineheight{1.25}\smash{\begin{tabular}[t]{l}2\end{tabular}}}}%
  \end{picture}%
\endgroup%

  \caption{The marked exterior of a decorated spatial graph. The colors of the vertices and the edge are indicated by the numbers.}
  \label{fig.markedexterior}
\end{figure}

In the next two propositions we will show the well-definedness and faithfulness of the marked exterior.

\begin{prop}[Well-definedness of marked exteriors]\label{prop.decext_welldef}
  Let $\Gamma$ be a decorated spatial graph and consider two marked exteriors $(X_k,P_k)$ for $k\in\{1,2\}$. Then there exists a homeomorphism of manifolds with boundary pattern $\Phi\colon (X_1,P_1)\to (X_2,P_2)$ such that for each vertex $v$ of $\Gamma$, $\Phi$ sends the corresponding vertex regions $R_{v,1}$ to $R_{v,2}$ and similarly for edge regions.
\end{prop}

\begin{proof}
  By Proposition~\ref{prop.orext_welldef} we know that the oriented marked exteriors $(X_k^\circ,P_k^\circ)$ used to construct the $(X_k,P_k)$ are homeomorphic as manifolds with boundary pattern via a homeomorphism $\Phi\colon (X_1^\circ,P_1^\circ)\to (X_2^\circ, P_2^\circ)$ respecting vertex and edge regions of the boundary, as well as the orientation of each factor.
  By the Regular Neighborhood Theorem \cite[Theorem~3.24]{RourkeSanderson} there is a PL isotopy $H_A$ of $\partial X_2$ fixing $P_2^\circ$ and pushing the $\Phi$-image of the chosen regular neighborhood of $P_1^\circ$ onto the chosen regular neighborhood of $P_2^\circ$. Denote the final homeomorphism of this isotopy by $\Psi$.
  
  Then, for each juncture $\gamma\subset P_1^\circ$, using the Disc Theorem for pairs (Theorem~\ref{thm.discpair}), we get a PL isotopy $H_\gamma$ of $A_{\Phi(\gamma)}$ that fixes $\partial A_{\Phi(\gamma)}$, preserves the juncture $\Phi(\gamma)$, and isotopes the postcomposition with $\Psi\circ \Phi$ of the embedding of the model disc at $\gamma$ to the embedding of the model disc at $\Phi(\gamma)$. By using all the PL isotopies $H_\gamma$ and extending to the whole boundary $\partial X_2$ as the identity, we obtain a PL isotopy $H_D$ of $\partial X_2$, carrying $\Psi\circ\Phi(P_1)$ to $P_2$.
  
  The isotopy of $\partial X_2$ obtained by the concatenation of $H_A$ and $H_D$ can be extended to a PL isotopy of $X_2$ \cite[Proposition 3.22(ii)]{RourkeSanderson}, whose final homeomorphism, when precomposed with $\Phi$, yields a homeomorphism of manifolds with boundary pattern $(X_1,P_1)\to (X_2,P_2)$ respecting the vertex and edge regions.
\end{proof}

\begin{prop}[Faithfulness of marked exteriors]\label{prop.decext_uniqueness}
  Let $(X_1,P_1)$ and $(X_2,P_2)$ be marked exteriors for two non-empty decorated spatial graphs $\Gamma_1$ and $\Gamma_2$, and $\Phi_X\colon (X_1,P_1)\to (X_2,P_2)$ a PL homeomorphism between them. Then $\Phi_X$~extends to an isomorphism $\Phi\colon \Gamma_1\to\Gamma_2$ of decorated spatial graphs, where $\langle \Phi\rangle\colon \langle\Gamma_1\rangle \to \langle\Gamma_2\rangle$ is the induced map of underlying graphs.
\end{prop}

\begin{proof}
  For $k\in\{1,2\}$ consider the boundary patterns $P_k^\circ\subset P_k$ consisting only of the junctures. Clearly, $\Phi_X$ has to map $P_1^\circ$ to $P_2^\circ$, because each $P_k^\circ$ can be intrinsically characterized as the union of all embedded circles in $P_k$ that are not contained in boundary regions corresponding to isolated vertices.
  As the components of $\partial X_k$ containing only a single component of $P_k$ are exactly the regions corresponding to isolated vertices, it is possible to intrinsically distinguish the two types of circles.
  
  Observe that $\Phi_X$ preserves vertex regions and edge regions: indeed, as we just saw, components of $\partial X_k$ with a single component of $P_k$ correspond to isolated vertices, and on all other connected components, vertex regions are those receiving arcs from 3 distinct points of each juncture, while edge regions only receive arcs from at most 2 distinct points of each juncture. As the number of arcs extended into the vertex region encodes the orientation of each juncture, $\Phi_X$ has to preserve the orientation of $P_k^\circ$, as well as the orientation of the triangles at isolated vertices. As the orientation of $X_k^\circ$ is determined by the orientation of the junctures and which side of those the vertex region is on, $\Phi_X$ also preserves the orientation of the $X_k^\circ$.
  
  Overall, we conclude that $\Phi_X$ is a PL homeomorphism of the oriented marked exteriors $(X_1^\circ,P_1^\circ)\to (X_2^\circ,P_2^\circ)$ preserving the orientation of each factor. Thus,  Proposition~\ref{prop.orext_uniqueness} can be applied to conclude that, up to decorations, $\Phi_X$ extends to an isomorphism $\Phi\colon \Gamma_1\to\Gamma_2$, where $\langle \Phi \rangle\colon \langle\Gamma_1\rangle \to \langle\Gamma_2\rangle$ is the induced map of underlying graphs. However, from the way the decorations got encoded into the boundary patterns, it is clear that $\Phi$ has to respect decorations.
\end{proof}
\subsection{Properties of marked exteriors}\label{sec.graphextproperties}

Having introduced marked exteriors, we explain how the indecomposability properties of a spatial graph~$\Gamma$ studied in Section~\ref{sec.preliminaries} (namely, being non-separable and having no cut vertices) translate into features of the marked exterior~$(X_\Gamma, P_\Gamma)$. 

\begin{dfn}Let $(M, P)$ be a manifold with boundary pattern.
\begin{itemize}
\item Let $S \subset M$ be a properly embedded PL $2$-sphere. If $S$~does not bound a PL $3$-ball in~$M$, then $S$~is called a \textbf{reducing sphere} for~$M$. We call~$M$ \textbf{reducible} if it admits a reducing sphere; otherwise it is \textbf{irreducible}. We also apply the same terminology to~$(M, P)$.
\item A subspace $X\subseteq M$~is called \textbf{clean} if $X \cap P = \emptyset$. 
 Let $D \subset M$~be a clean properly embedded PL disc. If $\partial D$~does not bound a clean disc in $\partial M$, then $D$~is called a \textbf{reducing disc} for~$(M, P)$. We call~$(M,P)$ \textbf{boundary-reducible} if it has a reducing disc; otherwise it is \textbf{boundary-irreducible}.\footnote{The definition of boundary-irreducibility for manifolds with boundary pattern is given in Matveev's book simply as the ``straightforward generalization'' of the notion for $3$-manifolds without boundary pattern \cite[p.~126]{Matveev}, leaving unclear whether the definition of a reducing disc~$D$ for~$(M, P)$ allows $\partial D$~to bound a non-clean disc in~$\partial M$. It is stated on p.~127 that if $M$~is a solid torus, then $(M, P)$~is boundary-irreducible if and only if $\partial M \setminus P$~contains a meridian of~$M$. This is only true if $\partial D$~is never allowed to bound a disc on $\partial M$, even when that disc intersects~$P$ (contrary to the definition we gave). We suspect this example is due to an oversight. Indeed, Matveev's usage of the notion (namely, in the proofs of Lemmas 4.1.33~and~4.1.35) relies on the boundary of a clean properly embedded disc in a boundary-irreducible~$(M, P)$ bounding a clean disc in~$\partial M$. This is compatible with the definition we present, and with the usage elsewhere in the literature~\cite[p.~16]{KauffmanManturov}.}
\end{itemize}
\end{dfn}

\begin{prop}[Separability and reducibility]\label{prop.sep<=>red}
Let $\Gamma$ be a decorated spatial graph and $(X_\Gamma, P_\Gamma)$ its marked exterior. Then $\Gamma$~is separable if and only if $X_\Gamma$~is reducible.
\end{prop}

Note that this statement concerns only the $3$-manifold~$X_\Gamma$, and is independent of the boundary pattern~$P_\Gamma$.

\begin{proof}
  $(\Rightarrow)$ Suppose $\Gamma$~is separable with $S$~a separating sphere. Then if we build~$X_\Gamma$ out of small enough regular neighborhoods $N_V$~and~$N_E$ to avoid~$S$, we see that $S$~is a reducing $2$-sphere for~$X_\Gamma$: indeed, no component of~$X_\Gamma \setminus S$ is an open $3$-ball, as both have non-empty boundary.

  $(\Leftarrow)$
  Denote by~$\SPH$ the ambient $3$-sphere of~$\Gamma$, assume $S$~is a reducing sphere for a marked exterior~$X_\Gamma \subset \SPH$, and let $B_1, B_2 \subset \SPH$ be the $3$-balls into which $S$~splits~$\SPH$. If for some $i\in \{1,2\}$ the intersection $|\Gamma| \cap B_i$ were empty, then we would have $B_i \subset X_\Gamma$, in contradiction with $S$~being a reducing sphere. Hence, if $S$~decomposes~$\Gamma$ as $\Gamma_1 \sqcup \Gamma_2$, then none of the~$\Gamma_i$ is empty.
\end{proof}

Note that the second part of the proof actually shows a finer statement:

\begin{cor}[Reducing spheres are separating]\label{cor.reducingsphereseparates}
If $S$~is a reducing sphere for a marked exterior~$(X_\Gamma, P_\Gamma)$, then $S$~is a separating sphere for~$\Gamma$.
\end{cor}

The relationship between cut vertices of a graph and boundary-reducibility of its marked exterior is more subtle, so we study each direction of the correspondence separately. 

\begin{prop}[Boundary reducibility from cut vertices]\label{prop.cut=>bred}
Let $\Gamma$~be a non-separable decorated spatial graph. If $\Gamma$~has a cut vertex, then the marked exterior~$(X_\Gamma, P_\Gamma)$ is boundary-reducible.
\end{prop}

\begin{proof}
Let $\Gamma = (\SPH, V, E)$, let $v$~be a cut vertex for~$\Gamma$, and let $S$~be a cut sphere through~$v$. We construct a marked exterior~$(X_\Gamma, P_\Gamma)$ using a small enough regular neighborhood~$N_V$ of the vertex set so that $N_V$~is in fact a regular neighborhood of~$V$ in~$(\SPH, |\Gamma| \cup S)$, and also, we use $N_E$~small enough to be disjoint from~$S$. Additionally, we ensure that $P_\Gamma$~is built from disc embeddings with image small enough to be disjoint from~$S$, so $S \cap P_\Gamma = \emptyset$. As $S \cap N_v$~is a regular neighborhood of~$\{v\}$ in~$S$, it is a disc. The other side $D := S\cap X_\Gamma$ is thus a clean properly embedded disc in~$(X_\Gamma, P_\Gamma)$.

We will show that $D$~is a reducing disc for~$(X_\Gamma, P_\Gamma)$. To this end, we consider the two balls~$B_1, B_2$ into which $S$~separates~$\SPH$. The curve $\partial D$ separates the component~$C$ of $\partial X_\Gamma$ containing the vertex region~$R_v$ into the two regions~$C_i:= C \cap B_i$, for $i \in \{1,2\}$. We need to show that none of the~$C_i$ is a clean disc. But since $S$~is a cut sphere, there is at least one edge~$e_i$ incident to~$v$ on each~$B_i$. Since the corresponding component~$N_{e_i}$ of~$N_E$ is disjoint from~$S$, we have~$N_{e_i} \subset B_i$, and in particular~$R_{e_i} \subset C_i$. The juncture between $R_{e_i}$~and~$R_v$ is thus contained in~$C_i$, whence $C_i$~is not clean, and certainly not a clean disc.
\end{proof}

We will see that a converse statement also holds, except for one particular (isomorphism type of) spatial graph, which we first define.

\begin{dfn}
A spatial graph is called a \textbf{one-edge graph} if it has exactly two vertices and one edge, with the edge being incident to both vertices. 
\end{dfn}

Note that the property of a spatial graph~$\Lambda$ being a one-edge graph is determined by the isomorphism type of~$\langle \Lambda \rangle$. Moreover, the following statement shows that two decorated one-edge graphs are isomorphic precisely if their underlying abstract graphs are isomorphic. In particular, any two undecorated one-edge graphs are isomorphic.

\begin{lem}[Uniqueness of one-edge graphs]\label{lem.one-edge}
Let~$\Lambda_1, \Lambda_2$ be decorated one-edge graphs, and let $F\colon \langle\Lambda_1\rangle \to \langle \Lambda_2 \rangle$ be an isomorphism of their underlying decorated abstract graphs. Then there is an isomorphism $\Phi \colon \Lambda_1 \to \Lambda_2$ with~$\langle \Phi \rangle = F$.
\end{lem}

This is a straightforward consequence of the following general fact (with $M$~a $3$-sphere).

\begin{prop}[Arcs in the interior of connected manifolds]
Let $M$~be a connected PL manifold of dimension at least~$2$. For each $k \in \{1,2\}$, let $I_k$~be a PL-embedded arc in~$\intr(M)$ with endpoints $v_k, u_k$. Then there is a PL isotopy of~$M$ carrying $I_1$~onto~$I_2$, $v_1$~to~$v_2$, and $u_1$~to~$u_2$.
\end{prop}

Note that it is not true in general that any two PL-embedded $n$-balls  in the interior of a PL manifold of dimension $\ge n+1$ are ambient-isotopic. For example, if one considers the cone~$D$ of a trefoil knot in $\partial([-1,1]^4)$, with the origin as cone point, then $D$~cannot be ambiently isotoped in~$\R^4$ onto the disc $[-1,1]^2 \times \{0\}^2$. This is a consequence of the fact that links of pairs of polyhedra are PL invariants \cite[pp.~50-51]{RourkeSanderson}.

\begin{proof}
We will need the following fact:
\begin{claim}
For every PL embedded arc $I\subset \intr(M)$ with endpoints $v,u$, and for every $u' \in I\setminus \{v\}$, there is a PL isotopy of~$M$ fixing~$v$ and carrying~$I$ to the sub-arc of~$I$ with endpoints~$v, u'$.
\end{claim}
Before justifying this claim, let us use it to prove the proposition.

By homogeneity of manifolds \cite[Lemma~3.33]{RourkeSanderson}, there is a PL isotopy of $M$~carrying~$v_1$ to~$v_2$, so we may assume $v_1=v_2=:v$.

Choose a star neighborhood~$N_v$ of~$v$ in the pair $(M, I)$, and denote by~$u_1', u_2'$, respectively, the points of intersection of the $(n-1)$-sphere~$\partial N_v$ with each arc~$I_1, I_2$. Using the above claim on both arcs reduces the problem to showing that there is a PL isotopy of~$M$ carrying the straight line segment~$[v,u_1']$ onto~$[v, u_2']$, with $v$~being carried to itself.

Since $\partial N_v$~is connected, again by homogeneity of manifolds, there is a PL isotopy of~$\partial N_v$ carrying~$u_1'$ to~$u_2'$ (this is the only point of the proof where we use the assumption that $\dim(M)\ne 1$). By coning at~$v$, this isotopy extends to~$N$, taking~$[v,u_1']$ to~$[v, u_2']$ as required. To extend it to all of~$M$, we use the general fact every PL isotopy of the boundary of a manifold (in this case~$M \setminus \intr(N_v)$) extends to the interior \cite[Proposition~3.22(ii)]{RourkeSanderson}.

All that is left is to prove the above claim:

\begin{proof}[Proof of the Claim]
It suffices to show that the subspace~$Q\subset I\setminus\{v\}$ of points~$u'$ for which the claim holds is non-empty, open, and closed in~$I\setminus\{v\}$. Clearly we have $u \in Q$.

Let us verify the openness condition, beginning with the point~$u$. A regular neighborhood~$N_u$ of~$\{u\}$ in the pair~$(M, I)$ is PL-homeomorphic to the standard $n$-ball $[-1,1]^n$, with $N_u \cap I$~corresponding to the straight line segment from~$0$ to a point~$p_0$ in~$\partial([-1,1]^n)$. Let $q_0$ be~in the interior of this line segment. Since $[-1,1]^n$~is a cone with base $\partial([-1,1]^n)$ over \emph{any} of its interior points, the formula  $tp \mapsto (1-t) q_0 + tp$ with $p \in \partial([-1,1]^n)$ can be used to define a PL homeomorphism of~$[-1,1]^n$ fixing the boundary and taking $[0, p_0]$ to $[q_0,p_0]$. Such a map is PL-isotopic to the identity on~$[-1,1]^n$ keeping the boundary fixed, by the Alexander trick \cite[Proposition~3.22(i)]{RourkeSanderson}. This isotopy can then be transferred to a PL isotopy of~$N_u$, and extended as the constant isotopy in all of~$M$. This shows that the point $q \in N$ corresponding to~$q_0$ is in~$Q$, and so $Q$~contains the half-open interval~$\intr(N_u) \cap I$.

To verify the openness condition at points~$u'$ of~$Q$ other than~$u$, one proceeds similarly, choosing a regular neighborhood~$N_{u'}$ of~$u'$ in~$(M, I)$, and modeling~$(N_{u'}, N_{u'}\cap I)$ as the standard ball pair $([-1,1]^n, [-1,1]\times \{0\}^{n-1})$. The construction from before shows the intersection~$\intr(N_{u'})\cap I$ is in~$Q$. The same argument also shows that $Q$ is closed in~$I \setminus \{v\}$. \phantom{\qedhere}
\end{proof}
With the claim established, the proof is complete.
\end{proof}

\begin{prop}[Cut vertices from boundary-reducibility]\label{prop.bred=>cut}
Let $\Gamma$~be a non-separable decorated spatial graph that is not a one-edge graph. If its marked exterior~$(X_\Gamma, P_\Gamma)$ is boundary-reducible, then $\Gamma$~has a cut vertex. Moreover, there is an algorithm to produce a cut sphere for~$\Gamma$ from any reducing disc for~$(X_\Gamma, P_\Gamma)$.
\end{prop}

\begin{proof}
Let $D$~be a reducing disc for a marked exterior~$(X_\Gamma, P_\Gamma)$, and denote by~$\SPH$ the ambient sphere of~$\Gamma$. Since~$\partial D$ is disjoint from~$P_\Gamma$, it is contained in a vertex region~$R_v$ or an edge region~$R_e$ of~$\partial X_\Gamma$.

Let us first treat the case where $\partial D \subset \partial R_v$ for some vertex~$v$ of~$\Gamma$. 
Consider the $3$-ball~$N_v$ containing~$v$, which is a component of the vertex set neighborhood~$N_V$ used in constructing~$(X_\Gamma, P_\Gamma)$. Since $N_v$~is a regular neighborhood of~$\{v\}$ in $(\SPH, |\Gamma|)$, the pair~$(N_v, N_v \cap |\Gamma|)$ is PL-homeomorphic to a cone of $(\partial N_v, \partial N_v\cap|\Gamma|)$, with $v$~corresponding to the cone point. Let $D_v\subset N_v$~be the disc corresponding to the cone of~$\partial D$, and consider the $2$-sphere $S := D \cup D_v$, which intersects~$|\Gamma|$ precisely at~$v$. We will see $S$~induces a non-trivial vertex sum decomposition of~$\Gamma$, and so is a cut sphere.

Let $B_1, B_2 \subset \SPH$ be the $3$-balls into which $S$~separates~$\SPH$, let $C$~be the component of~$\partial X_\Gamma$ containing~$\partial D$, and consider the two surfaces $C_i = B_i \cap C$ into which $\partial D$~cuts~$C$. Since $D$~is a reducing disc, none of the~$C_i$ is a clean disc.
This implies that there are edges of~$\Gamma$ incident to~$v$ on both sides of~$S$, and so none of the summands in the decomposition $\Gamma = \Gamma_1 \vsum{v}{v} \Gamma_2$ induced by~$S$ is a one-point graph. Hence $S$~is a cut sphere for~$\Gamma$, and $v$~a cut vertex.

Now we treat the case where $\partial D\subset R_e$ for some edge $e$~of~$\Gamma$. First, observe that one of the vertices incident to~$e$ has degree at least~$2$. Indeed, if both had degree~$1$, then the component of~$\partial X_\Gamma$ containing~$R_e$ would be a $2$-sphere with only the edge~$e$ in one of its sides, and no other vertices besides its endpoints. Since $\Gamma$~is non-separable, this would be all of~$\Gamma$. Hence $\Gamma$~would be a one-edge graph, contrary to assumption. So let $v$~be a vertex incident to~$e$ of degree at least~$2$.

Since no component of~$P_\Gamma$ is contained in~$R_e$ and $\partial D$, being a reducing disc, is not allowed to bound a clean disc in~$R_e$, we conclude $\partial D$~does not bound a disc in~$R_e$. It is therefore parallel to one of the two boundary components of~$R_e$; in other words, it cuts~$R_e$ into two annuli. Let $R_e'$~be one such annulus having one of its boundary components in~$R_v$, and consider the enlarged disc $D':= D \cup R_e'$. The boundary~$\partial D'$ of this disc is contained in the boundary pattern~$P_\Gamma$, being the juncture between~$R_e$ and~$R_v$. As before, let $D_v$~be the disk obtained by coning~$\partial D'$ at~$v$, and define $S:=D' \cup D_v$.

We now show $S$~is a cut sphere for~$\Gamma$. Clearly, $S \cap |\Gamma| = \{v\}$. From the description of~$S\cap N_v$ as a cone of the juncture between $R_e$~and~$R_v$, we see that one side of~$S$ contains~$e$, and the other side contains all other edges of~$\Gamma$ that are incident to~$v$. Since $v$~has degree at least two, it follows that there are edges incident to~$v$ on both sides of~$S$. Hence $S$~induces a non-trivial vertex sum decomposition of~$\Gamma$.
\end{proof}

We finish this section by discussing the relation between degree-$1$ vertices in a spatial graph, and one-edge graphs. Using uniqueness of one-edge graphs (Lemma~\ref{lem.one-edge}), we then deduce a uniqueness result about spatial graphs whose underlying graphs are trees (Theorem~\ref{thm.trivialtrees}).

\begin{lem}[One-edge graph summands from leaves]\label{lem.leafsummand}
Let $\Gamma = (\SPH, V, E)$ be a decorated spatial graph. Let $u$~be a leaf of~$\Gamma$, let $e$~be the edge incident to~$u$, and let $v$~be the other vertex incident to~$e$. For the sub-graph $\Gamma_0 := (\SPH, V\setminus \{u\}, E \setminus \{e\})$ and the one-edge sub-graph $\Lambda := (\SPH, \{u,v\}, \{e\})$, we have~$\Gamma = \Gamma_0 \vsum{v}{v} \Lambda$.
\end{lem}
\begin{proof}
By Lemma~\ref{lem.vspheresdontlie}, we need only find a $2$-sphere~$S$ intersecting~$|\Gamma|$ exactly at~$v$, such that $|\Gamma_0|$~is in one side of~$S$, and~$|\Lambda|$ is in the other.

Let $(X_\Gamma^\circ, P_\Gamma^\circ)$~be an oriented marked exterior for~$\Gamma$, and denote by~$\gamma$ the juncture between the regions~$R_e, R_v$ of~$\partial X_\Gamma^\circ$. The component~$N_v$ of the vertex set neighborhood used in constructing~$X_\Gamma^\circ$ is PL-homeomorphic to a cone of the pair~$(\partial N_v, \partial N_v \cap |\Gamma|)$, with $v$~corresponding to the cone point. We denote by~$D$ the disc properly embedded in~$N_v$ that corresponds to the cone of~$\gamma$, and by~$C$ the $3$-ball that corresponds to the cone over the disc~$N_v \cap N_e$.

Recall that $R_e$~is a cylinder, and that since~$u$ is a leaf, the region~$R_u$ of~$\partial X_\Gamma^\circ$ corresponding to~$u$ is a disc. The $2$-sphere $S := R_u \cup R_e \cup D$ now decomposes $\Gamma$ as~$\Gamma_0 \vsum{v}{v}\Lambda$: First, it is clear that $S \cap |\Gamma| = \{v\}$. Moreover, one of the sides of~$S$ is the $3$-ball $N_u \cup N_e \cup C$, which intersects~$|\Gamma|$ precisely at~$|\Lambda|$. The other side must then contain~$|\Gamma_0|$.
\end{proof}

\begin{dfn} A \textbf{spatial tree} is a spatial graph~$\Gamma$ whose underlying graph~$\langle \Gamma \rangle$~is a tree.
\end{dfn}

\begin{thm}[Uniqueness of spatial trees]\label{thm.trivialtrees}
Let $\Gamma, \Gamma'$~be decorated spatial trees, and let $F\colon \langle \Gamma\rangle \to \langle \Gamma'\rangle$ be an isomorphism of their underlying decorated graphs. Then there is an isomorphism $\Phi \colon \Gamma \to \Gamma'$ with $\langle \Phi \rangle = F$.
\end{thm}

This result gives an algorithm for testing whether two decorated spatial trees are isomorphic, without appeal to the sophisticated machinery of Matveev. Indeed, it tells us that decorated spatial trees are isomorphic if and only if their underlying graphs are, reducing the problem to a (finite) search for an isomorphism of abstract trees.

\begin{proof}
We proceed by induction on the number of vertices of~$\Gamma$, with the cases $\Gamma \cong \zero$ and $\Gamma \cong \one$ being trivial.

Assume now that $\Gamma$~has at least two vertices. By a standard result in graph theory \cite[Exercise~1.2.5]{Jun05}, finite tress with at least two vertices always have leaves, so $\langle \Gamma \rangle$, and thus also~$\Gamma$, has a leaf~$u$. Denote by~$e$ the only edge of~$\Gamma$ incident to~$u$, and by~$v$ the other vertex incident to~$e$. Denote by~$\Lambda$ the one-edge sub-graph of~$\Gamma$ comprised of the vertices $u,v$ and the edge~$e$, and by~$\Gamma_0$ the sub-graph of~$\Gamma$ obtained by excluding $e$~and~$u$. By Lemma~\ref{lem.leafsummand}, we have $\Gamma = \Gamma_0 \vsum{v}{v} \Lambda$.

Similarly, let $\Gamma_0'$~be the sub-graph of~$\Gamma'$ obtained by excluding the edge~$F(e)$ and the leaf~$F(u)$, and let $\Lambda'$~be the one-edge sub-graph of~$\Gamma'$ comprised of $F(u), F(v)$~and~$F(e)$. As before, we have $\Gamma' = \Gamma_0' \vsum{F(v)}{F(v)} \Lambda'$.

Now by induction, the isomorphism of decorated trees $F|_{\Gamma_0} \colon \langle \Gamma_0\rangle \to \langle \Gamma_0'\rangle$ is induced by an isomorphism of decorated spatial trees $\Phi_0 \colon \Gamma_0 \to \Gamma_0'$. On the other hand, Lemma~\ref{lem.one-edge} gives an isomorphism $\Phi_\Lambda \colon \Lambda \to \Lambda'$ inducing $F|_{\Lambda}\colon \langle \Lambda \rangle \to \langle \Lambda' \rangle$. Assembling these two isomorphisms (as in Lemma~\ref{lem.isovsum}), we obtain the desired $\Phi := \Phi_0 \vsum{v}{v} \Phi_\Lambda$.
\end{proof}

\section{Algorithmic theory of spatial graphs}\label{sec.algorithms}

In this section, we import the final pieces of terminology that will allow us to state Matveev's Recognition Theorem, and assemble the theory developed so far into a proof of Theorem~\ref{thm.algdetection}.

\begin{dfn}Let $M$~be a compact PL $3$-manifold and $\Sigma \subset M$ a properly PL-embedded surface.
\begin{itemize}
\item The surface~$\Sigma$ is \textbf{incompressible} if for every PL-embedded disc~$D \subset M$ such that $D \cap \Sigma = \partial D$, the boundary~$\partial D$ bounds a disk in~$\Sigma$. 
\item We say $S$~is \textbf{two-sided} if its normal bundle is trivial. 
\end{itemize}
\end{dfn}

\begin{dfn}[{\cite[Definition~4.1.20]{Matveev}}]
A compact PL $3$-manifold~$M$ is \textbf{sufficiently large} if there exists a PL-embedded closed connected surface~$\Sigma\subset M$, such that $\Sigma$~is incompressible, two-sided, and not a $2$-sphere or a real projective plane.
\end{dfn}

\begin{dfn}[{\cite[Definition~6.1.5]{Matveev}}]\label{dfn.Haken}
A manifold with boundary pattern $(M,P)$ with $M,P$~compact is called \textbf{Haken} if it is irreducible, boundary-irreducible, and either:\begin{itemize}
\item $M$~is sufficiently large, or
\item $P\neq\emptyset$ and $M$ is a handlebody of positive genus.
\end{itemize}
\end{dfn}

The following proposition, which we also outsource to Matveev's book \cite[Corollary~4.1.27]{Matveev}, tells us that when $M$~has non-empty boundary, we are very close to satisfying one of the conditions in the bullet-points: 

\begin{prop}[Non-triviality of the boundary]\label{prop.nontrivboundary}
Every irreducible PL $3$-manifold with non-empty boundary is either a handlebody or sufficiently large.
\end{prop}

Having the definition of a Haken manifold with boundary pattern, we can state Matveev's Theorem \cite[Theorem~6.1.6]{Matveev}:

\begin{thm}[Matveev's Recognition Theorem]\label{thm.matveev}
  There is an algorithm to decide whether or not any two given Haken 3-manifolds with boundary pattern are PL-homeomorphic (as manifolds with boundary pattern).
\end{thm}

Our goal is to apply Matveev's Recognition Theorem to the marked exterior of blocks, and in order to ensure the ``Haken'' condition is met, we intend to apply~Proposition~\ref{prop.nontrivboundary}. However, that proposition leaves room for the exterior of a block to be a genus-$0$ handlebody, that is, a $3$-ball. To control that case, we will use the ``only if'' direction in the following lemma:

\begin{lem}[Blocks with $3$-ball exteriors]\label{lem.3ballexterior}
Let $\Lambda$~be a block with marked exterior~$(X_\Lambda, P_\Lambda)$. Then $X_\Lambda$~is a $3$-ball if and only if $\Lambda$~is a one-edge graph.
\end{lem}
\begin{proof}
$(\Leftarrow)$
If $\Lambda$~is a one-edge graph in~$\SPH$, then the vertex set neighborhood~$N_V$ used in constructing~$X_\Lambda$ is a pair of~$3$-balls. The edge-set neighborhood~$N_E$ is then a single $3$-ball connected to the previous two by a pair of discs on its boundary. Hence the union~$N_V \cup N_E$~is a $3$-ball, and so the exterior $X_\Lambda = \SPH \setminus \intr(N_V \cup N_E)$ is a $3$-ball as well.

$(\Rightarrow)$
Suppose~$X_\Lambda$~is a $3$-ball and let us consider the oriented marked exterior~$(X_\Lambda^\circ, P_\Lambda^\circ)$. Since $X_\Lambda$~is recovered from~$X_\Lambda^\circ$ merely by forgetting the orientation, $X_\Lambda^\circ$ is a $3$-ball. Since $P_\Lambda^\circ$~is a $1$-submanifold of the $2$-sphere~$\partial X_\Lambda^\circ$, it is a collection of circles (non-empty, since $\Gamma$, being a block, has at least one edge). The disk bounded by an innermost such circle is then the region~$R_u$ corresponding to some leaf~$u$. By Lemma~\ref{lem.leafsummand}, we obtain a vertex sum decomposition $\Lambda = \Lambda_0 \vsum{v}{v} \Lambda'$, where $\Lambda'$~is the one-edge sub-graph comprised of~$u$, its incident edge~$e$, and the other vertex~$v$ incident to~$e$. Since $\Lambda$~is a block, it follows that~$\Lambda_0 \cong \one$ and so~$\Lambda = \Lambda'$.
\end{proof}

\begin{prop}[Algorithmic detection of blocks]\label{prop.blockrecognition}
There is an algorithm to detect whether any two blocks with decorations of the same type are isomorphic. 
\end{prop}

\begin{proof}
For $k \in \{1,2\}$, let $\Lambda_k$~be decorated blocks whose isomorphism types we wish to compare. We first consult the underlying graphs~$\langle \Lambda_k\rangle$ to check whether they are both one-edge graphs. Of course if exactly one among the~$\Lambda_k$ is a one-edge graph, then they are not isomorphic. In case both~$\Lambda_k$ are one-edge graphs, Lemma~\ref{lem.one-edge} reduces the problem to testing whether the abstract graphs~$\langle \Lambda_k\rangle$ are isomorphic, which is a straightforward verification.

We now consider the case where none of the~$\Lambda_k$ is a one-edge graph. Construct marked exteriors $(X_k, P_k)$ for the~$\Lambda_k$, and note that they are Haken:
\begin{itemize}
\item The $X_k$~are irreducible by the ``if'' direction in Proposition~\ref{prop.sep<=>red}.
\item The $(X_k, P_k)$~are boundary-irreducible by Proposition~\ref{prop.bred=>cut}.
\item Since the~$X_k$ have non-empty boundary, Proposition~\ref{prop.nontrivboundary} tells us they are either sufficiently large or handlebodies. In the handlebody case, we exclude the genus-$0$ case by the ``only if'' direction in Lemma~\ref{lem.3ballexterior}. Since blocks have edges, the condition $P_k \ne \emptyset$ is certainly satisfied.
\end{itemize}

We are thus allowed to apply the algorithm in Matveev's Recognition Theorem (Theorem~\ref{thm.matveev}) to test whether the $(X_k, P_k)$~are homeomorphic. Propositions \ref{prop.decext_welldef}~and~\ref{prop.decext_uniqueness} tell us this is equivalent to the~$\Lambda_k$ being isomorphic.
\end{proof}

We will need a slightly refined version of this proposition:

\begin{lem}[Algorithmic detection of multi-pointed blocks]\label{lem.multipointedblocs}
There is an algorithm that takes as input:
\begin{itemize}
\item two blocks $\Lambda_1, \Lambda_2$ with decorations of the same type,
\item a tuple $(v_1^1, \ldots, v_1^r)$ of distinct vertices of~$\Lambda_1$, and
\item a same-sized tuple $(v_2^1, \dots, v_2^r)$ of distinct vertices of~$\Lambda_2$,
\end{itemize}
and decides whether there is an isomorphism $\Phi \colon \Lambda_1 \to \Lambda_2$ satisfying, for each $l \in \{1, \ldots , r\}$, $\Phi(v_1^l) = v_2^l$. 
\end{lem}

\begin{proof}
Since having no vertex coloring is the same as having a vertex coloring where all vertices are $0$-colored, we may assume from now on that vertex colorings~$f_1, f_2$ are part of the decoration on~$\Lambda_1, \Lambda_2$, respectively.

Let $n \in \N$ be such that both $f_1$~and~$f_2$ have range contained in~$\{0,\ldots, n-1\}$. For each $k \in \{1,2\}$, let $\Lambda_k^+$~be the same block as~$\Lambda_k$, except that its vertex coloring~$f_k^+$ is now defined as
\[f_k^+(v) = \begin{cases}
 f_k(v) & \text{if $v$ is not one of the~$v_k^l$,}\\
 nl + f_k(v) & \text{if $v = v_k^l$.}
\end{cases}\]
The coloring~$f_k^+$ encodes, for each vertex~$v$, the original coloring~$f_k(v)$ as the mod-$n$ residue. Moreover, the division with remainder of~$f_k^+(v)$ by~$n$ returns~$l$ if $v$~is one of the~$v_k^l$, and otherwise it returns~$0$.

An isomorphism~$\Lambda_1 \to \Lambda_2$ as in the statement of the lemma is therefore the same as an isomorphism $\Lambda_1^+ \to \Lambda_2^+$, whose existence can be algorithmically determined by Proposition~\ref{prop.blockrecognition}.
\end{proof}

With Lemma~\ref{lem.multipointedblocs} at our disposal, we can bootstrap our algorithm for recognition of spatial graphs. We develop it first for pieces (Proposition~\ref{prop.piecerecognition}), and then finally for spatial graphs in full generality (Theorem~\ref{thm.spatialgraphrecognition}). Note that we make no further explicit usage of Matveev's Recognition Theorem.

\begin{lem}[Algorithmic decomposition into a tree of blocks]\label{lem.findtree}
There is an algorithm that, given a non-separable spatial graph~$\Gamma$ other than a one-point graph, produces a tree of blocks~$\Tree$ such that $\Gamma = [\Tree]$.
\end{lem}

Before proving Lemma~\ref{lem.findtree}, we import the following result from Matveev's book \cite[Theorem~4.1.13]{Matveev}.

\begin{thm}[Algorithmic detection of boundary-reducibility]\label{thm.bred}
  There exists an algorithm to decide whether or not a given irreducible 3-manifold with boundary pattern~$(M,P)$ is boundary-irreducible. In case it is boundary reducible, the algorithm constructs a reducing disc.
\end{thm}

\begin{proof}[Proof of Lemma~{\ref{lem.findtree}}]
The proof is essentially the same as the one given when showing that every non-separable spatial graph (other than~$\one$) can be obtained as the realization of a tree of blocks (Proposition~\ref{prop.treeexists}). This time however, when constructing the tree of blocks $\Tree = (T, I, J, L,  (\Lambda_i)_{i \in I}, (v(l))_{l \in L})$,  we must be careful that all steps can be carried out algorithmically.

We induct on the number of edges of~$\Gamma$. The case $\Gamma = \zero$ is trivial and the case where $\Gamma$~is a one-point graph is excluded by assumption. Otherwise, we need to determine whether $\Gamma$~is a block, and in case it is not, we need to find a cut sphere for~$\Gamma$.

To that end, we construct a marked exterior~$(X_\Gamma, P_\Gamma)$ for~$\Gamma$. Since $\Gamma$~is non-separable, Proposition~\ref{prop.sep<=>red} tells us~$X_\Gamma$ is irreducible and so we can use the algorithm of Theorem~\ref{thm.bred} to determine whether~$(X_\Gamma, P_\Gamma)$ is boundary-reducible.

If $(X_\Gamma, P_\Gamma)$~is boundary-irreducible, then it follows from Proposition~\ref{prop.cut=>bred} that $\Gamma$~has no cut vertices, and hence $\Gamma$~is a block. We thus take $T$~to have a single vertex~$i_0 \in I$, and $\Lambda_{i_0}:= \Gamma$.

If on the other hand $(X_\Gamma, P_\Gamma)$~is boundary-reducible, then the algorithm of Theorem~\ref{thm.bred} constructs a reducing disk, which Proposition~\ref{prop.bred=>cut} converts into a cut sphere~$S$ for~$\Gamma$. The two vertex-summands in the induced decomposition~$\Gamma = \Gamma_1 \vsum{v}{v} \Gamma_2$ are then the sub-graphs of~$\Gamma$ supported on each side of~$S$. Each of them has strictly fewer edges than~$\Gamma$, so by induction we can algorithmically produce trees of blocks~$\Tree_1, \Tree_2$ for them, which are then assembled into~$\Tree$ as explained in the proof of Proposition~\ref{prop.treeexists} (with a quick overview given in Figure~\ref{fig.findtree}).
\end{proof}

\begin{prop}[Algorithmic detection of pieces]\label{prop.piecerecognition}
There is an algorithm to detect whether any two pieces with decorations of the same type are isomorphic.
\end{prop}

\begin{proof}
Let $\Gamma_1, \Gamma_2$~be the decorated pieces we wish to compare. If at least one of the~$\Gamma_k$ is a one-point graph, it is trivial to determine whether they are isomorphic. From now on we thus assume none of the~$\Gamma_k$ is a one-point graph.

For each~$k \in \{1,2\}$, use the algorithm from Lemma~\ref{lem.findtree} to decompose $\Gamma_k$ as the realization of a tree of blocks $\Tree_k = (T_k, I_k, J_k, L_k, (\Lambda_i)_{i \in I_k}, (v(l))_{l \in L_k})$. We then list the isomorphisms of abstract trees $f\colon T_1 \to T_2$ satisfying $f(I_1)=I_2$, which is a finite combinatorial problem. By Proposition~\ref{prop.uniquetree}, if no such isomorphism exists, then the $\Gamma_k$ are non-isomorphic.

What is more, that proposition tells us that if the~$\Gamma_k$ are isomorphic, then for some such~$f$, there is a family of isomorphisms $(\Phi_i \colon \Lambda_i \to \Lambda_{f(i)})_{i \in I_1}$, and this family of isomorphisms is compatible with the assignments $l \mapsto v(l)$. For each $f$~in our list, we thus use the algorithm of Lemma~\ref{lem.multipointedblocs} at every $i \in I_1$ to determine whether there is an isomorphism~$\Phi_i \colon \Lambda_i \to \Lambda_{f(i)}$ mapping the vertex tuple $(v(l))_{l}$ to~$(v(f(l))_{l}$, where $l$~ranges over the edges in~$L_1$ incident to~$i$.

If for some~$f\colon T_1 \to T_2$ we find a collection of isomorphisms~$\Phi_i$ as above, then they can be assembled into an isomorphism $\Phi \colon \Gamma_1 \to \Gamma_2$ (see Lemma~\ref{lem.treeofisos}). If this does not happen for any~$f$, we conclude by Proposition~\ref{prop.uniquetree} that the pieces $\Gamma_1, \Gamma_2$ are not isomorphic.
\end{proof}

Extending our algorithm detection result from pieces to general spatial graphs is now comparatively easy. We begin by establishing an algorithmic decomposition as a disjoint union of pieces.

\begin{lem}[Algorithmic decomposition as the disjoint union of pieces]\label{lem.findpieces}
There is an algorithm that, given a decorated spatial graph~$\Gamma$, produces a finite collection of pieces $(\Lambda_i)_{i \in I}$ such that $\Gamma = \bigsqcup_{i \in I} \Lambda_i$.
\end{lem}

Proving this lemma requires importing an algorithm to test for reducibility of $3$-manifolds \cite[pp.~161-162]{Matveev}.

\begin{thm}[Algorithmic detection of reducibility]\label{thm.red}
  There exists an algorithm to decide whether or not a given PL 3-manifold is irreducible. In case it is reducible, the algorithm constructs a reducing sphere.
\end{thm}

\begin{proof}[Proof of Lemma~\ref{lem.findpieces}]
As with the proof of Lemma~\ref{lem.findtree}, this amounts to refining our argument establishing the existence of a decomposition as a finite disjoint union, to show that the construction can be carried out algorithmically.

We begin by constructing a marked exterior~$(X_\Gamma, P_\Gamma)$, and use the algorithm of Theorem~\ref{thm.red} to test whether $X_\Gamma$~is reducible. If $X_\Gamma$~is irreducible, then by Proposition~\ref{prop.sep<=>red} we conclude $\Gamma$~is non-separable. In this case, either $\Gamma = \zero$, in which case we take $I = \emptyset$, or $\Gamma$~is a piece, so it is a disjoint union of pieces indexed by a single element.

On the other hand, if $X_\Gamma$~is reducible, then the algorithm of Theorem~\ref{thm.red} produces a reducing sphere~$S$ for~$X_\Gamma$. What is more, Corollary~\ref{cor.reducingsphereseparates} ensures that $S$~is a separating sphere for~$\Gamma$. Hence, one can write $\Gamma = \Gamma_1 \sqcup \Gamma_2$, where the disjoint union summands are the non-empty graphs supported on each side of~$S$. Since $\Gamma_1, \Gamma_2$ both have strictly fewer vertices than~$\Gamma$, we may assume by induction that $\Gamma_1, \Gamma_2$ are algorithmically decomposable into pieces, and these decompositions assemble into one for~$\Gamma$.
\end{proof}

Finally, we obtain our main result in full generality:

\begin{thm}[Algorithmic detection of spatial graphs]\label{thm.spatialgraphrecognition}
There is an algorithm to detect whether any two spatial graphs with decorations of the same type are isomorphic.
\end{thm}

\begin{proof}
Let $\Gamma_1, \Gamma_2$~be the decorated graphs we wish to compare. For each $k \in \{1,2\}$, we use Lemma~\ref{lem.findpieces}, to decompose~$\Gamma_k$ as a disjoint union of pieces $\bigsqcup_{i \in I_k} \Lambda_i$.
By Proposition~\ref{prop.uniquepieces}, if the indexing sets~$I_k$ have different cardinality, then the~$\Gamma_k$ are not isomorphic. Moreover, the same proposition guarantees that if the $\Gamma_k$~are isomorphic, then there is a bijection $f\colon I_1 \to I_2$ such that for every $i \in I_1$, there is an isomorphism $\Phi \colon \Lambda_i \to \Lambda_{f(i)}$.

We thus run through all possible bijections~$f \colon I_1 \to I_2$, and for each one we use the algorithm of Proposition~\ref{prop.piecerecognition} to test whether every $\Lambda_i$~is isomorphic to~$\Lambda_{f(i)}$. If no~$f$ has such a compatible family of isomorphisms, then the $\Gamma_k$ are non-isomorphic. On the other hand, if for some~$f$ there is a suitable family~$(\Phi_i \colon \Lambda_i \to \Lambda_{f(i)})_{i \in I_1}$, then an iterated application of Lemma~\ref{lem.isodisjunion} allows us to assemble the $\Phi_i$ into an isomorphism $\Phi \colon \Gamma_1 \to \Gamma_2$.
\end{proof}

\printbibliography

\bigskip

S. Friedl, \textsc{Universit\"at Regensburg, Fakult\"at f\"ur Mathematik, 
93053 Regensburg, Deutschland}, \par\nopagebreak
  \textit{E-mail address:} \texttt{stefan.friedl@ur.de}

  \medskip
L. Munser, \textsc{Universit\"at Regensburg, Fakult\"at f\"ur Mathematik, 
93053 Regensburg, Deutschland}, \par\nopagebreak
  \textit{E-mail address:} \texttt{lars.munser@ur.de}
  \medskip
  
J.P. Quintanilha, \textsc{Universit\"at Regensburg, Fakult\"at f\"ur Mathematik, 
  93053 Regensburg, Deutschland}, \par\nopagebreak
    \textit{E-mail address:} \texttt{jose-pedro.quintanilha@ur.de}
      \medskip
      
Y. Santos Rego, \textsc{Otto-von-Guericke-Universit\"at Magdeburg, Fakult\"at f\"ur Mathematik -- Institut f\"ur Algebra und Geometrie, Postfach 4120, 39016 Magdeburg, Deutschland}, \par\nopagebreak
        \textit{E-mail address:} \texttt{yuri.santos@ovgu.de}
\end{document}